\documentclass[letterpaper]{amsart}
\usepackage[margin=1in]{geometry} 
\usepackage[utf8]{inputenc}
\usepackage{amsmath,amsthm,amssymb,amsfonts,mathrsfs}
\usepackage{setspace}
\usepackage{dsfont}
\usepackage{cite}
\usepackage{bookmark}

\numberwithin{equation}{section}

\newtheorem{theorem}{Theorem}[section]
\newtheorem{lemma}[theorem]{Lemma}
\newtheorem{proposition}[theorem]{Proposition}
\newtheorem{corollary}[theorem]{Corollary}

\theoremstyle{definition}
\newtheorem{definition}[theorem]{Definition}

\theoremstyle{remark}
\newtheorem{remark}[theorem]{Remark}

\title[$L^q$ Estimates on the Restriction of Schr\"odinger Eigenfunctions]{$L^q$ Estimates on the Restriction of Schr\"odinger Eigenfunctions with singular potentials}
\author[M. D. Blair]{Matthew D. Blair}
\author[C. Park]{Chamsol Park}
\address{Department of Mathematics and Statistics, University of New Mexico, Albuquerque, NM 87131, USA}
\email{blair@math.unm.edu}
\address{Department of Mathematics, Johns Hopkins University, Baltimore, MD 21218, USA}
\email{cspark@jhu.edu}
\date{}

\usepackage{pgfplots}
\pgfplotsset{compat=1.16}

\begin{document}

\begin{abstract}
We consider eigenfunction estimates in $L^p$ for Schr\"odinger operators, $H_V=-\Delta_g+V(x)$, on compact Riemannian manifolds $(M, g)$. Eigenfunction estimates over the full manifolds were already obtained by Sogge \cite{Sogge1988concerning} for $V\equiv 0$ and the first author, Sire, and Sogge \cite{BlairSireSogge2021Quasimode}, and the first author, Huang, Sire, and Sogge \cite{BlairHuangSireSogge2022UniformSobolev} for critically singular potentials $V$. For the corresponding restriction estimates for submanifolds, the case $V\equiv 0$ was considered in Burq, G\'erard, and Tzvetkov \cite{BurqGerardTzvetkov2007restrictions}, and Hu \cite{Hu2009lp}. In this article, we will handle eigenfunction restriction estimates for some submanifolds $\Sigma$ on compact Riemannian manifolds $(M, g)$ with $n:=\dim M\geq 2$, where $V$ is a singular potential.
\end{abstract}

\maketitle

\section{Introduction}
Let $(M, g)$ be an $n$-dimensional compact Riemannian manifold without boundary. The main purpose of this paper is to find estimates on restrictions of eigenfunctions to submanifolds, associated with the Schr\"odinger operator $H_V$
\begin{align*}
    H_V=-\Delta_g +V(x),
\end{align*}
where $\Delta_g$ is the Laplace-Beltrami operator associated with the metric $g$, and $V$ is a real-valued potential.

We shall focus mostly on critically singular potentials $V(x)$, and so, mostly we shall assume that
\begin{align*}
    V\in L^{\frac{n}{2}} (M).
\end{align*}
We know from \cite[Appendix]{BlairHuangSireSogge2022UniformSobolev} that if $V\in L^{\frac{n}{2}} (M)$, the Schr\"odinger operator $H_V$ is bounded from below and a semi-bounded self-adjoint operator on $L^2$ (see \cite[Proposition A.1]{BlairHuangSireSogge2022UniformSobolev}), and thus, adding a positive constant if needed, we may assume that the spectrum of $H_V$ is positive and its eigenfunctions $e_\lambda^V$'s are satisfying
\begin{align}\label{Hv eigfcn decomp}
    H_V e_\lambda^V=\lambda^2 e_\lambda^V \text{ for some } \lambda>0,\quad \text{i.e.,} \quad \sqrt{H_V} e_\lambda^V=\lambda e_\lambda^V.
\end{align}

The Kato class potential is also known as a critically singular potential with the same scaling properties as the $L^{\frac{n}{2}} (M)$ potentials. We recall that a potential $V$ is said to be in the Kato class, denoted by $V\in \mathcal{K}(M)$, if
\begin{align*}
    \lim_{r \searrow 0} \sup_x \int_{B_r(x)} h_n (d_g (x, y)) |V(y)|\:dy=0,
\end{align*}
where $d_g (\cdot, \cdot)$ denotes geodesic distance, $B_r (x)$ denotes the geodesic ball of radius $r$ centered with $x$, $dy$ denotes the volume element on $(M, g)$, and, for $r>0$,
\begin{align}\label{hn set up}
    h_n (r)=\begin{cases}
    |\log r|, & \text{if } n=2, \\
    r^{2-n}, & \text{if } n\geq 3.
    \end{cases}
\end{align}
By definition, we note that
\begin{align*}
    \mathcal{K}(M)\subset L^1 (M),
\end{align*}
and thus, we will make use of this whenever we need in the calculation with the Kato class potentials.

As explained in the work \cite{BlairSireSogge2021Quasimode} of the first author, Sire, and Sogge, there are some advantages of assuming the Kato condition on the potential $V$. First, it helps us to make use of a heat kernel bound and to obtain quasimode estimates for $p=\infty$. Second, if $V\in \mathcal{K}(M)$, then the $e_\lambda^V$ are continuous on $M$. In addition, the Schr\"odinger operator $H_V$ is self-adjoint when $V$ is a Kato potential or $V\in L^{\frac{n}{2}} (M)$, i.e., we may assume that the spectrum of $H_V$ is positive and its eigenfunctions $e_\lambda^V$'s are satisfying \eqref{Hv eigfcn decomp}.

Thoughout this work, let $\lambda\geq 1$. There are substantial results when $V\equiv 0$. We denote by $\mathds{1}_{[\lambda, \lambda+h(\lambda)]} (\sqrt{H_V})$ the spectral projection operator associated with the operator $\sqrt{H_V}$ where the eigenvalues lie in the interval $[\lambda, \lambda+h(\lambda)]$, for some $h(\lambda)>0$. If $V\equiv 0$, we write $H_V$ as
\begin{align*}
    H_0=-\Delta_g.
\end{align*}
For $h(\lambda)\equiv 1$, Sogge \cite{Sogge1988concerning} showed that, for all $n\geq 2$, there exists a uniform constant $C>0$ such that
\begin{align}\label{Sogge's estimates}
    \|\mathds{1}_{[\lambda, \lambda+1]} (\sqrt{H_0})\|_{L^2(M)\to L^q (M)}\leq C \lambda^{\sigma(q)},\quad \text{when } V\equiv 0,
\end{align}
where
\begin{align}\label{Sogge's exponents}
    \begin{split}
        \sigma(q)=\begin{cases}
        n\left(\frac{1}{2}-\frac{1}{q} \right)-\frac{1}{2}, & \frac{2(n+1)}{n-1}\leq q\leq \infty, \\
        \frac{n-1}{2}\left( \frac{1}{2}-\frac{1}{q} \right), & 2\leq q\leq \frac{2(n+1)}{n-1}.
        \end{cases}
    \end{split}
\end{align}
Consequently, if $e_\lambda^0$ is the eigenfunction of $\sqrt{H_0}$ in that $\sqrt{H_0} e_\lambda^0=\lambda e_\lambda^0$, we have
\begin{align}\label{Sogge's eigfcn estimates}
    \|e_\lambda^0 \|_{L^q (M)}\leq C \lambda^{\sigma(p, n)} \|e_\lambda^0\|_{L^2 (M)}.
\end{align}
If $(M, g)$ has nonpositive sectional curvatures, there are logarithmically improved estimates of the case $h(\lambda)=(\log \lambda)^{-1}$ of the form
\begin{align}\label{Log imp over full manifolds}
    \| \mathds{1}_{[\lambda,\lambda+(\log \lambda)^{-1}]} (\sqrt{H_0}) \|_{L^2(M)\to L^q (M)} \leq C_q \frac{\lambda^{\sigma(q)}}{(\log \lambda)^{\alpha(q, n)}},\quad \text{when } q>2,
\end{align}
for some constant $\alpha(q, n)>0$. It then follows that
\begin{align*}
    \|e_\lambda^0\|_{L^q (M)}\leq C_p \frac{\lambda^{\sigma(q)}}{(\log \lambda)^{\alpha (q, n)} } \|e_\lambda^0\|_{L^2 (M)},\quad \text{when } q>2.
\end{align*}
If $q\geq q_c=\frac{2(n+1)}{n-1}$, then it has been known from \cite{HassellTacy2015improvement} and \cite{BlairHuangSogge2022Improved} that
\begin{align}\label{log lambda exponent from L2(M) to Lq(M)}
    \begin{split}
        \alpha(q, n)=\begin{cases}
            \frac{2}{(n+1)q_c}, & \text{if } n\geq 2 \text{ and } q=q_c=\frac{2(n+1)}{n-1}, \\
            \frac{1}{2}, & \text{if } n\geq 2 \text{ and } q>\frac{2(n+1)}{n-1}.
        \end{cases}
    \end{split}
\end{align}
There are many prior results on these estimates. We refer the reader to B\'erard \cite{Berard1977onthewaveequation}, Hassell and Tacy \cite{HassellTacy2015improvement}, Canzani and Galkowski \cite{CanzaniGalkowski2020Growth}, the first author and Sogge \cite{BlairSogge2017refined, BlairSogge2018concerning, BlairSogge2019logarithmic}, Sogge \cite{Sogge2011KakeyaNikodym}, Sogge and Zelditch \cite{SoggeZelditch2014eigenfunction}, and the first author, Huang, and Sogge \cite{BlairHuangSogge2022Improved}.

On the other hand, Burq, G\'erard, and Tzvetkov \cite{BurqGerardTzvetkov2007restrictions}, and Hu \cite{Hu2009lp} proved the following restriction versions of \eqref{Sogge's estimates}
\begin{align}\label{BGT and Hu's estimates}
    \|\mathds{1}_{[\lambda, \lambda+1]} (\sqrt{H_0})\|_{L^2(M) \to L^q (\Sigma)} \leq C\lambda^{\delta(q, k)} (\log \lambda)^{\nu(q, k)},
\end{align}
where $\Sigma$ is a $k$-dimensional embedded submanifold of $M$, and
\begin{align}\label{BGT and Hu lambda exponent}
    \delta(q, k)=\begin{cases}
    \frac{n-1}{2}-\frac{n-1}{q}, & \text{if } \frac{2n}{n-1}<q\leq \infty, \; k=n-1, \\
    \frac{n-1}{4}-\frac{n-2}{2q}, & \text{if } 2\leq q<\frac{2n}{n-1},\; k=n-1, \\
    \frac{n-1}{2}-\frac{k}{q}, & \text{if } 1\leq k\leq n-2,
    \end{cases}
\end{align}
and
\begin{align}\label{BGT and Hu log lambda exponent}
    \nu(q, k)=\begin{cases}
    \frac{1}{2}, & \text{if } (q, k)=(2, n-2), \\
    0, & \text{otherwise.}
    \end{cases}
\end{align}
This gives us automatically that
\begin{align*}
    \|e_\lambda^0\|_{L^q (\Sigma)}\leq C \lambda^{\delta(q, k)} (\log \lambda)^{\nu (q, k)} \|e_\lambda^0\|_{L^2 (M)},\quad \text{when } V\equiv 0.
\end{align*}
We note that, if $\dim M=2$ and $\gamma$ is any curve, then this estimate is translated into
\begin{align}\label{BGT and Hu eig est for any curve}
    \|e_\lambda^0\|_{L^q (\gamma)}\leq C\lambda^{\delta(q, 1)} \|e_\lambda^0\|_{L^2 (M)}.
\end{align}

If the potential $V$ is non-zero, we may need different arguments to find eigenfunction estimates. Finding analogues of ``uniform Sobolev estimates'' of Kenig, Ruiz, and Sogge \cite{KenigRuizSogge1987UniformSobolev} and Dos Santos Ferreira, Kenig and Salo \cite{DSFKenigSalo2014Forum}, the first author, Huang, Sire, and Sogge \cite[Theorem 1.1]{BlairHuangSireSogge2022UniformSobolev} proved that if $V\in L^{n/2}(M)$, $u\in \mathrm{Dom}(H_V)$, and
\begin{align*}
    2<q\leq \frac{2n}{n-4} \text{ for } n\geq 5, \quad \text{or} \quad 2<q<\infty \text{ for } n=3, 4,
\end{align*}
then
\begin{align}\label{BHSS Thm 1.1 estimate}
    \|u\|_{L^q (M)}\leq C_V \lambda^{\sigma(q)-1}\|(H_V-\lambda^2+i\lambda)u\|_{L^2 (M)},\quad \text{when } \lambda\geq 1.
\end{align}
See also the work of \cite[Theorem 1.3]{BlairSireSogge2021Quasimode} for similar estimates with $V\in L^{\frac{n}{2}} (M) \cap \mathcal{K}(M)$. Our first result is a restriction analogue of the estimate \eqref{BHSS Thm 1.1 estimate} with $n=2$ and $V\in \mathcal{K}(M)$ for curve segments.

\begin{theorem}\label{Thm for univ est for any curves}
Let $\dim M=2$, $V\in \mathcal{K}(M)$, and $\lambda\geq 1$. Suppose the curve $\gamma$ is any curve segment. If $\delta (q, 1)$ is as in \eqref{BGT and Hu lambda exponent}, then
\begin{align}\label{Est for any curves}
    \|u\|_{L^q (\gamma)}\leq C_V \lambda^{\delta(q, 1)-1}\|(H_V-(\lambda+i)^2)u \|_{L^2(M)}, \quad \begin{cases}
        \text{if } 2\leq q<\infty \text{ and } u\in \mathrm{Dom}(H_V), \\
        \text{if } q=\infty \text{ and } u\in \mathrm{Dom}(H_V)\cap C(M),
    \end{cases}
\end{align}
where $C(M)$ denotes the space of continuous functions on $M$.
\end{theorem}
As a consequence, for any eigenfunction as in \eqref{Hv eigfcn decomp} we have
\begin{align}\label{Eigfcn restriction for general curves}
    \begin{split}
        \|e_\lambda^V\|_{L^q (\gamma)}\leq C_V \lambda^{\delta(q, 1)}\|e_\lambda^V\|_{L^2 (M)},\quad \text{when } 2\leq q\leq \infty.
    \end{split}
\end{align}
In what follows, we focus mostly on the estimates of the form \eqref{Est for any curves}, since estimates of the form \eqref{Est for any curves} automatically imply the estimates of the form \eqref{Eigfcn restriction for general curves} and the estimates for quasimodes as well.

Burq, G\'erard, and Tzvetkov \cite{BurqGerardTzvetkov2007restrictions}, and Hu \cite{Hu2009lp} also showed that, if $\dim M=2$, $\gamma$ is a curve with nonvanishing geodesic curvature, and
\begin{align}\label{Exponent for curved curves}
    \Tilde{\delta}(q)=\frac{1}{3}-\frac{1}{3q},\quad 2\leq q\leq 4,
\end{align}
then
\begin{align}\label{BGT and Hu curved curve estimates}
    \|\mathds{1}_{[\lambda, \lambda+1]} (\sqrt{H_V}) \|_{L^2 (M) \to L^q (\gamma)} \leq C \lambda^{\Tilde{\delta}(p)},\quad \text{when } V\equiv 0 \text{ and } 2\leq q\leq 4,
\end{align}
and thus,
\begin{align}\label{BGT and Hu eig est for curved curves}
    \|e_\lambda^0 \|_{L^q (\gamma)}\leq C \lambda^{\Tilde{\delta}(q)}\|e_\lambda^0 \|_{L^2 (M)},\quad \text{when } 2\leq q\leq 4.
\end{align}
Our next result is a restriction analogue of \eqref{BHSS Thm 1.1 estimate} with $n=2$ and $V\in \mathcal{K}(M)$ for curve segments with nonvanishing geodesic curvature.

\begin{theorem}\label{Thm for univ est for curved curves}
Let $\dim M=2$, $V\in \mathcal{K}(M)$, and $\lambda\geq 1$. Suppose $\gamma$ is a curve segment with nonvanishing geodesic curvatures, i.e.,
\begin{align*}
    g(D_t \dot{\gamma}, D_t \dot{\gamma})\not=0.
\end{align*}
If $\Tilde{\delta}(q)$ is as in \eqref{Exponent for curved curves}, then
\begin{align}\label{Est for curved curves}
    \|u\|_{L^q (\gamma)}\leq C_V \lambda^{\Tilde{\delta}(q)-1}\|(H_V-(\lambda+i)^2)u \|_{L^2 (M)},\quad \text{if } u\in \mathrm{Dom}(H_V).
\end{align}
\end{theorem}
As a consequence, for any eigenfunction as in \eqref{Hv eigfcn decomp} we have
\begin{align}\label{Eigfcn restriction for curved curves}
    \|e_\lambda^V \|_{L^q (\gamma)} \leq C_V \lambda^{\Tilde{\delta}(q)} \|e_\lambda^V\|_{L^2 (M)},\quad \text{when } 2\leq q\leq 4.
\end{align}
We note that \eqref{Eigfcn restriction for general curves} and \eqref{Eigfcn restriction for curved curves} are analogous to \eqref{BGT and Hu eig est for any curve} and \eqref{BGT and Hu eig est for curved curves}, respectively. The work of \cite[\S 5]{BurqGerardTzvetkov2007restrictions} showed that the estimates \eqref{Eigfcn restriction for general curves} are sharp when $V\equiv 0$ and $M=\mathbb{S}^2$, in that, there exists a set of eigenfunctions on $\mathbb{S}^2$ such that
\begin{align*}
    \|e_\lambda^0 \|_{L^q (\gamma)}\geq c \lambda^{\delta(q, 1)} \|e_\lambda^0 \|_{L^2 (\mathbb{S}^2)},\quad \text{for some uniform } c>0,\text{ and } 2\leq q\leq \infty.
\end{align*}
The work of \cite[\S 5]{BurqGerardTzvetkov2007restrictions} also showed that the estimates \eqref{Eigfcn restriction for curved curves} are also sharp when $V\equiv 0$, $M=\mathbb{S}^2$, and $\gamma$ is any curve with nonvanishing geodesic curvatures. For constructing examples of exact eigenfunctions and quasimodes, we also refer the reader to Tacy \cite{Tacy2018Constructing}.

The next two results are the restriction analogues of \eqref{BHSS Thm 1.1 estimate} to obtain the analogue of \eqref{BGT and Hu's estimates} when $V\in L^{\frac{n}{2}} (M)$ for hypersurfaces and $(n-2)$-dimensional submanifolds.

\begin{theorem}\label{Thm for hypersurfaces}
    Let $M$ be a compact Riemannian manifold of dimension $n\geq 3$ and $\Sigma$ be a hypersurface of $M$. Suppose $\lambda\geq 1$, $V\in L^{\frac{n}{2}} (M)$, and $u\in \mathrm{Dom}(H_V)$. Let $\Sigma$ be a hypersurface.
    \begin{enumerate}
        \item We have
        \begin{align}\label{Hypersurface estimates}
            \|u\|_{L^q (\Sigma)}\leq C_V \lambda^{\delta(q, n-1)-1}\|(H_V-\lambda^2+i\lambda)u\|_{L^2 (M)},
        \end{align}
        provided that
        \begin{enumerate}
            \item $n\in \{3, 4, 5\}$ and $\frac{2(n-1)^2}{n^2-3n+4}\leq q<\frac{2(n-1)}{n-3}$, or
            \item $n\in \{6, 7\}$ and $\frac{2n^2-5n+4}{n^2-4n+8}\leq q<\frac{2(n-1)}{n-3}$, or
            \item $n\geq 8$ and $\frac{2n^2-5n+4}{n^2-4n+8}<q<\frac{2(n-1)}{n-3}$.
        \end{enumerate}
        \item If either $q=\frac{2(n-1)}{n-3}$ and $n\geq 4$, or $q=\frac{2n^2-5n+4}{n^2-4n+8}$ and $n\geq 8$, then
        \begin{align}\label{Hypersurface log loss estimates}
            \|u\|_{L^q (\Sigma)}\leq C_V \lambda^{\delta(q, n-1)-1} (\log \lambda) \|(H_V-\lambda^2+i\lambda)u\|_{L^2 (M)}.
        \end{align}
    \end{enumerate}
\end{theorem}

We note that even if the Sobolev trace formula holds for $V\in L^{\frac{n}{2}} (M)$, i.e.,
\begin{align*}
    \|f_\lambda^V \|_{L^q (\Sigma)}\leq C_V'\lambda^{\frac{1}{q}} \|f_\lambda^V\|_{L^q (M)},\quad \text{for } f_\lambda^V=\mathds{1}_{[\lambda, \lambda+1]} (\sqrt{H_V}) f_\lambda^V,
\end{align*}
then by this and \cite[Theorem 1.1]{BlairHuangSireSogge2022UniformSobolev}, we have that, for $q\geq \frac{2(n+1)}{n-1}$,
\begin{align}\label{Quasimode computation for hypersurfaces}
    \begin{split}
        \|f_\lambda^V \|_{L^q (\Sigma)}&\leq C_V' \lambda^{\frac{1}{q}} \|f_\lambda^V \|_{L^q (M)} \\
        &\leq C_V \lambda^{\frac{1}{q}} \left(\lambda^{\frac{n-1}{2}-\frac{n}{q}} \|f_\lambda^V\|_{L^2 (M)} \right) \\
        &\leq C_V \lambda^{\frac{n-1}{2}-\frac{n-1}{q}} \|f_\lambda^V\|_{L^2 (M)} \\
        &=C_V \lambda^{\delta(q, n-1)} \|f_\lambda^V \|_{L^2 (M)},\quad q\geq \frac{2(n+1)}{n-1}.
    \end{split}
\end{align}
We note that the condition $q\geq \frac{2(n+1)}{n-1}$ is essential in this computation, since $\sigma(q)=\frac{n-1}{2}-\frac{n}{q}$ when $q\geq \frac{2(n+1)}{n-1}$. This means that even if the Sobolev trace formula holds, we have the quasimode estimates only for $q\geq \frac{2(n+1)}{n-1}$, and thus, the advantage of Theorem \ref{Thm for hypersurfaces} is to consider the quasimode estimates for the $q$'s less than $\frac{2(n+1)}{n-1}$.

\begin{theorem}\label{Thm for codim 2}
    Let $M$ be a compact Riemannian manifold of dimension $n\geq 3$ and $\Sigma$ be a submanifold of dimension $k=n-2$. Suppose $\lambda \geq 1$, $V\in L^{\frac{n}{2}} (M)$, and $u\in \mathrm{Dom}(H_V)$.
    \begin{enumerate}
        \item If $n\geq 3$ and $\max\left(2, \frac{2(n-2)^2}{n^2-5n+8}\right)<q<\frac{2(n-2)}{n-3}$, then
        \begin{align}\label{Codimension 2 estimates}
            \|u\|_{L^q (\Sigma)}\leq C_V \lambda^{\delta(q, n-2)-1} \|(H_V-\lambda^2+i\lambda)u\|_{L^2 (M)}.
        \end{align}
        \item If $(n, q)=(3, 2)$, then
        \begin{align*}
            \|u\|_{L^2 (\Sigma)}\leq C_V \lambda^{\delta(2, 1)-1} (\log \lambda)^{\frac{1}{2}} \|(H_V-\lambda^2+i\lambda)u\|_{L^2 (M)}.
        \end{align*}
        \item If $n\geq 4$ and $q\in\{\frac{2(n-2)^2}{n^2-5n+8}, \frac{2(n-2)}{n-3}\}$, then
        \begin{align*}
            \|u\|_{L^q (\Sigma)}\leq C_V \lambda^{\delta(q, n-2)-1} (\log \lambda)^{\frac{3n-7}{2(n-2)}} \|(H_V-\lambda^2+i\lambda)u\|_{L^2 (M)}.
        \end{align*}
        \item If $n=3$, suppose $\Sigma$ is either a geodesic or a curve with nonvanishing geodesic curvatures. Then
        \begin{align}\label{Geodesic or curved no log loss codim 2 estimates}
            \|u\|_{L^2 (\Sigma)}\leq C_V \lambda^{\delta(2, 1)-1} \|(H_V-\lambda^2+i\lambda)u\|_{L^2 (M)}.
        \end{align}
    \end{enumerate}
\end{theorem}

Since we know that
\begin{align*}
    \begin{cases}
        \frac{2(n-2)^2}{n^2-5n+8}=1, & \text{when } n=3, \\
        \frac{2(n-2)^2}{n^2-5n+8}=2, & \text{when } n=4, \\
        2<\frac{2(n-2)^2}{n^2-5n+8}<\frac{2(n-2)}{n-3}, & \text{when } n\geq 5,
    \end{cases}
\end{align*}
the estimate \eqref{Codimension 2 estimates} holds for $2<q<\frac{2(n-2)}{n-3}$ when $n=3, 4$. If $n=3$, we interpret $\frac{2(n-2)}{n-3}$ as $\infty$.

We also note that as above, even if the Sobolev trace formula holds for $V\in L^{\frac{n}{2}} (M)$, then as in the computation in \eqref{Quasimode computation for hypersurfaces}, we would have that, for the quasimode $f_\lambda^V$,
\begin{align*}
    \|f_\lambda^V\|_{L^q (\Sigma)}\leq C_V\lambda^{\delta(q, n-2)}\|f_\lambda^V\|_{L^2 (M)},\quad \dim \Sigma=n-2,\quad q\geq \frac{2(n+1)}{n-1},
\end{align*}
and again, the advantage of Theorem \ref{Thm for codim 2} is to consider the exponents $q$'s less than $\frac{2(n+1)}{n-1}$ for quasimode estimates.

If we further assume $V\in L^{\frac{n}{2}} (M) \cap \mathcal{K}(M)$, by using \cite[Corollary 1.4]{BlairSireSogge2021Quasimode}, we have eigenfunction estimates in Theorem \ref{Thm for hypersurfaces} and Theorem \ref{Thm for codim 2} with larger ranges of exponents $q$'s.

\begin{corollary}\label{Cor for V in Ln/2 and K(M)}
Let $M$ be a compact Riemannian manifold of dimension $n\geq 3$ and $\Sigma$ be a submanifold of dimension $k\in \{n-1, n-2\}$. If $V\in L^{\frac{n}{2}} (M) \cap \mathcal{K}(M)$ and the $e_\lambda$ are eigenfunctions as in \eqref{Hv eigfcn decomp}, then
\begin{align}\label{Eigfcn restriction est with both potential classes}
    \|e_\lambda^V \|_{L^q (\Sigma)}\leq C_V \lambda^{\delta(q, k)} \|e_\lambda^V \|_{L^2 (M)},
\end{align}
when one of the following holds.
\begin{enumerate}
    \item $k=n-1$, $n\in \{3, 4, 5\}$, and $\frac{2(n-1)^2}{n^2-3n+4}\leq q\leq \infty$.
    \item $k=n-1$, $n\in \{6, 7\}$, and $\frac{2n^2-5n+4}{n^2-4n+8}\leq q\leq \infty$.
    \item $k=n-1$, $n\geq 8$, and $\frac{2n^2-5n+4}{n^2-4n+8}<q\leq \infty$.
    \item $(n, k)=(3, 1)$, and $2<q\leq \infty$.
    \item $k=n-2$, $n\geq 4$, and $\frac{2(n-2)^2}{n^2-5n+8}<q\leq \infty$.
\end{enumerate}
\end{corollary}

\begin{proof}
    In the proof, the constants $C_V$ may be different on different lines, but they are independent on $\lambda$. We recall a basic $L^p$ norm property:
    \begin{align}\label{A basic Lp interpolation property}
        \text{If $0<a<b<c\leq \infty$, $\theta=\frac{b^{-1}-c^{-1}}{a^{-1}-c^{-1}}$, and $f\in L^a \cap L^c$, then $f\in L^b$ and $\|f\|_b \leq \|f\|_a^\theta \|f\|_c^{1-\theta}$.}
    \end{align}
    We note that since $V\in \mathcal{K}(M)$, the $e_\lambda^V$ are continuous on $M$ (see e.g., \cite[Theorem 2.21]{Guneysu2012OnGeneralizedSchrodingerSemigroups} and \cite{Sturm1993SchrodingerSemigroups}), and so, by compactness of $M$, we can freely apply \eqref{A basic Lp interpolation property} to the $e_\lambda^V$. We first show that
    \begin{align}\label{First reduction to Cor.1.5}
        \begin{split}
            & \text{If we have } \|e_\lambda^V\|_{L^{q_0}(\Sigma)}\leq C_V \lambda^{\delta(q_0, n-1)}\|e_\lambda^V \|_{L^2 (M)}\text{ for some } q_0\geq \frac{2n}{n-1},\text{ then} \\
            & \|e_\lambda^V\|_{L^q(\Sigma)}\leq C_V \lambda^{\delta(q, n-1)}\|e_\lambda^V \|_{L^2 (M)} \text{ for all } q_0\leq q\leq \infty.
        \end{split}
    \end{align}
    The assumption in \eqref{First reduction to Cor.1.5} is
    \begin{align}\label{Assumption in the first reduction to Cor.1.5}
        \|e_\lambda^V\|_{L^{q_0}(\Sigma)}\leq C_V \lambda^{\delta(q_0, n-1)}\|e_\lambda^V \|_{L^2 (M)}.
    \end{align}
    We recall that by \cite[Corollary 1.4]{BlairSireSogge2021Quasimode}
    \begin{align}\label{L2 to L infty eigfcn V estimate}
        \|e_\lambda^V\|_{L^\infty (\Sigma)}\leq\|e_\lambda^V \|_{L^\infty (M)}\leq C_V \lambda^{\frac{n-1}{2}} \|e_\lambda^V \|_{L^2 (M)},\quad V\in L^{\frac{n}{2}} (M) \cap \mathcal{K}(M),\quad n\geq 3.
    \end{align}
    Since $\frac{2n}{n-1}\leq q_0\leq q\leq \infty$, if we set $\theta=\frac{q^{-1}-0}{q_0^{-1}-0}=\frac{q_0}{q}$, by \eqref{A basic Lp interpolation property}, \eqref{Assumption in the first reduction to Cor.1.5}, and \eqref{L2 to L infty eigfcn V estimate}, we have
    \begin{align*}
        \|e_\lambda^V\|_{L^q (\Sigma)}&\leq \|e_\lambda^V\|_{L^{q_0}}^\theta \|e_\lambda^V\|_{L^\infty (\Sigma)}^{1-\theta} \\
        &\leq (C_V \lambda^{\delta(q_0, n-1)} \|e_\lambda^V\|_{L^2 (M)})^\theta (C_V \lambda^{\frac{n-1}{2}} \|e_\lambda^V\|_{L^2 (M)})^{1-\theta} \\
        &=C_V \lambda^{\delta(q_0, n-1)\theta+\frac{n-1}{2}(1-\theta)} \|e_\lambda^V\|_{L^2 (M)},
    \end{align*}
    which proves \eqref{First reduction to Cor.1.5}.
    
    We first consider the case where $k=n-1$, $n\geq 3$, and $\frac{2n}{n-1}\leq q\leq \infty$. It follows from Theorem \ref{Thm for hypersurfaces} that
    \begin{align*}
        \|e_\lambda^V \|_{L^{\frac{2n}{n-1}}(\Sigma)}\leq C_V \lambda^{\frac{n-1}{2n}} \|e_\lambda^V \|_{L^2 (M)},\quad V\in L^{\frac{n}{2}} (M), \quad n\in \{3, 4, 5, 6, 7\}.
    \end{align*}
    By this and \eqref{First reduction to Cor.1.5}, we have \eqref{Eigfcn restriction est with both potential classes} where $k=n-1$, $n\in \{3, 4, 5, 6, 7\}$, and $\frac{2n}{n-1}\leq q\leq \infty$, i.e.,
    \begin{align}\label{Eigfcn est for hypsurf with low dim with both potential classes}
        \|e_\lambda^V\|_{L^q (\Sigma)}\leq C_V \lambda^{\delta(q, n-1)} \|e_\lambda^V \|_{L^2 (M)},\quad \text{where } n\in \{3, 4, 5, 6, 7\} \text{ and } \frac{2n}{n-1}\leq q\leq \infty.
    \end{align}

    If $k=n-1, n\in \{3, 4, 5\}$ and $\frac{2(n-1)^2}{n^2-3n+4}\leq q\leq \frac{2n}{n-1}$, then by Theorem \ref{Thm for hypersurfaces}, we have that
    \begin{align*}
        \|e_\lambda^V\|_{L^q (\Sigma)}\leq C_V \lambda^{\delta(q, n-1)}\|e_\lambda^V\|_{L^2 (M)},\quad \frac{2(n-1)^2}{n^2-3n+4}\leq q\leq \frac{2n}{n-1}, \quad V\in L^{\frac{n}{2}}(M), \quad n\in \{3, 4, 5\}.
    \end{align*}
    Combining this with \eqref{Eigfcn est for hypsurf with low dim with both potential classes}, we have \eqref{Eigfcn restriction est with both potential classes} when the first condition holds, i.e., $k=n-1, n\in \{3, 4, 5\}$, and $\frac{2(n-1)^2}{b^2-3n+4}\leq q\leq \infty$. The second case, i.e., $n\in \{6, 7\}$ and $\frac{2n^2-5n+4}{n^2-4n+8}\leq q\leq \frac{2n}{n-1}$, follows similarly. By Theorem \ref{Thm for hypersurfaces},
    \begin{align*}
        \|e_\lambda^V \|_{L^q (\Sigma)}\lesssim \lambda^{\delta(q, n-1)} \|e_\lambda^V \|_{L^2 (M)},\quad V\in L^{\frac{n}{2}} (M),\quad \frac{2n^2-5n+4}{n^2-4n+8}\leq q\leq \frac{2n}{n-1},\quad n=6, 7.
    \end{align*}
    By this and \eqref{Eigfcn est for hypsurf with low dim with both potential classes}, \eqref{Eigfcn restriction est with both potential classes} holds for the case where $k=n-1$, $n\in \{6, 7\}$, and $\frac{2n^2-5n+4}{n^2-4n+8}\leq q\leq \frac{2n}{n-1}$. When $n\geq 8$ for $k=n-1$, the results follow similarly from \eqref{Eigfcn est for hypsurf with low dim with both potential classes} and \eqref{Hypersurface estimates} for $n\geq 8$.

    The codimension $2$ cases follow if we apply Theorem \ref{Thm for codim 2} instead of applying Theorem \ref{Thm for hypersurfaces}. Indeed, if we set
    \begin{align*}
        0<2<q_1<q<\infty,
    \end{align*}
    then the above argument gives us that
    \begin{align*}
        \|e_\lambda^V\|_{L^q (\Sigma)}&\leq \|e_\lambda^V\|_{L^{q_1}(\Sigma)}^{\frac{q_1}{q}} \|e_\lambda^V\|_{L^\infty (\Sigma)}^{1-\frac{q_1}{q}} \\
        &\lesssim \left(\lambda^{\delta(q_1, n-2)} \|e_\lambda^V\|_{L^2 (M)} \right)^{\frac{q_1}{q}} \left(\lambda^{\frac{n-1}{2}} \|e_\lambda^V\|_{L^2 (M)} \right)^{1-\frac{q_1}{q}} \\
        &=\lambda^{\frac{n-1}{2}-\frac{n-2}{q}}\|e_\lambda^V\|_{L^2 (M)}.
    \end{align*}
    If $(n, k)=(3, 1)$, then we choose $2<q_1<q$. If $k=n-2$ and $n\geq 4$, then we choose $\frac{2(n-2)^2}{n^2-5n+8}<q_1<q$. This completes the proof.
\end{proof}

As in \eqref{Log imp over full manifolds}, if we assume nonpositive sectional curves in $(M, g)$, we have logarithmic improved estimates of the form
\begin{align}\label{Log imp restriction estimates}
    \|\mathds{1}_{[\lambda, \lambda+(\log \lambda)^{-1}]} (\sqrt{H_0}) \|_{L^2 (M)\to L^q (\Sigma)}\leq C \frac{\lambda^{\delta(q, k)}}{(\log \lambda)^{\kappa(q, k)} },
\end{align}
for some constant $\kappa(q, k)>0$ with the same $\delta(q, k)$ as in \eqref{BGT and Hu lambda exponent}. For details, we refer the reader to Chen \cite{Chen2015improvement}, the first author and Sogge \cite{BlairSogge2018concerning}, the first author \cite{Blair2018logarithmic}, Xi and Zhang \cite{XiZhang2017improved}, Zhang \cite{Zhang2017improved}, and so on.

Assuming nonpositive curvatures on $M$, the work of \cite[Theorem 1.3]{BlairHuangSireSogge2022UniformSobolev} also proved estimates analogous to \eqref{Log imp over full manifolds} with $V\in L^{n/2} (M)$ (see also \cite[Theorem 5.1]{BlairHuangSireSogge2022UniformSobolev} for $V\in \mathcal{K}(M)$) in that
\begin{align*}
    \|u\|_{L^q (M)}\leq C \lambda^{\sigma(q)-1} (\epsilon (\lambda))^{-1+\alpha (q, n)} \|(H_V-(\lambda+i\epsilon (\lambda))^2)u \|_{L^2 (M)},
\end{align*}
for $\alpha (q, n) >0$ as in \eqref{log lambda exponent from L2(M) to Lq(M)}, where $\epsilon(\lambda)=(\log (2+\lambda))^{-1}$, and
\begin{align*}
    \frac{2(n+1)}{n-1}\leq q\leq \frac{2n}{n-4} \text{ if } n\geq 5,\quad \text{or} \quad \frac{2(n+1)}{n-1}\leq q<\infty \text{ if } n=3, 4.
\end{align*}
As a consequence, this gives an analogue of \eqref{Log imp over full manifolds} for $V\in L^{n/2}(M)$.

Motivated by \cite{BlairHuangSireSogge2022UniformSobolev}, our next result is analogous to \eqref{Log imp restriction estimates} when $2\leq n\leq 4$, and $V\in \mathcal{K}(M)$ or $V\in L^{\frac{n}{2}} (M)$. Suppose the curves $\gamma_i$ denote that, for $i=1, 2, 3$,
\begin{align}\label{gamma i classifications}
    \begin{split}
        & \gamma_1 \text{ is any curve segment,} \\
        & \gamma_2 \text{ is a geodesic segment,} \\
        & \gamma_3 \text{ is a curve segment with nonvanishing geodesic curvatures.}
    \end{split}
\end{align}
Let
\begin{align}\label{n=2 lambda exponent classifications}
    \begin{split}
        \delta(q, \gamma_i)=\begin{cases}
        \delta(q, 1)=\frac{1}{2}-\frac{1}{q}, & \text{if } \gamma_i=\gamma_1 \text{ and } q>4, \\
        \delta(q, 1)=\frac{1}{4}, & \text{if } \gamma_i=\gamma_2 \text{ and } 2\leq q\leq 4, \\
        \Tilde{\delta}(q)=\frac{1}{3}-\frac{1}{3q}, & \text{if } \gamma_i=\gamma_3 \text{ and } 2\leq q\leq 4,
    \end{cases}
    \end{split}
\end{align}
and
\begin{align}\label{n=2 log lambda exponent classifications}
    \begin{split}
        \kappa (q, \gamma_i)=\begin{cases}
        \frac{1}{2}, & \text{if } \gamma_i=\gamma_1 \text{ and } q>4, \\
        \frac{1}{4}, & \text{if } \gamma_i=\gamma_2 \text{ and } 2\leq q\leq 4, \\
        \frac{1}{2}, & \text{if } \gamma_i=\gamma_3 \text{ and } 2\leq q<4, \\
        \frac{1}{8}, & \text{if } \gamma_i=\gamma_3 \text{ and } q=4.
        \end{cases}
    \end{split}
\end{align}

By the work of \cite{Chen2015improvement}, \cite{BlairSogge2018concerning}, \cite{Blair2018logarithmic}, \cite{XiZhang2017improved}, and \cite{Park2023Restriction}, if $n=2$, then \eqref{Log imp restriction estimates} is translated into
\begin{align*}
    \|\mathds{1}_{[\lambda, \lambda+(\log \lambda)^{-1}]} (\sqrt{H_0})\|_{L^2 (M)\to L^q (\gamma_i)} \leq C \frac{\lambda^{\delta (q, \gamma_i)}}{(\log \lambda)^{\kappa(q, \gamma_i)} },\quad \text{when } \dim M=2.
\end{align*}
We also recall that the first author \cite{Blair2018logarithmic} and Zhang \cite{Zhang2017improved} showed $\kappa(q, k)=\frac{1}{2}$ in \eqref{Log imp restriction estimates} when $(n, k)=(3, 1)$, where $M$ has constant negative sectional curvatures and the submanifold is a geodesic segment. For higher dimensional cases, Chen \cite{Chen2015improvement} showed $\kappa (q, k)=\frac{1}{2}$ when $k=n-1$ with $q>\frac{2n}{n-1}$, and $1\leq k\leq n-2$ with $q>2$. We have the following analogues of $n=2$ with $V\in \mathcal{K}(M)$, or $n\in \{3, 4\}$ with $V\in L^{\frac{n}{2}} (M)$.

\begin{theorem}\label{Thm for log improved}
    Suppose $(M, g)$ is an $n$-dimensional compact Riemannian manifold with nonpositive sectional curvatures, $u\in \mathrm{Dom}(H_V)$, $\lambda\geq 1$, and
    \begin{align*}
        \epsilon(\lambda)=(\log (2+\lambda))^{-1}.
    \end{align*}
    \begin{enumerate}
        \item Let $n=2$ and $V\in \mathcal{K}(M)$. For the curves $\gamma_i$ as in \eqref{gamma i classifications}, if the exponents $\delta(q, \gamma_i)$ and $\kappa(p, i)$ are as in \eqref{n=2 lambda exponent classifications} and \eqref{n=2 log lambda exponent classifications}, we have
        \begin{align}\label{Estimate for log imp}
            \begin{split}
                \|u\|_{L^q (\gamma_i)}\leq C_{V, \gamma_i}\lambda^{\delta(q, \gamma_i)-1} (\epsilon(\lambda))^{-1+\kappa(q, \gamma_i)} \|(H_V-(\lambda+i\epsilon(\lambda))^2)u\|_{L^2 (M)},
            \end{split}
        \end{align}
        where if $q=\infty$ in \eqref{n=2 lambda exponent classifications} and \eqref{n=2 log lambda exponent classifications}, we further assume $u\in \mathrm{Dom}(H_V)\cap C(M)$, where as usual, $C(M)$ is the space of continuous functions on $M$.
        \item Let $n=3$ or $n=4$. We assume that $\Sigma$ is a $k$-dimensional submanifold.
        \begin{enumerate}
            \item Suppose
            \begin{align}\label{(n, k, q, V) for n=3 with 4/3}
                V\in L^{\frac{n}{2}}(M),\quad \text{and} \quad \begin{cases}
                    (n, k)=(3, 2),\;\text{ and }\; 3<q<\infty, \text{ or} \\
                    (n, k)=(3, 1),\;\text{ and }\; 2<q<\infty.
                \end{cases}
            \end{align}
            Then we have
            \begin{align}\label{Estimate for log imp submflds for 4/3}
                \|u\|_{L^q (\Sigma)}\leq C_V \lambda^{\delta(q, k)-1} (\epsilon(\lambda))^{-\frac{1}{2}} \|(H_V-(\lambda+i\epsilon(\lambda))^2)u\|_{L^2 (M)}.
            \end{align}
            \item Suppose
            \begin{align}\label{(n, k, q, V) except 4/3}
                V\in L^{\frac{n}{2}}(M),\quad \text{and} \quad\begin{cases}
                    (n, k)=(3, 2),\;\text{ and } 4<q<\infty, \text{ or} \\
                    (n, k)=(3, 1),\;\text{ and } 4<q<\infty, \text{ or} \\
                    (n, k)=(4, 3),\;\text{ and } 3<q<6, \text{ or} \\
                    (n, k)=(4, 2),\;\text{ and } 2<q<4.
                \end{cases}
            \end{align}
            Then we have
            \begin{align}\label{Estimate for log imp submflds}
                \|u\|_{L^q (\Sigma)}\leq C_V \lambda^{\delta(q, k)-1} (\epsilon(\lambda))^{-\frac{1}{2}}\|(H_V-(\lambda+i\epsilon(\lambda))^2)\|_{L^2 (M)}.
            \end{align}
        \end{enumerate}
    \end{enumerate}
\end{theorem}

The estimates so far may be improved when $M$ is a torus.

\begin{theorem}\label{Thm for tori}
    Let $\mathbb{T}^n$ be an $n$-dimensional torus, and $u\in \mathrm{Dom}(H_V)$.
    \begin{enumerate}
        \item If $n=2$, $V\in \mathcal{K}(\mathbb{T}^2)$, $q>4$, $\gamma$ is a segment of any curve, and
        \begin{align}\label{Epsilon setup for curves on 2D tori}
            \epsilon (\lambda)=\lambda^{\frac{4-q}{3q}},\quad q>4,
        \end{align}
        then
        \begin{align}\label{2D torus est for general curve}
            \begin{split}
                \|u\|_{L^q (\Sigma)}\leq \begin{cases}
                    C_V \lambda^{\delta(q, 1)-1}\bigg[(\epsilon (\lambda))^{-\frac{1}{2}} \|(H_V-(\lambda+i\epsilon (\lambda))^2)u\|_{L^2 (\mathbb{T}^2)} \\
                    \hspace{70pt}+\lambda^{\frac{1}{6}}\|(H_V-(\lambda+i\lambda^{-\frac{1}{3}})^2)u\|_{L^2 (\mathbb{T}^2)}\bigg],& \text{if } 4<q<\infty, \\
                    C_V \lambda^{-\frac{1}{3}}\|(H_V-(\lambda+i\lambda^{-\frac{1}{3}})^2)u\|_{L^2 (\mathbb{T}^2)}, & \text{if } q=\infty \text{ and } u\in \mathrm{Dom}(H_V)\cap C(M),
                \end{cases}
            \end{split}
        \end{align}
        where $C(M)$ is the space of continuous functions on $M$.
        \item If $n=2$, $V\in \mathcal{K}(\mathbb{T}^2)$, $\gamma$ is a geodesic segment, and
        \begin{align}\label{Window for tori geodesic cases}
            \epsilon (\lambda)=\lambda^{-\frac{5}{21}}, \quad \text{if }\; 2\leq q<\frac{8}{3},
        \end{align}
        then
        \begin{align}\label{2D torus est for geodesic}
            \begin{split}
                \|u\|_{L^q (\gamma)}&\leq C_V \bigg[\lambda^{-\frac{3}{4}} (\epsilon (\lambda))^{-\frac{3}{4}} \|(H_V-(\lambda+i\epsilon(\lambda))^2)u\|_{L^2 (\mathbb{T}^2)} \\
                &\hspace{100pt}+\lambda^{-\frac{5}{6}} (\epsilon(\lambda))^{-\frac{3}{2}} \|(H_V-(\lambda+i\lambda^{-\frac{1}{3}})^2)u\|_{L^2 (\mathbb{T}^2)} \bigg].
            \end{split}
        \end{align}
        \item Let $n=3$ or $n=4$, and $V\in L^{\frac{n}{2}} (M)$. We assume that $\Sigma$ is a $k$-dimensional submanifold, and the following.
        \begin{align}\label{Window set up for nD tori for submfld}
            \begin{split}
                \epsilon_1 (\lambda)=\lambda^{-\frac{2}{n+1}\left(\frac{n-1}{2}-\frac{2k}{q} \right)}, \quad \epsilon_2 (\lambda)=\begin{cases}
                    \lambda^{-\frac{3}{16}+c_0}, & \text{if } n=3, \\
                    \lambda^{-\frac{1}{5}+c_0}, & \text{if } n=4,
                \end{cases}
            \end{split}
        \end{align}
        where $c_0>0$ is arbitrary.
        \begin{enumerate}
            \item If
            \begin{align}\label{n=3 k q for tori 4/3 potential}
                \begin{cases}
                    (n, k)=(3, 2)\text{ and } 4<q<8, \text{ or} \\
                    (n, k)=(3, 1)\text{ and } 2<q<4,
                \end{cases}
            \end{align}
            then
            \begin{align}\label{Estimate n=3 k q for tori 4/3 potential}
                \begin{split}
                    \|u\|_{L^q (\Sigma)}&\leq C_V \lambda^{\delta(q, k)-1}\bigg((\epsilon_1 (\lambda))^{-\frac{1}{2}} \|(H_V-(\lambda+i\epsilon_1 (\lambda))^2)u\|_{L^2 (\mathbb{T}^3)} \\
                    &\hspace{100pt}+(\epsilon_2 (\lambda))^{-\frac{3}{4}} \|(H_V-(\lambda+i\epsilon_2 (\lambda))^2)u\|_{L^2 (\mathbb{T}^3)} \bigg).
                \end{split}
            \end{align}
            \item If
            \begin{align}\label{n=4 k q for tori L2 potential}
                \begin{cases}
                    (n, k)=(4, 3)\text{ and } 4<q<6, \text{ or} \\
                    (n, k)=(4, 2)\text{ and } \frac{8}{3}<q<4,
                \end{cases}
            \end{align}
            then
            \begin{align}\label{Estimate n=4 k q for tori L2 potential}
                \|u\|_{L^q (\Sigma)}&\leq C_V \bigg(\lambda^{\delta(q, k)-1}(\epsilon_1 (\lambda))^{-\frac{1}{2}}\|(H_V-(\lambda+i\epsilon_1 (\lambda))^2)u\|_{L^2 (\mathbb{T}^4)}+\lambda^{2-\frac{2k}{q}} \|u\|_{L^2 (\mathbb{T}^4)} \bigg).
            \end{align}
        \end{enumerate}
    \end{enumerate}
\end{theorem}

Note that \eqref{2D torus est for general curve} is not sharp by many existing results. For example, Burq, G\'erard, and Tzvetkov \cite[Introduction]{BurqGerardTzvetkov2007restrictions} showed that, for any $0<\epsilon\ll 1$ and a curve segment $\gamma\in \mathbb{T}^2$, if $V\equiv 0$, then
\begin{align*}
    & \|e_\lambda^0 \|_{L^2(\gamma)}\leq C\lambda^\epsilon \|e_\lambda^0\|_{L^2 (\mathbb{T}^2)}, \\
    & \|e_\lambda^0\|_{L^\infty (\gamma)}\leq C\lambda^{\epsilon} \|e_\lambda^0\|_{L^2 (\mathbb{T}^2)}.
\end{align*}
By interpolation, we then have
\begin{align*}
    \|e_\lambda^0 \|_{L^q (\gamma)}\leq C\lambda^\epsilon \|e_\lambda^0\|_{L^2 (\mathbb{T}^2)},\quad 2\leq q\leq \infty,
\end{align*}
and this is much better than the following result from \eqref{2D torus est for general curve}.
\begin{align*}
    \|e_\lambda^V \|_{L^q (\gamma)}\leq C_V \lambda^{\frac{1}{3}-\frac{1}{3q}} \|e_\lambda^V \|_{L^2 (\mathbb{T}^2)},\quad \text{if } q>4.
\end{align*}
We note that \eqref{2D torus est for geodesic} is also far from being sharp by many existing results. For the spectral projection estimates, estimates in Lemma \ref{Lemma 2D tori geodesic spectral proj bound} are better than \eqref{2D torus est for geodesic}. For the estimates for exact eigenfunctions, Huang and Zhang \cite{HuangZhang2021RestrictionOfToral} showed that there exists an eigenfunction $e_\lambda^0$ for $V\equiv 0$ such that
\begin{align*}
    c\sqrt{N_{\lambda, 1}} \|e_\lambda^0\|_{L^2 (\mathbb{T}^2)}\leq\|e_\lambda^0 \|_{L^2 (\gamma)}\leq C \sqrt{N_{\lambda, 1}} \|e_\lambda^0\|_{L^2 (\mathbb{T}^2)},\quad C, c>0,
\end{align*}
for a geodesic segment $\gamma\subset \mathbb{T}^2$, where $N_{\lambda, 1}$ is a number-theoretic constant, which is known to be
\begin{align*}
    0\leq N_{\lambda, 1}\leq C\log \lambda.
\end{align*}
To the best of our knowledge, it is conjectured that the desired bound is $N_{\lambda, 1}\leq C$ (cf. \cite[Introduction]{HuangZhang2021RestrictionOfToral}, etc.). This is better than the following result from \eqref{2D torus est for geodesic}
\begin{align*}
    \|e_\lambda^V \|_{L^2 (\gamma)}\leq C_V \lambda^{\frac{4}{21}} \|e_\lambda^V\|_{L^2 (\mathbb{T}^2)},\quad \text{if }\; 2\leq q<\frac{8}{3}.
\end{align*}

\subsection*{Outline of the work} In \S \ref{S: Review for fold singularities}, we review the notion of submersions with folds, or fold singularities, in Greenleaf and Seeger \cite{GreenleafSeeger1994fourier} and Hu \cite{Hu2009lp}, since the oscillatory integral estimates related with fold singularities are used throughout this paper. In \S \ref{S: Preliminaries}, we will reduce Theorem \ref{Thm for univ est for any curves}-\ref{Thm for codim 2} to Proposition \ref{Prop Reduction for curves Kato case}-\ref{Prop Reduction for codim 2 thm} by using the perturbation arguments in \cite{BlairSireSogge2021Quasimode} and \cite{BlairHuangSireSogge2022UniformSobolev}. A resolvent formula in Bourgain-Shao-Sogge-Yao \cite{BourgainShaoSoggeYao2015Resolvent} will play an important role in the computation. We will prove Proposition \ref{Prop Reduction for curves Kato case}-\ref{Prop Reduction for codim 2 thm} in \S \ref{S: Proof for curves for universal estimates}-\S \ref{S: proof of codim 2 thm}, completing the proofs of Theorem \ref{Thm for univ est for any curves}-\ref{Thm for codim 2}. For Theorem \ref{Thm for univ est for any curves}-\ref{Thm for univ est for curved curves}, we shall use the perturbation arguments of \cite{BlairSireSogge2021Quasimode} and \cite{BlairHuangSireSogge2022UniformSobolev}. We will also make some scaling argument, which is a reminiscent of the work of Sogge \cite{Sogge1988concerning}, Huang-Sogge \cite{HuangSogge2014Concerning}, Bourgain-Shao-Sogge-Yao \cite{BourgainShaoSoggeYao2015Resolvent}, and so on. We need interpolation computation at the end to finish each proof of Theorem \ref{Thm for hypersurfaces}-\ref{Thm for codim 2}.

In \S \ref{S: proof for log imp} and \S \ref{S: Thm for tori}, we will prove Theorem \ref{Thm for log improved} and Theorem \ref{Thm for tori}, respectively. As in the other theorems, the main idea here is to consider a resolvent operator as in \cite{BourgainShaoSoggeYao2015Resolvent} first, and the perturbation arguments as in \cite{BlairHuangSireSogge2022UniformSobolev} next.

In \S \ref{S: Partial result and future work}, we shall briefly talk about partial results and related future work.

\subsection*{Notation}
If we consider an integral operator $K$, then we denote it as its kernel $K(x, y)$ in that
\begin{align*}
    Kf(x)=\int K(x, y)f(y)\:dy.
\end{align*}
The constants $C$ are uniform constants with respect to $\lambda$, may depend on manifolds $M$, curve $\gamma$, and exponent $p$, and may be different at different lines, but each of the constants are different up to some uniform constant. We write $A\lesssim B$ when there is a uniform constant $C>0$ such that $A\leq CB$. We write $A\approx B$, if $A\lesssim B$ and $B\lesssim A$. We also write $A\ll B$ or $B\gg A$, if $CA\leq B$ for some sufficiently large $C>0$.

\subsection*{Acknowledgements}
We are grateful to Christopher Sogge for helpful and numerous comments and suggestions throughout the course of this work, which greatly improved the early version of this work. The second author is also grateful to Xiaoqi Huang, Andreas Seeger, Yannick Sire, and Cheng Zhang for helpful discussions for this work. The second author is also grateful to Suresh Eswarathasan and Blake Keeler for helpful comments and suggestions and for their hospitality during his visit to the Dalhousie University.

\section{Review of Submersions with Folds}\label{S: Review for fold singularities}
In this section, we briefly review the parts of the work of \cite[\S 2]{GreenleafSeeger1994fourier} and \cite[\S 4]{Hu2009lp}, since we will make use of the arguments in the papers frequently in the rest of this paper.

Let $M$ and $N$ be smooth manifolds and $F: M\to N$ be a smooth map. If $\phi:\mathbb{R}\to N$ is a $C^\infty$ map with
\begin{align*}
    \phi(0)=y\in N,\quad \phi'(0)=\eta\in \ker F'(y).
\end{align*}
As in the computation in \cite[Appendix C.4]{Hormander2007LPDE3}, one can consider an invariantly defined quadratic form
\begin{align*}
    \ker F'(y)\ni \eta \mapsto \langle F''(y)\eta, \eta \rangle\in \mathrm{coker}F'(y).
\end{align*}
This is called the Hessian of $F$. With this in mind, we first recall the definition of a submersion with folds (for details, see \cite[Chapter 3]{Golubitsky-Guillemin1973StableMappings} \cite[p.36]{GreenleafSeeger1994fourier}, \cite[Appendix C.4]{Hormander2007LPDE3}, and so on).

\begin{definition}
    Let $M$ and $N$ be smooth manifolds of dimensions $m$ and $n$, respectively, with $m\geq n$. Then a $C^\infty$ map $F: M\to N$ is a submersion with folds at $x_0\in M$ if the following hold.
    \begin{enumerate}
        \item $\mathrm{rank} F'(x_0)=n-1$ (and thus, $\dim \ker F'(x_0)=m-n+1$ and $\dim \mathrm{coker} F'(x_0)=1$), and
        \item The Hessian of $F$ at $x_0$ is nondegenerate.
    \end{enumerate}
    If $m=n$, a submersion with folds is a Whitney fold.
\end{definition}
As noted in \cite[p.36]{GreenleafSeeger1994fourier}, the variety $\mathcal{L}$ where $F'$ is degenerate is a smooth submanifold in $M$ of codimension $m-n+1$.

We now suppose $U$ and $V$ are open sets in $\mathbb{R}^d$ and $\mathbb{R}^{d+r}$, respectively. We define the oscillatory integral operators $T_\lambda$ by
\begin{align*}
    T_\lambda f (x)=\int e^{i\lambda \Phi(x, y)} a(x, y) f(y)\:dy,
\end{align*}
where $\Phi\in C^\infty (U\times V)$ and $a\in C_0^\infty (U\times V)$. If we consider the canonical relation $\mathcal{C}_\Phi$ associated with the phase function $\Phi$ of the form
\begin{align*}
    \mathcal{C}_\Phi=\{(x, \Phi_x' (x, y); y, -\Phi_y' (x, y))\},
\end{align*}
we define the left projection $\pi_L$ and right projection $\pi_R$ as follows.
\begin{align*}
    & \pi_L: \mathcal{C}_\Phi \to T^* (U),\quad \pi_L (x, y)=(x, \Phi_x'(x, y)), \\
    & \pi_R: \mathcal{C}_\Phi \to T^* (V),\quad \pi_R (x, y)=(y, -\Phi_y'(x, y)).
\end{align*}

Moreover, if $\dim U=\dim V$, the variety $\mathcal{L}$ for the left projection $\pi_L$ is the submanifold of codimension $1$, i.e., the hypersurface. If $\pi^U:\mathcal{L}\to X$ is a submersion, then for each $x$ the projection of $\mathcal{L}$ onto the fiber, denoted by
\begin{align}\label{Sigma at x def in GS}
    H_x=\pi^{T_x^* U} (\mathcal{L}),
\end{align}
is a hypersurface in $T_x^* U$. In \cite{GreenleafSeeger1994fourier}, Greenleaf and Seeger showed the following.

\begin{theorem}[Theorem 2.1 in \cite{GreenleafSeeger1994fourier}]\label{Thm 2.1 in GS}
    Suppose $\dim U=d$, $\dim V=d+r$, and that the left projection $\pi_L$ is a submersion with folds. Let $\lambda\geq 2$.
    \begin{enumerate}
        \item If $r=0$, then
        \begin{align*}
            \|T_\lambda f\|_{L^q (U)} \lesssim \begin{cases}
                \lambda^{-\frac{d-1}{q}-\frac{1}{4}} \|f\|_{L^2 (V)}, & \text{if } 2\leq q\leq 4, \\
                \lambda^{-\frac{d}{q}} \|f\|_{L^2 (V)}, & \text{if } 4\leq q\leq \infty.
            \end{cases}
        \end{align*}
        \item If $r=1$, then
        \begin{align*}
            \|T_\lambda f\|_{L^q (U)}\lesssim \begin{cases}
                \lambda^{-\frac{d}{2}} (\log \lambda)^{\frac{1}{2}} \|f\|_{L^2 (V)}, & \text{if } q=2, \\
                \lambda^{-\frac{d}{q}}\|f\|_{L^2 (V)}, & \text{if } 2<q\leq \infty.
            \end{cases}
        \end{align*}
        \item If $r\geq 2$, then
        \begin{align*}
            \|T_\lambda f\|_{L^q (U)}\lesssim \lambda^{-\frac{d}{q}} \|f\|_{L^2 (V)},\quad \text{if } 2\leq q\leq \infty.
        \end{align*}
    \end{enumerate}
\end{theorem}

\begin{theorem}[Theorem 2.2 in \cite{GreenleafSeeger1994fourier}]\label{Thm for hypersurfaces in Greenleaf-Seeger}
    Suppose $\dim U=\dim V=d$ and that the left projection $\pi_L$ is either nondegenerate or a Whitney fold. Suppose in addition that for each $x\in U$, for each $\zeta\in H_x$ at least $l$ principal curvatures do not vanish, where $H_x$ is as in \eqref{Sigma at x def in GS}. Then for $\lambda \geq 1$
    \begin{align*}
        \|T_\lambda f\|_{L^q (U)}\lesssim \begin{cases}
            \lambda^{-\frac{d-1}{q}-\frac{l+1}{4}+\frac{l}{2q}} \|f\|_{L^2 (V)}, & \text{if } 2\leq q\leq \frac{2l+4}{l+1}, \\
            \lambda^{-\frac{d}{q}} \|f\|_{L^2 (V)}, & \text{if } \frac{2l+4}{l+1}\leq q\leq \infty.
        \end{cases}
    \end{align*}
\end{theorem}

We now let $(M, g)$ be an $n$-dimensional compact Riemannian manifold without boundary. Let $\Sigma$ be a $k$-dimensional submanifold of $M$. We also let $d_g (x, y)$ denote the Riemannian distance between $x$ and $y$. In the geodesic normal coordinates centered at $x_0\in M$, $\Sigma$ can be parametrized by
\begin{align*}
    x(u_1, u_2, \cdots, u_k),\quad \text{and } x(0)=0.
\end{align*}
Using a partition of unity, we may assume that $\Sigma$ is contained in a coordinate patch $U$ so that $|x(u)|\leq c\epsilon$ and $x(0)=0$, and $c_1 \epsilon\leq |y|\leq c_2 \epsilon$. If we use the polar coordinates for $y$, we can write
\begin{align*}
    y=r\omega,\quad c_1\epsilon\leq r\leq c_2\epsilon,\quad \omega\in \mathbb{S}^{n-1}.
\end{align*}
In this setting, if we set
\begin{align}\label{Phase fcn in Hu}
    \Psi (x, \omega)=-d_g (x, r\omega),
\end{align}
then Hu showed the following.

\begin{theorem}[\S 4 in \cite{Hu2009lp}]\label{Thm Section 4 in Hu}
    Let $\Psi$ be the phase function defined as in \eqref{Phase fcn in Hu}.
    \begin{enumerate}
        \item If $\dim \Sigma=k\leq n-2$, then the left projection $\pi_L$ associated with the phase function $\Psi$ has at most fold singularities, i.e., $\Psi$ is a submersion with folds at most, satisfying $\dim \Sigma=d$, where $d$ is as in the hypothesis of Theorem \ref{Thm 2.1 in GS}.
        \item If $\dim \Sigma=k=n-1$, then the left projection $\pi_L$ associated with the phase function $\Psi$ is either nondegerate or a Whitney fold. Also, for each $x$ in a coordinate patch $U$ containing $\Sigma$, for each $\zeta\in H_x$, at least $n-2$ principal curvatures do not vanish.
    \end{enumerate}
\end{theorem}

In other words, Hu showed that the phase function $\Psi$ in \eqref{Phase fcn in Hu} satisfies the hypotheses in \cite[Theorem 2.1-2.2]{GreenleafSeeger1994fourier}, and this is how Hu showed \eqref{BGT and Hu's estimates}.

\section{Preliminary Reductions for Theorem \ref{Thm for univ est for any curves}-\ref{Thm for codim 2}}\label{S: Preliminaries}

Let $P=\sqrt{-\Delta_g}$. By \cite[(2.3)]{BourgainShaoSoggeYao2015Resolvent} and \cite[\S 3-5]{BlairHuangSireSogge2022UniformSobolev}, we can write
\begin{align*}
    (-\Delta_g-(\lambda+i))^{-1} =\frac{i}{\lambda+i}\int_0^\infty e^{i\lambda t} e^{-t} (\cos tP)\:dt.
\end{align*}
Let $\mu_0\in C_0^\infty (\mathbb{R})$ be such that, for $0<\epsilon_0 \ll 1$,
\begin{align*}
    \mu_0 (t)=\begin{cases}
        1, & \text{if } |t|\leq \frac{\epsilon_0}{2}, \\
        0, & \text{if } |t|\geq \epsilon_0.
    \end{cases}
\end{align*}
We then write
\begin{align*}
    (-\Delta_g-(\lambda+i))^{-1}=S_\lambda+W_\lambda,
\end{align*}
where
\begin{align}\label{S lambda W lambda set up}
    \begin{split}
        & S_\lambda=\frac{i}{\lambda+i}\int_0^\infty \mu_0 (t) e^{i\lambda t} e^{-t} (\cos tP)\:dt, \\
        & W_\lambda=\frac{i}{\lambda+i}\int_0^\infty (1-\mu_0 (t)) e^{i\lambda t} e^{-t} (\cos tP)\:dt.
    \end{split}
\end{align}
We first note that we can obtain the estimates of $(-\Delta_g-(\lambda+i)^2)^{-1}$ from the estimates \eqref{BGT and Hu's estimates}.

\begin{lemma}\label{Lemma Resolvent opr norm from L2 to Lq for univ est}
    Let $\Sigma$ be a $k$-dimensional submanifold of $M$. Suppose that
    \begin{align}\label{Epsilon window conditions}
        \begin{split}
            & \lambda^{-1}\lesssim \epsilon(\lambda)\lesssim 1, \\
            & \epsilon(\lambda) \text{ is nonincreasing}, \quad \text{and}\\
            & \epsilon(4\lambda)\lesssim \epsilon(\lambda).
        \end{split}
    \end{align}
    We also assume that, for $\lambda \geq 1$
    \begin{align}\label{Spectral projection window est}
        \|\mathds{1}_{[\lambda, \lambda+\epsilon(\lambda)]} (P)\|_{L^2 (M)\to L^q (\Sigma)}\lesssim \lambda^{\delta(q, k)}(\log \lambda)^{\nu(q, k)} (\epsilon(\lambda))^{\rho(q, k)},\quad \text{for some } 0<\rho(q, k)\leq 1.
    \end{align}
    If $\delta(q, k)<\frac{3}{2}$, then for $\lambda \gg 1$
    \begin{align}\label{Resolvent opr norm from L2 to Lq for univ est}
        \|(-\Delta_g-(\lambda+i\epsilon(\lambda))^2)^{-1}\|_{L^2 (M)\to L^q (\Sigma)}\lesssim \lambda^{\delta(q, k)-1} (\log \lambda)^{\nu (q, k)} (\epsilon(\lambda))^{\rho(q, k)-1},
    \end{align}
    where $\delta(q, k)$ are $\nu (q, k)$ are as in \eqref{BGT and Hu lambda exponent} and \eqref{BGT and Hu log lambda exponent}, respectively.
\end{lemma}

It is natural to assume \eqref{Epsilon window conditions}, since in this paper we want to consider the $\epsilon(\lambda)$ satisfying either $\epsilon(\lambda)=(\log (2+\lambda))^{-1}$, or $\epsilon(\lambda)=\lambda^{-\alpha}$ for some $0\leq \alpha\leq 1$ (in particular, $\epsilon(\lambda)=1$ when $\alpha=0$). See also \cite[\S 2]{BlairHuangSireSogge2022UniformSobolev} or \cite[Introduction]{HuangSoggeTaylor2023ProductManifolds} for details explaining that assuming \eqref{Epsilon window conditions} is reasonable.

\begin{proof}[Proof of Lemma \ref{Lemma Resolvent opr norm from L2 to Lq for univ est}]
    To prove this lemma, we split the operator norm into three pieces:
    \begin{align}\label{High freq opr norm}
        \begin{split}
            \|(-\Delta_g-(\lambda+i\epsilon(\lambda))^2)^{-1} \mathds{1}_{[2\lambda, \infty)} (P)\|_{L^2 (M)\to L^q (\Sigma)}\lesssim \lambda^{\delta(q, k)-1}(\log \lambda)^{\nu (q, k)}(\epsilon(\lambda))^{\rho(q, k)-1},\\
            \text{if } 2\leq q\leq \infty \text{ and } \delta(q, k)<\frac{3}{2},
        \end{split}
    \end{align}
    \begin{align}\label{Low freq opr norm}
        \|(-\Delta_g-(\lambda+i\epsilon(\lambda))^2)^{-1} \mathds{1}_{\left[0, \frac{\lambda}{2}\right]}(P)\|_{L^2 (M)\to L^q (\Sigma))}\lesssim \lambda^{\delta(q, 1)-1}(\log \lambda)^{\nu (q, k)}(\epsilon(\lambda))^{\rho(q, k)-1},\quad \text{if } 2\leq q\leq \infty.
    \end{align}
    and
    \begin{align}\label{Intermediate freq opr norm}
        \begin{split}
            \|(-\Delta_g-(\lambda+i\epsilon(\lambda))^2)^{-1} \mathds{1}_{\left[\frac{\lambda}{2}, 2\lambda\right]}(P)\|_{L^2 (M)\to L^q (\Sigma))}\lesssim \lambda^{\delta(q, 1)-1}(\log \lambda)^{\nu (q, k)}(\epsilon(\lambda))^{\rho(q, k)-1},\quad \text{if } 2\leq q\leq \infty.
        \end{split}
    \end{align}
    
    We first prove \eqref{High freq opr norm}. By the Sobolev trace formula, if we set $s=\frac{n}{2}-\frac{k}{q}$, then
    \begin{align*}
        \|(-\Delta_g-(\lambda+i\epsilon(\lambda))^2)^{-1} \mathds{1}_{[2\lambda, \infty)} (P) f\|_{L^q (\Sigma)}&\lesssim \|(-\Delta_g)^{\frac{s}{2}} (-\Delta_g-(\lambda+i\epsilon(\lambda))^2)^{-1} \mathds{1}_{[2\lambda, \infty)}(P) f\|_{L^2 (M)} \\
        &\lesssim \left(\sup_{\tau\geq 2\lambda} |\tau^s(\tau^2-(\lambda+i\epsilon(\lambda))^2)^{-1}| \right)\|f\|_{L^2 (M)} \\
        &\lesssim \left(\sup_{\tau\geq 2\lambda} \tau^{s-2} \right)\|f\|_{L^2 (M)} \\
        &\lesssim \lambda^{s-2}\|f\|_{L^2 (M)}=\lambda^{\frac{n}{2}-\frac{k}{q}-2}\|f\|_{L^2 (M)},
    \end{align*}
    provided that $s-2<0$ if $2\leq q\leq \infty$. We used the assumption that $\epsilon(\lambda)\lesssim 1$ in the third inequality. Since $s-2\leq \delta(q, k)-\frac{3}{2}$, we have $s-2<0$ when $\delta(q, k)<\frac{3}{2}$, and this is where we need the assumption $\delta(q, k)<\frac{3}{2}$. The estimate \eqref{High freq opr norm} then follows, since $\lambda^{\frac{n}{2}-\frac{k}{q}-2}\leq \lambda^{\delta(q, 1)-1}(\epsilon(\lambda))^{-\frac{1}{2}}$ when $2\leq q\leq \infty$ and $\delta(q, k)<\frac{3}{2}$.

    We next show \eqref{Low freq opr norm}. As above, it follows from \eqref{BGT and Hu's estimates} that
    \begin{align*}
        \|(-\Delta_g-(\lambda+i\epsilon(\lambda))^2)^{-1}\mathds{1}_{\left[0, \frac{\lambda}{2} \right]} (P) f\|_{L^q (\Sigma)}&\leq \sum_{1\leq j\leq \frac{\lambda}{2}} \|(-\Delta_g-(\lambda+i\epsilon(\lambda))^2)^{-1} \mathds{1}_{[j-1, j)} (P) f\|_{L^q (\Sigma)} \\
        &\lesssim \sum_{1\leq j\leq \frac{\lambda}{2}} \sup_{\tau\in [j-1, j)}|\tau^2-(\lambda+i\epsilon(\lambda))^2|^{-1} j^{\delta(q, k)} (\log j)^{\nu (q, k)} \\
        &\lesssim \lambda^{-2} (\log \lambda)^{\nu (q, k)}\sum_{1\leq j\leq \frac{\lambda}{2}} j^{\delta(q, k)} \\
        &\lesssim \lambda^{-2} (\log \lambda)^{\nu (q, k)}\lambda^{\delta(q, k)+1} \\
        &=\lambda^{\delta(q, k)-1} (\log \lambda)^{\nu (q, k)} \\
        &\lesssim \lambda^{\delta(q, k)-1} (\log \lambda)^{\nu (q, k)} (\epsilon(\lambda))^{\rho(q, k)-1},
    \end{align*}
    which proves \eqref{Low freq opr norm}. In the last inequality, we used the fact that $\lambda^{-1}\leq \epsilon(\lambda)\lesssim 1$ and $0<\rho(q, k)\leq 1$.
        
    We are left to prove \eqref{Intermediate freq opr norm}. To see this, note that if
    \begin{align*}
        \frac{\lambda}{4}\leq \epsilon(\lambda) j\leq 4\lambda \quad\text{and}\quad \tau\in [\epsilon(\lambda)j, \epsilon(\lambda)(j+1)],\quad \text{for } j\in \mathbb{N},
    \end{align*}
    then
    \begin{align*}
        |\tau^2-(\lambda+i\epsilon(\lambda))^2|^{-1}&=|\tau-\lambda+i\epsilon(\lambda)|^{-1} |\tau+\lambda+i\epsilon(\lambda)|^{-1} \\
        &\lesssim \lambda^{-1} (\epsilon(\lambda)+|\epsilon(\lambda)j-\lambda|)^{-1}.
    \end{align*}
    Moreover, we know $\epsilon(4\lambda)\lesssim \epsilon(\lambda)$ by \eqref{Epsilon window conditions}, and thus, it follows from \eqref{Spectral projection window est} that
    \begin{align*}
        & \|(-\Delta_g-(\lambda+i\epsilon(\lambda))^2)^{-1} \mathds{1}_{\left[\frac{\lambda}{2}, 2\lambda \right]} (P)f\|_{L^q (\Sigma)} \\
        &\lesssim \sum_{\frac{\lambda}{4}\leq \epsilon(\lambda)j\leq 4\lambda} \lambda^{-1} (\epsilon(\lambda)+|\epsilon(\lambda)j-\lambda|)^{-1} \left(\lambda^{\delta(q, k)} (\log \lambda)^{\nu(q, k)} (\epsilon(\lambda))^{\rho(q, k)} \|\mathds{1}_{[\epsilon(\lambda)j, \epsilon(\lambda)(j+1))} (P) f\|_{L^2 (M)} \right) \\
        &\lesssim\lambda^{\delta(q, k)-1}(\epsilon(\lambda))^{\rho(q, k)-1} (\log \lambda)^{\nu (q, k)} \sum_{\frac{\lambda}{4}\leq \epsilon(\lambda)j\leq 4\lambda} (1+|j-\epsilon(\lambda)^{-1}\lambda|)^{-1}\|\mathds{1}_{[\epsilon(\lambda)j, \epsilon(\lambda)(j+1))}(P) f\|_{L^2 (M)}.
    \end{align*}
    By this, if we set $f_j=\mathds{1}_{[\epsilon(\lambda)j, \epsilon(\lambda)(j+1))} (P)f$ for convenience, then by the Cauchy-Schwarz inequality and orthogonality we have that
    \begin{align*}
        &\|(-\Delta_g-(\lambda+i\epsilon(\lambda))^2)^{-1} \mathds{1}_{\left[\frac{\lambda}{2}, 2\lambda \right]} (P)f\|_{L^q (\Sigma)}\\
        &\lesssim \lambda^{\delta(q, k)-1}(\epsilon(\lambda))^{\rho(q, k)-1} (\log \lambda)^{\nu (q, k)} \left(\sum_{j=1}^\infty (1+|j-\epsilon(\lambda)^{-1}\lambda|)^{-2} \right)^{\frac{1}{2}}\left(\sum_{j=1}^\infty \|f_j\|_{L^2 (M)}^2 \right)^{\frac{1}{2}} \\
        &\lesssim \lambda^{\delta(q, k)-1}(\log \lambda)^{\nu (q, k)}(\epsilon(\lambda))^{\rho(q, k)-1} \|f\|_{L^2 (M)},
    \end{align*}
    which proves \eqref{Intermediate freq opr norm}, completing the proof of this lemma.
\end{proof}

We note that the assumption $\delta(q, k)<\frac{3}{2}$ in Lemma \ref{Lemma Resolvent opr norm from L2 to Lq for univ est} holds in the statements of Theorem \ref{Thm for hypersurfaces}-\ref{Thm for codim 2}. In Theorem \ref{Thm for univ est for any curves}-\ref{Thm for univ est for curved curves}, we know $\delta(q, 1), \Tilde{\delta}(q)<\frac{3}{2}$ automatically, and thus, analogous estimates for $(-\Delta_g-(\lambda+i)^2)^{-1}$ hold.

We next consider the operator $W_\lambda$. If we consider the map, as in \cite{BlairHuangSireSogge2022UniformSobolev},
\begin{align*}
    \tau \mapsto m_\lambda (\tau)=\frac{i}{\lambda+i}\int_0^\infty (1-\mu_0 (t)) e^{i\lambda t} e^{-t} (\cos t\tau)\:dt,
\end{align*}
we have
\begin{align}\label{m lambda size est for n=2}
    |m_\lambda (\tau)|\lesssim \lambda^{-1}(1+|\lambda-\tau|)^{-N},\quad \text{if } \tau\geq 0, \lambda\geq 1, N=1, 2, 3, \cdots.
\end{align}
By this, \eqref{BGT and Hu's estimates}, and an orthogonality argument (as in the proof of Lemma \ref{Lemma Resolvent opr norm from L2 to Lq for univ est}), one can see that the operator $W_\lambda=m_\lambda (P)$ satisfies
\begin{align}\label{W lambda Lq to L2 estimate}
    \|W_\lambda \|_{L^2 (M)\to L^q (\Sigma)}\lesssim \lambda^{\delta(q, k)-1} (\log \lambda)^{\nu(q, k)},
\end{align}
and so,
\begin{align}\label{W lambda composition estimate}
    \begin{split}
        \|W_\lambda \circ (-\Delta_g-(\lambda+i)^2)\|_{L^2 (M)\to L^q (\Sigma)}&\lesssim \lambda^{\delta(q, k)-1} (\log \lambda)^{\nu (q, k)} \lambda \|u\|_{L^2 (M)} \\
        &\lesssim \lambda^{\delta(q, k)-1} (\log \lambda)^{\nu (q, k)} \|(H_V-(\lambda+i)^2 ) u\|_{L^2 (M)},
    \end{split}
\end{align}
where we used the spectral theorem in the last inequality (see also \cite{BlairHuangSireSogge2022UniformSobolev}). Since $S_\lambda=(-\Delta_g-(\lambda+i)^2)^{-1}-W_\lambda$, it follows from \eqref{Resolvent opr norm from L2 to Lq for univ est} and \eqref{W lambda Lq to L2 estimate} that
\begin{align}\label{S lambda L2 to Lq estimate}
    \|S_\lambda\|_{L^2 (M)\to L^q (\Sigma)}\lesssim \lambda^{\delta(q, k)-1} (\log \lambda)^{\nu (q, k)}.
\end{align}
We note that
\begin{align*}
    u&=(-\Delta_g-(\lambda+i)^2)^{-1}\circ (-\Delta_g-(\lambda+i)^2)u \\
    &=(S_\lambda+W_\lambda)\circ (-\Delta_g-(\lambda+i)^2)u \\
    &=S_\lambda(-\Delta_g+V-(\lambda+i)^2)u+W_\lambda (-\Delta_g-(\lambda+i)^2)u-S_\lambda(Vu).
\end{align*}
Using this, \eqref{W lambda composition estimate} and \eqref{S lambda L2 to Lq estimate}, we have
\begin{align}\label{S(Vu) appears}
    \|u\|_{L^q (\Sigma)}\lesssim \lambda^{\delta(q, k)-1} (\log \lambda)^{\nu (q, k)} \|(H_V-(\lambda+i)^2)u\|_{L^2 (M)}+\|S_\lambda (Vu)\|_{L^q (\Sigma)},
\end{align}
and thus, we shall focus on the term $\|S_\lambda (Vu)\|_{L^q (\Sigma)}$ for Theorem \ref{Thm for univ est for any curves}-\ref{Thm for codim 2}.

\subsection{Reductions for Theorem \ref{Thm for univ est for any curves}-\ref{Thm for univ est for curved curves}}

If $(n, k)=(2, 1)$, then $\nu (q, k)=0$, and so, Theorem \ref{Thm for univ est for any curves} follows from the following proposition.
\begin{proposition}\label{Prop Reduction for curves Kato case}
    Suppose $V\in \mathcal{K}(M)$, $\dim M=2$, $u\in \mathrm{Dom}(H_V)$, and $\gamma$ is any curve in $M$. Then
    \begin{align}\label{S(Vu) claim for curves of universal estimates}
        \begin{split}
            \|S_\lambda (Vu)\|_{L^q (\gamma)}\leq \begin{cases}
                C_V \lambda^{-\frac{1}{2}} \|(H_V-(\lambda+i)^2)u\|_{L^2 (M)}, & \text{if } q=\infty, \\
                C_V \lambda^{-\frac{1}{2}-\frac{1}{q}} \|(H_V-(\lambda+i)^2)u\|_{L^2 (M)}, & \text{if } 2<q<\infty, \\
                C_V \lambda^{-1} (\log \lambda)^{\frac{1}{2}} \|(H_V-(\lambda+i)^2)u\|_{L^2 (M)}, & \text{if } q=2.
            \end{cases}
        \end{split}
    \end{align}
\end{proposition}
In fact, we can also show that Proposition \ref{Prop Reduction for curves Kato case} also implies Theorem \ref{Thm for univ est for curved curves}. Indeed, let $\gamma$ be a curve with nonvanishing geodesic curvatures in a Riemannian surface $M$ as in Theorem \ref{Thm for univ est for curved curves}. Instead of using \eqref{BGT and Hu's estimates}, if we use \eqref{BGT and Hu curved curve estimates}, then similar arguments as above give us that, for $2\leq q\leq 4$,
\begin{align*}
    & \|W_\lambda (-\Delta_g-(\lambda+i)^2)u\|_{L^q (\gamma)}\lesssim \lambda^{\Tilde{\delta}(q)},\quad 2\leq q\leq 4, \\
    & \|S_\lambda \|_{L^2 (M)\to L^q (\gamma)}\lesssim \lambda^{\Tilde{\delta}(q)-1},\quad 2\leq q\leq 4,
\end{align*}
and thus, instead of having \eqref{S(Vu) appears}, we have
\begin{align}\label{Curved curves WTS}
    \|u\|_{L^q (\gamma)}\lesssim \lambda^{\Tilde{\delta}(q)-1}\|(H_V-(\lambda+i)^2)u\|_{L^2 (M)}+\|S_\lambda (Vu)\|_{L^q (\gamma)},\quad 2\leq q\leq 4,
\end{align}
and hence, it is enough to control $\|S_\lambda (Vu)\|_{L^q (\gamma)}$. If Proposition \ref{Prop Reduction for curves Kato case} holds, then it should also holds for curves with nonvanishing geodesic curvatures. Since
\begin{align*}
    \lambda^{-\frac{1}{2}-\frac{1}{q}}\leq \lambda^{\Tilde{\delta}(q)-1}, \text{ for } 2\leq q\leq 4,\quad\text{and } \lambda^{-1}(\log \lambda)^{\frac{1}{2}}\leq \lambda^{\Tilde{\delta}(2)-1},
\end{align*}
we have
\begin{align*}
    \|S_\lambda (Vu)\|_{L^q (\gamma)}\leq C_V \lambda^{\Tilde{\delta}(q)-1} \|(H_V-(\lambda+i)^2)u\|_{L^2 (M)},\quad 2\leq q\leq 4,
\end{align*}
which proves Theorem \ref{Thm for univ est for curved curves} by \eqref{Curved curves WTS}. Thus, we would have Theorem \ref{Thm for univ est for any curves}-\ref{Thm for univ est for curved curves}, if we could prove Proposition \ref{Prop Reduction for curves Kato case}. We shall prove Proposition \ref{Prop Reduction for curves Kato case} later in \S \ref{S: Proof for curves for universal estimates}.

\subsection{Reduction for Theorem \ref{Thm for hypersurfaces}}\label{SS: Thm for hypersurfaces}
If $k=n-1$, then $\nu (q, k)=0$, and thus, by \eqref{S(Vu) appears}, we want to control the perturbation term $\|S_\lambda (Vu)\|_{L^q (\Sigma)}$ by using the following propositions.

\begin{proposition}\label{Prop Reduction for hypersurface thm}
    Suppose $\Sigma$ is a hypersurface of $M$, where $\dim M=n$. If $n\geq 3$ and $\frac{n}{p}-\frac{n-1}{q}=2$, then
    \begin{align}\label{S(Vu) WTS hypersurface}
        \begin{split}
            \|S_\lambda f\|_{L^q (\Sigma)}\lesssim \begin{cases}
                \|f\|_{L^p (M)}, & \text{if } 3\leq n\leq 5,\; \text{ and } \frac{2(n-1)^2}{n^2-3n+4}<q<\frac{2(n-1)}{n-3}, \\
                \|f\|_{L^p (M)}, & \text{if } n\geq 6,\; \text{ and } \frac{2n^2-5n+4}{n^2-4n+8}<q<\frac{2(n-1)}{n-3}, \\
                (\log \lambda)\|f\|_{L^p (M)}, & \text{if } 3\leq n\leq 5,\; \text{ and } q=\frac{2(n-1)^2}{n^2-3n+4}, \\
                (\log \lambda)\|f\|_{L^p (M)}, & \text{if } 4\leq n\leq 5,\; \text{ and } q=\frac{2(n-1)}{n-3}, \\
                (\log \lambda)\|f\|_{L^p (M)}, & \text{if } n\geq 6,\; \text{ and } q\in \left\{\frac{2n^2-5n+4}{n^2-4n+8}, \frac{2(n-1)}{n-3} \right\},
            \end{cases}
        \end{split}
    \end{align}
\end{proposition}

In this subsection, we show that Proposition \ref{Prop Reduction for hypersurface thm} implies \eqref{Hypersurface estimates} when $\Sigma$ is a hypersurface. We note that
\begin{enumerate}
    \item $\frac{2(n-1)^2}{n^2-3n+4}<\frac{2n}{n-1}$ if $n\in \{3, 4, 5\}$,
    \item $\frac{2n^2-5n+4}{n^2-4n+8}<\frac{2n}{n-1}$ if $ n\in \{6, 7\}$,
    \item $\frac{2n}{n-1}<\frac{2n^2-5n+4}{n^2-4n+8}$ if $n\geq 8$.
\end{enumerate}
With this in mind, we first consider either
\begin{align}\label{Supercritical condition}
    \frac{2n}{n-1}\leq q<\frac{2(n-1)}{n-3} \text{ and } n\in \{3, 4, 5, 6, 7\}, \quad \text{or}\quad \frac{2n^2-5n+4}{n^2-4n+8}<q<\frac{2(n-1)}{n-3} \text{ and } n\geq 8.
\end{align}
By H\"older's inequality and Proposition \ref{Prop Reduction for hypersurface thm}, if $\frac{n}{p}-\frac{n-1}{q}=2$, we have
\begin{align}\label{S(Vu) Holder's inequality, hypersurface}
    \|S_\lambda (Vu)\|_{L^q (\Sigma)}\lesssim \|Vu\|_{L^p (M)}\leq \|V\|_{L^{\frac{n}{2}}(M)}\|u\|_{L^{\frac{np}{n-2p}}(M)}.
\end{align}
By a direct computation, we have that
\begin{align*}
    \frac{np}{n-2p}>\frac{2(n+1)}{n-1}, \quad \text{if} \quad \frac{n}{p}-\frac{n-1}{q}=2 \text{ and } q\geq \frac{2n}{n-1},
\end{align*}
and thus, for $\frac{2n}{n-1}\leq q<\frac{2(n-1)}{n-3}$, by \cite[Theorem 1.1]{BlairHuangSireSogge2022UniformSobolev} and \eqref{S(Vu) Holder's inequality, hypersurface}, for $V\in L^{\frac{n}{2}} (M)$, $\sigma(q)$ as in \eqref{Sogge's exponents} and $u\in \mathrm{Dom}(H_V)$,
\begin{align}\label{S(vu) est computation hypersurface}
    \begin{split}
        \|S_\lambda (Vu)\|_{L^q (\Sigma)}& \lesssim \|V\|_{L^\frac{n}{2}(M)} \|u\|_{L^{\frac{np}{n-2p}} (M)} \\
        &\leq C_V \lambda^{\sigma\left(\frac{np}{n-2p}\right)-1} \|(H_V-\lambda^2+i\lambda)u\|_{L^2 (M)} \\
        &\lesssim C_V \lambda^{\left(\frac{n-1}{2}-\left(\frac{n}{p}-2 \right) \right)-1}\|(H_V-(\lambda+i)^2)u\|_{L^2 (M)} \\
        &=C_V \lambda^{\frac{n-1}{2}-\frac{n-1}{q}-1}\|(H_V-(\lambda+i)^2)u\|_{L^2 (M)} \\
        &=C_V \lambda^{\delta(q, n-1)-1} \|(H_V-(\lambda+i)^2)u\|_{L^2 (M)}.
    \end{split}
\end{align}
In the third inequality, we used the triangle inequality and the spectral theorem to obtain $\|(H_V-\lambda^2+i\lambda)u\|_{L^2 (M)}\lesssim \|(H_V-(\lambda+i)^2)u\|_{L^2 (M)}$. By \eqref{S(Vu) appears} and \eqref{S(vu) est computation hypersurface}, we would have \eqref{Hypersurface estimates} for $(q, n)$ as in \eqref{Supercritical condition}, if we could prove Proposition \ref{Prop Reduction for hypersurface thm}.

We need to consider the remaining cases where
\begin{align}\label{Subcritical condition for n=6, 7}
    \frac{2n^2-5n+4}{n^2-4n+8} \leq q<\frac{2n}{n-1} \quad \text{and} \quad n\in \{6, 7\},
\end{align}
\begin{align}\label{Subcritical condition for n=3, 4, 5}
    \frac{2(n-1)^2}{n^2-3n+4} \leq q<\frac{2n}{n-1} \quad \text{and} \quad n\in \{3, 4, 5\},
\end{align}
and
\begin{align}\label{Upper endpt condition}
    q=\frac{2(n-1)}{n-3} \text{ and } n\geq 4, \quad \text{or} \quad q=\frac{2n^2-5n+4}{n^2-4n+8} \text{ and } q\geq 8.
\end{align}
which are not in \eqref{Supercritical condition}.

We first suppose \eqref{Subcritical condition for n=6, 7}. By direct computations, one can see that
\begin{align}\label{Exponent comparisons for subcritical exponents}
    \begin{split}
        \begin{cases}
            \sigma\left(\frac{np}{n-2p}\right)\leq \delta(q, n-1), & \text{when } n\in \{6, 7\} \text{ and } \frac{2n^2-5n+4}{n^2-4n+8}<q\leq \frac{2n}{n-1}, \\
            \sigma\left(\frac{np}{n-2p}\right)<\delta(q, n-1), & \text{when } n\in \{6, 7\} \text{ and } q=\frac{2n^2-5n+4}{n^2-4n+8}.
        \end{cases}
    \end{split}
\end{align}
Since \eqref{Exponent comparisons for subcritical exponents} follows from routine calculations, we skip the calculations here and leave the details to the reader. By \eqref{Exponent comparisons for subcritical exponents}, if we assume $\frac{2n^2-5n+4}{n^2-4n+8}<q\leq \frac{2n}{n-1}$ for $n\in \{6, 7\}$, then by Proposition \ref{Prop Reduction for hypersurface thm}, \eqref{S(Vu) Holder's inequality, hypersurface}, and \cite[Theorem 1.1]{BlairHuangSireSogge2022UniformSobolev} again,
\begin{align}\label{Perturbation for hypsurf for subcritical exp}
    \begin{split}
        \|S_\lambda (Vu) \|_{L^q (\Sigma)}&\lesssim \|V\|_{L^{\frac{n}{2}}(M)} \|u\|_{L^{\frac{np}{n-2p}}(M)} \\
        &\leq C_V \lambda^{\sigma\left(\frac{np}{n-2p} \right)-1} \|(H_V-\lambda^2+i\lambda)u\|_{L^2 (M)} \\
        &\leq C_V \lambda^{\delta(q, n-1)-1} \|(H_V-\lambda^2+i\lambda)u\|_{L^2 (M)}.
    \end{split}
\end{align}
If $n\in \{6, 7\}$ and $q=\frac{2n^2-5n+4}{n^2-4n+8}$, then since $\sigma\left(\frac{np}{n-2p} \right)<\delta(q, n-1)$ by \eqref{Exponent comparisons for subcritical exponents}, we have, by Proposition \ref{Prop Reduction for hypersurface thm}, \eqref{S(Vu) Holder's inequality, hypersurface}, and \cite[Theorem 1.1]{BlairHuangSireSogge2022UniformSobolev} again,
\begin{align}\label{Perturbation for hypsurf for low endpt}
    \begin{split}
        \|S_\lambda (Vu) \|_{L^q (\Sigma)}&\lesssim (\log \lambda) \|V\|_{L^{\frac{n}{2}}(M)} \|u\|_{L^{\frac{np}{n-2p}}(M)} \\
        &\leq C_V \lambda^{\sigma\left(\frac{np}{n-2p} \right)-1} (\log \lambda) \|(H_V-\lambda^2+i\lambda)u\|_{L^2 (M)} \\
        &\leq C_V \lambda^{\delta(q, n-1)-1} \|(H_V-\lambda^2+i\lambda)u\|_{L^2 (M)}.
    \end{split}
\end{align}
Thus, the estimate \eqref{Hypersurface estimates} is satisfied when \eqref{Subcritical condition for n=6, 7} holds.

We next assume \eqref{Subcritical condition for n=3, 4, 5}. If $n\in \{3, 4, 5\}$ and $\frac{n}{p}-\frac{n-1}{q}=2$, then
\begin{align*}
    \frac{np}{n-2p}=\frac{nq}{n-1}<\frac{2(n+1)}{n-1},\quad \text{when } q=\frac{2(n-1)^2}{n^2-3n+4}.
\end{align*}
With this in mind, by straightforward computations, one can obtain that
\begin{align*}
    \begin{cases}
        \frac{n-1}{2}\left(\frac{1}{2}-\frac{n-1}{nq} \right)<\frac{n-1}{4}-\frac{n-2}{2q}, & \text{if } n\in \{3, 4, 5\},\; \frac{nq}{n-1}\leq \frac{2(n+1)}{n-1}, \text{ and } \frac{2(n-1)^2}{n^2-3n+4}<q\leq \frac{2n}{n-1}, \\
        n\left(\frac{1}{2}-\frac{n-1}{2q} \right)-\frac{1}{2}\leq \frac{n-1}{4}-\frac{n-2}{2q}, & \text{if } n\in \{3, 4, 5\},\; \frac{2(n+1)}{n-1}\leq \frac{nq}{n-1}, \text{ and } \frac{2(n-1)^2}{n^2-3n+4}<q\leq \frac{2n}{n-1}, \\
        \frac{n-1}{2}\left(\frac{1}{2}-\frac{n-1}{2q} \right)<\frac{n-1}{4}-\frac{n-2}{2q}, & \text{if } n\in \{3, 4, 5\}, \text{ and } q=\frac{2(n-1)^2}{n^2-3n+4}.
    \end{cases}
\end{align*}
This gives us that
\begin{align*}
    \begin{split}
        \begin{cases}
            \sigma\left(\frac{np}{n-2p}\right)\leq \delta(q, n-1), & \text{when } n\in \{3, 4, 5\} \text{ and } \frac{2(n-1)^2}{n^2-3n+4}<q\leq \frac{2n}{n-1}, \\
            \sigma\left(\frac{np}{n-2p}\right)<\delta(q, n-1), & \text{when } n\in \{3, 4, 5\} \text{ and } q=\frac{2(n-1)^2}{n^2-3n+4},
        \end{cases}
    \end{split}
\end{align*}
and thus, by this, Proposition \ref{Prop Reduction for hypersurface thm}, \eqref{S(Vu) Holder's inequality, hypersurface}, and \cite[Theorem 1.1]{BlairHuangSireSogge2022UniformSobolev} again, as in the computation in \eqref{Perturbation for hypsurf for subcritical exp}-\eqref{Perturbation for hypsurf for low endpt}, we have that
\begin{align*}
    \|S_\lambda (Vu)\|_{L^q (\Sigma)}\leq C_V \lambda^{\delta(q, n-1)} \|(H_V-\lambda^2+i\lambda)u\|_{L^2 (M)}.
\end{align*}

The remaining cases for Theorem \ref{Thm for hypersurfaces} are $q=\frac{2(n-1)}{n-3}$ and $n\geq 4$, or $q=\frac{2n^2-5n+4}{n^2-4n+8}$ and $q\geq 8$, i.e., \eqref{Upper endpt condition}. For \eqref{Upper endpt condition}, all the computations are the same as in \eqref{S(vu) est computation hypersurface} except a log loss from Proposition \ref{Prop Reduction for hypersurface thm}, and thus,
\begin{align*}
    \|S_\lambda (Vu)\|_{L^{q} (\Sigma)}\leq C_V \lambda^{\delta(q, n-1)-1} (\log \lambda) \|(H_V-(\lambda+i)^2)u\|_{L^2 (M)}.
\end{align*}
Putting these altogether with \eqref{S(Vu) appears} yields Theorem \ref{Thm for hypersurfaces}. We shall show Proposition \ref{Prop Reduction for hypersurface thm} in \S \ref{S: Proof for any hypersurfaces}.

\subsection{Reduction for Theorem \ref{Thm for codim 2}}
Considering \eqref{S(Vu) appears} again, Theorem \ref{Thm for codim 2} follows from the following proposition.

\begin{proposition}\label{Prop Reduction for codim 2 thm}
    Suppose $\Sigma$ is an $(n-2)$-dimensional submanifold of $M$, where $\dim M=n\geq 3$.
    \begin{enumerate}
        \item Let $n=3$. If
        \begin{align*}
            \frac{3}{p}-\frac{1}{q}=2,\quad 2\leq q<\infty,
        \end{align*}
        then
        \begin{align*}
            \begin{split}
                \|S_\lambda f\|_{L^q (\Sigma)}\lesssim \|f\|_{L^p (M)}.
            \end{split}
        \end{align*}
        \item Let $n\geq 4$. If
        \begin{align*}
            \frac{n}{p}-\frac{n-2}{q}=2,\quad \frac{2(n-2)^2}{n^2-5n+8}\leq q\leq \frac{2(n-2)}{n-3},
        \end{align*}
        then
        \begin{align*}
            \begin{split}
                \|S_\lambda f\|_{L^q (\Sigma)}\lesssim \begin{cases}
                    \|f\|_{L^p (M)}, & \text{if } \frac{2(n-2)^2}{n^2-5n+8}<q<\frac{2(n-2)}{n-3}, \\
                    (\log \lambda)^{\frac{3n-7}{2(n-2)}} \|f\|_{L^p (M)}, & \text{if } q=\frac{2(n-2)^2}{n^2-5n+8}, \text{ or } q=\frac{2(n-2)}{n-3}.
                \end{cases}
            \end{split}
        \end{align*}
    \end{enumerate}
\end{proposition}

If Proposition \ref{Prop Reduction for codim 2 thm} is true, for any codimension $2$ submanifold $\Sigma$, by using H\"older's inequality, for $V\in L^{\frac{n}{2}} (M)$, $\frac{n}{p}-\frac{n-2}{q}=2$, and the arguments in \eqref{S(vu) est computation hypersurface}, we have
\begin{align}\label{S(Vu) computation for codim 2 cases}
    \begin{split}
        &\|S_\lambda (Vu)\|_{L^q (\Sigma)} \\
        &\leq \begin{cases}
            C_V \lambda^{\delta(q, n-2)-1} \|(H_V-(\lambda+i)^2) u\|_{L^2 (M)}, & \text{if } n=3 \text{ and } 2\leq q<\infty, \\
             & \text{or } n\geq 4 \text{ and } \frac{2(n-2)^2}{n^2-5n+8}<q<\frac{2(n-2)}{n-3}, \\
            C_V \lambda^{\delta(q, n-2)-1} (\log \lambda)^{\frac{3n-7}{2(n-2)}} \|(H_V-(\lambda+i)^2)u\|_{L^2 (M)}, & \text{if } n\geq 4, \text{ and } q=\frac{2(n-2)}{n^2-5n+8} \text{ or } q=\frac{2(n-2)}{n-3}.
        \end{cases}
    \end{split}
\end{align}
Using \eqref{S(Vu) computation for codim 2 cases} and \eqref{S(Vu) appears}, we have Theorem \ref{Thm for codim 2} for general codimension $2$ submanifolds.

We are left to consider \eqref{Geodesic or curved no log loss codim 2 estimates} when $\Sigma$ is either a geodesic segment or a curve segment with nonvanishing geodesic curvatures for $n=3$. By Chen and Sogge \cite[Theorem 1.1]{ChenSogge2014few} and Wang and Zhang \cite[Theorem 3]{WangZhang2021Codim2}, we have, instead of \eqref{BGT and Hu's estimates},
\begin{align*}
    \|\mathds{1}_{[\lambda, \lambda+1]} (P) f\|_{L^2 (\Sigma)}\lesssim \lambda^{\frac{1}{2}} \|f\|_{L^2 (M)},
\end{align*}
and thus, $\nu (2, 1)=0$ when $n=3$. Since there is no log loss for $(n, k, q)=(3, 1, 2)$ in \eqref{S(Vu) computation for codim 2 cases}, by the above arguments (cf. \eqref{S(Vu) computation for codim 2 cases}), we have \eqref{Geodesic or curved no log loss codim 2 estimates}, which completes the proof.

We shall prove Proposition \ref{Prop Reduction for codim 2 thm} later in \S \ref{S: proof of codim 2 thm}.

\section{Proof of Theorem \ref{Thm for univ est for any curves}-\ref{Thm for univ est for curved curves}}\label{S: Proof for curves for universal estimates}
As we discussed in \S \ref{S: Preliminaries}, we shall prove Proposition \ref{Prop Reduction for curves Kato case} here to finish the proofs of Theorem \ref{Thm for univ est for any curves}-\ref{Thm for univ est for curved curves}. By the argument in \cite[\S 5]{BlairHuangSireSogge2022UniformSobolev}, we have the kernel estimates of $S_\lambda$
\begin{align}\label{S lambda kernel est for n=2}
    \begin{split}
    |S_\lambda (x, y)|\lesssim \begin{cases}
    |\log (\lambda d_g (x, y)/2)|, & \text{if } d_g (x, y)\leq \lambda^{-1}, \\
    \lambda^{-\frac{1}{2}} (d_g (x, y))^{-\frac{1}{2}}, & \text{if } \lambda^{-1}\leq d_g (x, y)\leq 1.
    \end{cases}
    \end{split}
\end{align}
We first consider $q=\infty$. We note that
\begin{align*}
    \sup_r\left|\int S_\lambda (\gamma(r), y) V(y) u(y)\:dy \right|\lesssim \left(\sup_r \int |S_\lambda (\gamma(r), y)| |V(y)|\:dy \right) \|u\|_{L^\infty (M)}.
\end{align*}
By \eqref{S lambda kernel est for n=2}, we know that, for $0<\epsilon \ll 1$,
\begin{align*}
    |S_\lambda (\gamma(r), y)|\lesssim h_2 (d_g (\gamma(r), y)) \mathds{1}_{d_g (\gamma(r), y)<\epsilon} (\gamma(r), y).
\end{align*}
Since $V\in \mathcal{K}(M)$ (i.e., $V\in \mathcal{K}(M)\cap L^1 (M)=\mathcal{K}(M)\cap L^{\frac{n}{2}} (M)$ for $n=2$, since $\mathcal{K}(M)\subset L^1 (M)$), we then have that, for $\lambda\geq 1$ large enough,
\begin{align*}
    \sup_r \int_{B_\epsilon (\gamma(r))} h_2 (d_g (\gamma(r), y)) |V(y)|\:dy \ll 1,
\end{align*}
by taking a sufficiently small $\epsilon>0$. By \cite[Theorem 1.3]{BlairSireSogge2021Quasimode}, we also know that
\begin{align*}
    \|u\|_{L^\infty (M)}\leq C_V \lambda^{-\frac{1}{2}} \|(-\Delta_g+V-(\lambda+i)^2)u\|_{L^2 (M)}.
\end{align*}
Combining these together, we have that
\begin{align*}
    \|S_\lambda (Vu) \|_{L^\infty (\gamma)}\leq \sup_r \left|\int S_\lambda (\gamma(r), y)V(y)u(y)\:dy \right|\leq C_V \lambda^{-\frac{1}{2}} \|(-\Delta_g+V-(\lambda+i)^2)u\|_{L^2 (M)},
\end{align*}
and this proves \eqref{S(Vu) claim for curves of universal estimates} for $q=\infty$.

We next consider $2<q<\infty$. By the triangle inequality and Minkowski's inequality, we have that
\begin{align}\label{Triangle Minkowski argument}
    \begin{split}
        \|S_\lambda (Vu)\|_{L^q (\gamma)} &=\left(\int \left|\int S_\lambda (\gamma(r), y) V(y) u(y)\:dy \right|^q\:dr \right)^{\frac{1}{q}} \\
        &\leq \int \left(\int |S_\lambda (\gamma(r), y) V(y) u(y)|^q\:dr \right)^{\frac{1}{q}}\:dy \\
        &=\int \left(\int |S_\lambda (\gamma(r), y)|^q\:dr \right)^{\frac{1}{q}} |V(y)| |u(y)|\:dy \\
        &\leq \sup_y \left(\int |S_\lambda (\gamma(r), y)|^q\:dr \right)^{\frac{1}{q}} \|u\|_{L^\infty (M)} \int |V(y)|\:dy \\
        &=\|u\|_{L^\infty (M)} \|V\|_{L^1 (M)} \sup_y \left(\int |S_\lambda (\gamma(r), y)|^q\:dr \right)^{\frac{1}{q}}.
    \end{split}
\end{align}
Again, we can apply \cite[Theorem 1.3]{BlairSireSogge2021Quasimode} to $\|u\|_{L^\infty (M)}$ at the end of computations. It thus suffices to bound the last factor
\begin{align*}
    \sup_y \left(\int |S_\lambda (\gamma(r), y)|^q\:dr \right)^{\frac{1}{q}}.
\end{align*}
Using \eqref{S lambda kernel est for n=2}, one can see that
\begin{align}\label{y sup S lambda kernel integration on curves estimate}
    \begin{split}
        \sup_y \left(\int |S_\lambda (\gamma (r), y)|^q\:dr \right)^{\frac{1}{q}}\lesssim \begin{cases}
            \lambda^{-\frac{1}{q}}, & \text{if } q>2, \\
            \lambda^{-\frac{1}{2}} (\log \lambda)^{\frac{1}{2}}, & \text{if } q=2.
        \end{cases}
    \end{split}
\end{align}
Indeed, if we take a local coordinate so that $\gamma$ is $\{(r, 0): |r|\ll 1\}$ and
\begin{align*}
    d_g (\gamma(r), y)=d_g ((r, 0), (y_1, y_2)) \approx |(r, 0)-(y_1, y_2)|,
\end{align*}
then the supremum over $y$ in \eqref{y sup S lambda kernel integration on curves estimate} is essentially obtained when $|y_2|\ll \lambda^{-1}$, at which point the bounds are easily verified. By \eqref{Triangle Minkowski argument}, \eqref{y sup S lambda kernel integration on curves estimate}, and \cite[Theorem 1.3]{BlairSireSogge2021Quasimode}, we have
\begin{align*}
    \|S_\lambda (Vu) \|_{L^q (\gamma)}\leq \begin{cases}
        C_V \lambda^{-\frac{1}{2}-\frac{1}{q}}\|(-\Delta_g+V-(\lambda+i)^2) u\|_{L^2 (M)}, & \text{when } 2<q<\infty, \\
        C_V \lambda^{-1} (\log \lambda)^{\frac{1}{2}} \|(-\Delta_g +V-(\lambda+i)^2)u\|_{L^2 (M)}, & \text{when } q=2,
    \end{cases}
\end{align*}
and this satisfies \eqref{S(Vu) claim for curves of universal estimates} when $2\leq q<\infty$. This completes the proof of Proposition \ref{Prop Reduction for curves Kato case}.

\section{Proof of Theorem \ref{Thm for hypersurfaces}}\label{S: Proof for any hypersurfaces}

In this section, we prove Proposition \ref{Prop Reduction for hypersurface thm}, which proves \eqref{Hypersurface estimates} and \eqref{Hypersurface log loss estimates} for any hypersurface, as explained in \S \ref{SS: Thm for hypersurfaces}. We first consider the case where $\Sigma$ is any hypersurface of $M$. Let $P=\sqrt{-\Delta_g}$. Recall from \eqref{S lambda W lambda set up} that
\begin{align*}
    S_\lambda=\frac{i}{\lambda+i}\int_0^\infty \mu_0 (t) e^{i\lambda t} e^{-t} (\cos tP)\:dt.
\end{align*}
We want to decompose $S_\lambda$ as in \cite[(2.19)-(2.20)]{BourgainShaoSoggeYao2015Resolvent}. We first fix a Littlewood-Paley type bump function $\beta_1 \in C_0^\infty (\mathbb{R})$ such that
\begin{align*}
    \beta_1 (t)=0 \;\text{ for } t\not\in [1/2, 2],\quad |\beta_1 (t)|\leq 1,\quad \text{and } \sum_{j=-\infty}^\infty \beta_1 (2^{-j}t)=1 \;\text{ for } t>0.
\end{align*}
We then define operators
\begin{align}\label{Sj definition}
    S_j f=\frac{i}{\lambda+i}\int_0^\infty \beta_1 (\lambda 2^{-j} t) \mu_0 (t) e^{i\lambda t} e^{-t} (\cos tP) f\:dt,\quad j=1, 2, 3, \cdots,
\end{align}
and
\begin{align}\label{S0 definition}
    S_0 f=\frac{i}{\lambda+i}\int_0^\infty \beta_0 (\lambda t) \mu_0 (t) e^{i\lambda t} e^{-t} (\cos tP) f\:dt,
\end{align}
where
\begin{align*}
    \beta_0 (t)=\left(1-\sum_{j=0}^\infty \beta_1 (2^{-j}t) \right)\in C^\infty (\mathbb{R}),
\end{align*}
and hence, $\beta_0 (t)=0$ if $|t|\geq 4$. We first consider the $S_0$ piece.

\begin{lemma}\label{Lemma S0 universal estimates for hypersurface}
    Let $\Sigma$ be a hypersurface of $M$. If $n\geq 3$ and $\frac{n}{p}-\frac{n-1}{q}=2$, then
    \begin{align*}
        \|S_0 f\|_{L^q (\Sigma)}\lesssim \|f\|_{L^p (M)}, \quad 2\leq q<\infty.
    \end{align*}
\end{lemma}

This is one reason why the case for $(n, q)=(3, \infty)$ does not hold in Theorem \ref{Thm for hypersurfaces}. To prove this lemma, we first recall the estimate of the kernel of the operator $S_0$. Using stationary phase as in the proof of \cite[(5.11)]{BlairHuangSireSogge2022UniformSobolev}, if $S_0 (x, y)$ denotes the kernel of the operator $S_0$, then
\begin{align*}
    |S_0 (x, y)|\lesssim d_g (x, y)^{2-n} \mathds{1}_{d_g (x, y)\lesssim \lambda^{-1}} (x, y).
\end{align*}
We choose coordinates so that $(z, 0)\in \mathbb{R}^{n-1}\times \mathbb{R}$ and $(z', s)\in \mathbb{R}^{n-1}\times \mathbb{R}$ are the coordinates of $x\in \Sigma$ and $y\in M$, respectively. Then
\begin{align*}
    (d_g (x, y))^{2-n}\lesssim |(z, 0)-(z', s)|^{2-n}=(|z-z'|+|s|)^{2-n}.
\end{align*}
With this in mind, Lemma \ref{Lemma S0 universal estimates for hypersurface} follows from the following proposition.

\begin{proposition}\label{Prop Sobolev trace for hypsurf}
    Suppose $\frac{n}{p}-\frac{n-1}{q}=2$, $1<p, q<\infty$, and $n\geq 3$. We write coordinates in $\mathbb{R}^n$ as $(y, s)\in \mathbb{R}^{n-1} \times \mathbb{R}$. Define
    \begin{align*}
        k(x, y, s)=(|x-y|+|s|)^{2-n},\quad \text{where } x\in \mathbb{R}^{n-1}.
    \end{align*}
    Then the operator
    \begin{align*}
        Tf(x):=\int k(x, y, s) f(y, s)\:ds\:dy
    \end{align*}
    defines a bounded linear map $T:L^p (\mathbb{R}^n) \to L^q (\mathbb{R}^{n-1})$.
\end{proposition}

\begin{proof}
    Without loss of generality, we assume $f\geq 0$. We begin with
    \begin{align*}
        |Tf(x)|\leq \int \left(\int |k(x, y, s)|^{p'}\:ds \right)^{\frac{1}{p'}} \|f(y, \cdot)\|_p\:dy,
    \end{align*}
    using H\"older's inequality. For convenience, set $\alpha=|x-y|$ so that
    \begin{align*}
        \left(\int |k(x, y, s)|^{p'}\:ds \right)^{\frac{1}{p'}}=\left(\int (\alpha+|s|)^{(2-n)p'}\:ds \right)^{\frac{1}{p'}}=\alpha^{2-n}\left(\int (1+\alpha^{-1}|s|)^{(2-n)p'}\:ds \right)^{\frac{1}{p'}}.
    \end{align*}
    After the change of variable $s=\alpha t$,
    \begin{align*}
        \left(\int (1+\alpha^{-1}|s|)^{(2-n)p'}\:ds \right)^{\frac{1}{p'}}=\left(\int (1+|t|)^{(2-n)p'} \alpha\:dt \right)^{\frac{1}{p'}}\leq C \alpha^{\frac{1}{p'}},
    \end{align*}
    for some $C$. Indeed, $(2-n)p'<-1$ since this is equivalent to $2-n<-\frac{1}{p'}$, which is trivial when $n\geq 3$. This shows that
    \begin{align*}
        |Tf(x)|\leq \int |x-y|^{\frac{1}{p'}+2-n} \|f(y, \cdot)\|_p\:dy.
    \end{align*}
    
    The claim now follows by the Hardy-Littlewood-Sobolev fractional integration. To see this, recall that convolution with $|y|^{-\frac{n-1}{r}}$ maps $L^p (\mathbb{R}^{n-1})\to L^q (\mathbb{R}^{n-1})$ boundedly if $\frac{1}{r}=1-\left(\frac{1}{p}-\frac{1}{q} \right)$. In our case, $r$ must satisfy
    \begin{align*}
        -\frac{n-1}{r}=\frac{1}{p'}+2-n=\left(1-\frac{1}{p}\right)+\left(\frac{n}{p}-\frac{n-1}{q} \right)-n=1-n+(n-1)\left(\frac{1}{p}-\frac{1}{q} \right).
    \end{align*}
    Dividing by $1-n$, we see that $\frac{1}{r}=1-\left(\frac{1}{p}-\frac{1}{q} \right)$, and so the claim follows.
\end{proof}

By Lemma \ref{Lemma S0 universal estimates for hypersurface}, it suffices to show that the $S_j$ for $j\geq 1$ satisfy the estimates in Proposition \ref{Prop Reduction for hypersurface thm}. By the proof of \cite[Lemma 5.1.3]{Sogge1993fourier}, modulo $O(\lambda^{-N})$ errors, we can write
\begin{align*}
    S_j f(x)=\lambda^{\frac{n-3}{2} } \int e^{i\lambda d_g (x, y) } \frac{a_\lambda (x, y) }{d_g (x, y)^{\frac{n-1}{2}} } f(y)\:dy,
\end{align*}
where the amplitude $a_\lambda$ is supported where
\begin{align*}
    \frac{1}{2}(\lambda^{-1} 2^j)\leq d_g (x, y)\leq 2(\lambda^{-1} 2^j).
\end{align*}
By this, the kernel of $S_j$ vanishes when $d_g (x, y)\not\in [\lambda^{-1} 2^{j-1}, \lambda^{-1} 2^{j+1}]$, and thus, by taking a suitable partition of unity, it can be seen that it suffices to assume
\begin{align}\label{Support of f in the dyadic operator}
    \mathrm{supp} f\subset B_{\lambda^{-1} 2^j} (0).
\end{align}
By construction, we note that $S_j=0$ if $j>\log_2 \lambda+C$, and thus, we restrict our attention to $j\leq \lfloor \log_2 \lambda \rfloor$ in what follows. If we set
\begin{align}\label{Scaling set-up for X, Y, and dj}
    x=\lambda^{-1} 2^j X,\quad y=\lambda^{-1} 2^j Y, \quad d_j (X, Y) =\lambda 2^{-j} d_g (\lambda^{-1} 2^j X, \lambda^{-1} 2^j Y),
\end{align}
then we can write
\begin{align}\label{Sj and Sj tilde set up}
    \begin{split}
        S_j f(x)&=S_j f(\lambda^{-1} 2^j X) \\
        &=\lambda^{\frac{n-3}{2} } \int e^{i 2^j d_j (X, Y) } \frac{a_\lambda (\lambda^{-1} 2^j X, \lambda^{-1} 2^j Y) }{d_j (X, Y)^{\frac{n-1}{2} } } \cdot (\lambda 2^{-j})^{\frac{n-1}{2} } f(\lambda^{-1} 2^j Y) (\lambda^{-1} 2^j)^n\:dY \\
        &=\lambda^{\frac{n-3}{2} } (\lambda^{-1} 2^j)^{\frac{n+1}{2}} \int e^{i 2^j d_j (X, Y) } \frac{a_\lambda (\lambda^{-1} 2^j X, \lambda^{-1} 2^j Y) }{d_j (X, Y)^{\frac{n-1}{2}}} f(\lambda^{-1}2^j Y)\:dY \\
        &=:\lambda^{\frac{n-3}{2} } (\lambda^{-1} 2^j)^{\frac{n+1}{2} } \Tilde{S}_j f_j(X),
    \end{split}
\end{align}
where $f_j (Y)=f(\lambda^{-1} 2^j Y)$. We note that $d_j (X, Y)$ is the Riemannian distance between $X$ and $Y$ with a ``stretched'' metric $g_{ij}(\lambda^{-1} 2^j X)$. We also note that by \eqref{Support of f in the dyadic operator}, we may assume that $f_j$ is supported in a compact set, that is, we may assume that
\begin{align}\label{Support of fj in the dyadic operator}
    \mathrm{supp} f_j \subset B_1 (0).
\end{align}
We are computing estimates for $S_j$ and $\Tilde{S}_j$ locally, and thus, in the practical computations of the estimates, abusing notations, we can write
\begin{align}\label{Sj Sj tilde notations}
    \begin{split}
        & \|S_j\|_{L^p (M)}, \|S_j\|_{L^q (\Sigma)}, \|\Tilde{S}_j\|_{L^p (M)}, \text{ and } \|\Tilde{S}_j\|_{L^q (\Sigma)}, \quad\text{as} \\
        & \|S_j\|_{L^p (\mathbb{R}^n)}, \|S_j \|_{L^q (\mathbb{R}^k)}, \|\Tilde{S}_j\|_{L^p (\mathbb{R}^n)}, \text{ and } \|\Tilde{S}_j\|_{L^q (\mathbb{R}^k)}, \text{ respectively,} \\
        & \text{where } n=\dim M \text{ and } k=\dim \Sigma.
    \end{split}
\end{align}
We can also use analogous notations for $\Tilde{S}_j$. One reason why we are abusing notations here is that the distance function $d_j$ in \eqref{Scaling set-up for X, Y, and dj} is the Riemannian distance function locally (but may not be the Riemannian distance function globally), and so, we shall use notations in \eqref{Sj Sj tilde notations} especially when we use the distance function $d_j$ directly, i.e., when we estimate $\Tilde{S}_j$.

We are also making use of the change of variables for $S_j$ and $\Tilde{S}_j$. Since $S_j$ is defined for the variable $x$ and $\Tilde{S}_j$ is defined for the variable $X$, to distinguish this difference, we write
\begin{align}\label{Sj Sj tilde Lp notations}
    \begin{split}
        & \|S_j f\|_{L_y^p (\mathbb{R}^n)}=\left(\int_{\mathbb{R}^n} |S_j f(y)|^p \:dy \right)^{\frac{1}{p}},\quad \|\Tilde{S}_j f_j\|_{L_Y^p (\mathbb{R}^n)}=\left(\int_{\mathbb{R}^n} |\Tilde{S}_j f_j(Y)|^p\:dY \right)^{\frac{1}{p}},\quad n=\dim M, \\
        & \|S_j f\|_{L_x^p (\mathbb{R}^k)}=\left(\int_{\mathbb{R}^k} |S_j f(x)|^p \:dx \right)^{\frac{1}{p}},\quad \|\Tilde{S}_j f_j\|_{L_X^p (\mathbb{R}^k)}=\left(\int_{\mathbb{R}^k} |\Tilde{S}_j f_j(X)|^p\:dX \right)^{\frac{1}{p}}, \quad k=\dim \Sigma.
    \end{split}
\end{align}
In fact, $\|\Tilde{S}_j f_j\|_{L_X^p (\mathbb{R}^{n-1})}$ (similarly $\|f_j\|_{L_X^p (\mathbb{R}^k)}$ and $\|S_j f\|_{L_x^p (\mathbb{R}^k)}$ as well) may be written as
\begin{align}\label{Sj tilde on submanifolds integration}
    \left(\int_{\mathbb{R}^k} |\Tilde{S}_j (h(X))|^p \kappa_h (X)\:dX \right)^{\frac{1}{p}},
\end{align}
where $h:\mathbb{R}^k \to \Sigma\subset M$ locally defined by
\begin{align*}
    X=(X_1, \cdots, X_k) \mapsto h(X)=(h_1 (X), \cdots, h_n (X))\in \Sigma
\end{align*}
is a smooth coordinate map, the $h_i$ are component functions, and $k_h (X)$ is a volume element from the coordinate map $h$ and the (induced) metric. For simplicity, we write \eqref{Sj tilde on submanifolds integration} as $\|\Tilde{S}_j f_j\|_{L_x^p (\mathbb{R}^k)}$ in \eqref{Sj Sj tilde Lp notations} considering that $\kappa_h$ may be absorbed to the amplitudes of oscillatory integral operators we shall think about. It then follows that
\begin{align}\label{Sj Sj tilde norm for hypersurf}
    \begin{split}
        & \|S_j \|_{L_y^p (\mathbb{R}^n)\to L_x^q (\mathbb{R}^{n-1})}=\lambda^{\frac{n-3}{2}} (\lambda^{-1} 2^j)^{\frac{n+1}{2}+\frac{n-1}{q}-\frac{n}{p}} \|\Tilde{S}_j \|_{L_Y^p (\mathbb{R}^n)\to L_X^q (\mathbb{R}^{n-1})}, \quad \text{and} \\
        & \|S_j \|_{L_y^p (\mathbb{R}^n)\to L_x^q (\mathbb{R}^{n-1})}=(2^j)^{\frac{n-3}{2}} \|\Tilde{S}_j \|_{L_Y^p (\mathbb{R}^n)\to L_X^q (\mathbb{R}^{n-1})},\quad \text{when } \frac{n}{p}-\frac{n-1}{q}=2.
    \end{split}
\end{align}

We also note that if $\Tilde{S}_j (X, Y)$ denotes the kernel of the operator $\Tilde{S}_j$, then by \eqref{Sj and Sj tilde set up} we can write
\begin{align*}
    \Tilde{S}_j (X, Y)=e^{i 2^j d_j (X, Y) } \frac{a_\lambda (\lambda^{-1} 2^j X, \lambda^{-1} 2^j Y) }{d_j (X, Y)^{\frac{n-1}{2}}},
\end{align*}
and thus, $|\Tilde{S}_j (X, Y)|\lesssim 1$. This gives us that for any $r\geq 1$,
\begin{align*}
    \sup_X \left(\int |\Tilde{S}_j (X, Y)|^r\:dX\right)^{\frac{1}{r}},\quad \sup_Y \left(\int |\Tilde{S}_j (X, Y)|^r\:dY\right)^{\frac{1}{r}} \lesssim 1,
\end{align*}
and hence, by Young's inequality, we have that
\begin{align}\label{A trivial Young's inequality consequence}
    \|\Tilde{S}_j\|_{L_Y^p (\mathbb{R}^n)\to L_X^q (\mathbb{R}^{n-1})}\lesssim 1,\quad \text{for any } 1\leq p\leq q\leq \infty.
\end{align}
With this in mind, we want to find a few nontrivial estimates of $S_j$ in the following lemmas.

\begin{lemma}\label{Lemma hypsurf p0 q0 estimate}
    If $1\leq j\leq \lfloor \log_2 \lambda \rfloor$, then
    \begin{align}\label{Dyadic critical estimate}
        \|S_j f\|_{L_x^{\frac{2n}{n-1}}(\mathbb{R}^{n-1})}\lesssim \lambda^{-\frac{2n+1}{2n}} (2^j)^{\frac{1}{2}} \|f\|_{L_y^2 (\mathbb{R}^n)},
    \end{align}
    and
    \begin{align}\label{Dyadic hypsurf p0 q0 estimate}
        \|S_j f\|_{L^{q_0} (\mathbb{R}^{n-1})}\lesssim (2^j)^{-\frac{3n-1}{2(2n^2-2n+1)}}\|f\|_{L^{p_0} (\mathbb{R}^n)},
    \end{align}
    where
    \begin{align*}
        q_0=\frac{2n^2-2n+1}{(n-1)(n-2)},\quad p_0=\frac{2n^2-2n+1}{n^2+1},\quad \frac{n}{p_0}-\frac{n-1}{q_0}=2.
    \end{align*}
\end{lemma}

\begin{proof}
    We first prove \eqref{Dyadic critical estimate}. As in \cite{Hu2009lp}, given the Riemannian distance $d_j (X, Y)$ (as in \eqref{Sj and Sj tilde set up}) for $x\in \Sigma$ and $y\in M$ with $x=\lambda^{-1}2^j X$ and $y=\lambda^{-1} 2^j Y$, we introduce the polar coordinates for $Y$, say, $Y=r\omega$ for $\omega\in \mathbb{S}^{n-1}$. We set the operator with $r$ fixed
    \begin{align*}
        \Tilde{S}_j^r (f_j)_r (X)=\int_{\mathbb{S}^{n-1}} e^{i2^j (d_j)_r (X, \omega)} \frac{a_r (\lambda^{-1}2^j X, \lambda^{-1} 2^j r\omega)}{d_j (X, r\omega)^{\frac{n-1}{2}}} (f_j)_r (\omega)\:d\omega,
    \end{align*}
    where $(f_j)_r (\omega)=f_j (Y), (d_j)_r (X, \omega)=d_j (X, Y)$, and $a_r (\lambda^{-1}2^j X, \lambda^{-1}2^j r\omega)=r^{n-1} a_\lambda (\lambda^{-1} 2^j X, \lambda^{-1} 2^j Y)$. By Theorem \ref{Thm Section 4 in Hu}, we know that the left projection of the canonical relation associated with $(d_j)_r (X, \omega)$ satisfies the hypothesis of \cite[Theorem 2.2]{GreenleafSeeger1994fourier}, and thus, by Theorem \ref{Thm for hypersurfaces in Greenleaf-Seeger}, we have
    \begin{align*}
        \|\Tilde{S}_j^r (f_j)_r\|_{L^{\frac{2n}{n-1}} (\mathbb{R}^{n-1})}\lesssim (2^j)^{-\frac{(n-1)^2}{2n}}\|(f_j)_r\|_{L^2 (\mathbb{S}^{n-1})}.
    \end{align*}
    By Minkowski's integral inequality,
    \begin{align*}
        \|\Tilde{S}_j f_j\|_{L^{\frac{2n}{n-1}} (\mathbb{R}^{n-1})}&\leq \int_{r\approx 1} \|\Tilde{S}_j^r (f_j)_r \|_{L^q (\mathbb{R}^{n-1})}\:dr \\
        &\lesssim (2^j)^{-\frac{(n-1)^2}{2n}}\int_{r\approx 1} \|(f_j)_r\|_{L^2 (\mathbb{S}^{n-1})}\:dr \\
        &\lesssim (2^j)^{-\frac{(n-1)^2}{2n}}\|f_j\|_{L^2 (\mathbb{R}^n)}.
    \end{align*}
    This and \eqref{Sj Sj tilde norm for hypersurf} imply \eqref{Dyadic critical estimate}.
    
    On the other hand, we have a trivial bound
    \begin{align*}
        \|S_j f\|_{L^\infty (\mathbb{R}^{n-1})}\lesssim \lambda^{n-2} (2^j)^{-\frac{n-1}{2}} \|f\|_{L^1 (\mathbb{R}^n)}.
    \end{align*}
    Interpolation between this and \eqref{Dyadic critical estimate} gives us the estimate \eqref{Dyadic hypsurf p0 q0 estimate}.
\end{proof}

\begin{lemma}\label{Lemma purple dot for any hypersurfaces}
    If $1\leq j\leq \lfloor \log_2 \lambda \rfloor$, then
    \begin{align}\label{Purple dot estimate}
        \|S_j f\|_{L^{\frac{2(n-1)}{n-3}}(\mathbb{R}^{n-1})}\lesssim \|f\|_{L^{\frac{2n}{n+1}} (\mathbb{R}^n)}.
    \end{align}
\end{lemma}

\begin{proof}
    Let $p_1=\frac{2n(n+1)}{2n^2-n+1}$ so that $\left(\frac{1}{2}, \frac{n-1}{2n}\right)$, $\left(\frac{n+1}{2n}, \frac{n-3}{2(n-1)} \right)$, and $\left(\frac{1}{p_1}, 0\right)$ are all collinear in the $\left(\frac{1}{p}, \frac{1}{q} \right)$ plane. It follows from \eqref{A trivial Young's inequality consequence} that
   \begin{align*}
       \|\Tilde{S}_j f_j\|_{L_X^\infty (\mathbb{R}^{n-1})}\lesssim \|f_j \|_{L_Y^{p_1} (\mathbb{R}^n)}.
   \end{align*}
   It follows from this and \eqref{Sj Sj tilde norm for hypersurf} that
   \begin{align*}
       \|S_j f\|_{L_x^\infty (\mathbb{R}^{n-1})}\lesssim \lambda^{\frac{n-3}{2}} (\lambda^{-1} 2^j)^{\frac{n+1}{2}-\frac{n}{p_1}} \|f\|_{L_y^{p_1}(\mathbb{R}^n)}.
   \end{align*}
   Interpolating this and \eqref{Dyadic critical estimate} gives the bound \eqref{Purple dot estimate}.
\end{proof}

\begin{lemma}\label{Lemma Sj critical uniform est}
    If $1\leq j\leq \lfloor \log_2 \lambda \rfloor$, then
    \begin{align}\label{Sj upper unif for hypersurf higher for low dimensions 3, 4, 5}
        \|S_j f\|_{L^{\frac{2(n-1)^2}{n^2-3n+4}}(\mathbb{R}^{n-1})}\lesssim \|f\|_{L^{\frac{2(n-1)}{n+1}}(\mathbb{R}^n)},\quad \text{for } n=3, 4, 5,
    \end{align}
    and
    \begin{align}\label{Sj upper unif for hypsurf higher}
        \|S_j f\|_{L^{\frac{2n^2-5n+4}{n^2-4n+8}}(\mathbb{R}^{n-1})}\lesssim \|f\|_{L^{\frac{2n^2-5n+4}{n^2-n+2}} (\mathbb{R}^n)},\quad \text{for } n\geq 6.
    \end{align}
\end{lemma}

\begin{proof}
    By Theorem \ref{Thm for hypersurfaces in Greenleaf-Seeger} and Theorem \ref{Thm Section 4 in Hu} as in the proof of Lemma \ref{Lemma hypsurf p0 q0 estimate},
    \begin{align*}
        \|\Tilde{S}_j f_j \|_{L_X^2 (\mathbb{R}^{n-1})}\lesssim (2^j)^{-\frac{n-2}{2}-\frac{1}{4}}\|f_j \|_{L_Y^2 (\mathbb{R}^n)}.
    \end{align*}
    We also have a trivial $L^1 \to L^\infty$ bound
    \begin{align*}
        \|\Tilde{S}_j f_j \|_{L_X^\infty (\mathbb{R}^{n-1})}\lesssim \|f_j \|_{L_Y^1 (\mathbb{R}^n)}.
    \end{align*}
    By interpolation,
    \begin{align*}
        \|\Tilde{S}_j f_j \|_{L_X^{\frac{2n-1}{n-2}} (\mathbb{R}^{n-1})} \lesssim (2^{-j})^{\frac{(2n-3)(n-2)}{2(2n-1)}} \|f_j \|_{L_Y^{\frac{2n-1}{n+1}}(\mathbb{R}^n)}.
    \end{align*}
    By this and \eqref{Sj Sj tilde norm for hypersurf},
    \begin{align}\label{Sj est q=(2n-1)/(n-2) for hypersurf}
        \|S_j f\|_{L_x^\frac{2n-1}{n-2} (\mathbb{R}^{n-1})}\lesssim (2^j)^{-\frac{3}{2(2n-1)}} \|f\|_{L_y^{\frac{2n-1}{n+1}}(\mathbb{R}^n)}.
    \end{align}

    Recall that $d_j(X,Y)$ as in \eqref{Sj and Sj tilde set up} is the Riemannian metric associated with the ``stretched metric" determined by the metric tensor $g(\lambda^{-1}2^jX)$. By taking a partition of unity and a careful change of coordinates, we can assume that the variables $X= (X',X_{n-1}) \in \mathbb{R}^{n-2}\times \mathbb{R}$, $Y = (Y',Y_n) \in \mathbb{R}^{n-1}\times \mathbb{R}$ are such that the mixed Hessian $ \frac{\partial^2 d_j}{\partial Y'\partial X'}$ is of full rank and the submanifold parameterized by $X'\mapsto \frac{\partial d_j}{\partial Y'} (X',X_{n-1},Y',Y_{n})$ defines a hypersurface in $\mathbb{R}^{n-1}$ with nonvanishing Gaussian curvature (i.e. the Carleson-Sj\"olin condition from \cite[\S2.2]{Sogge1993fourier} is satisfied in $n-1$ dimensions). Indeed, this can be verified in the special case where the submanifold is a subset of the $Y_n=0$ hyperplane, the restricted distance function $|(Y'-X,Y_n)|$  is Euclidean, and a partition of unity localizes to a small cone where $|Y'-X| \lesssim |X_{n-1}-Y_{n-1}|$ and $|Y_n| \ll1$. The Carleson-Sj\"olin condition is then stable under small perturbations. Hence we can apply \cite[Theorem 2.2.1]{Sogge1993fourier} in $n-1$ dimensions to the following operator which fixes $X_{n-1}, Y_n$ 
    $$
    (T_{j,X_{n-1},Y_n} G)(Y')  = \int \tilde{S}_j^*(X',X_{n-1}, Y',Y_n) G(X') dX'
    $$
    (where as usual, $\tilde{S}_j^*(X',X_{n-1}, Y',Y_n)$ denotes the integral kernel of $\tilde{S}_j^*$). Since the amplitude defining $\tilde{S}_j^*(X,Y)$ is supported in a small neighborhood of the diagonal, it then follows that 
    $$
    \|\tilde{S}_j^* f_j\|_{L_Y^\frac{2n}{n-2}(\mathbb{R}^n)} \lesssim (2^j)^{- \frac{(n-1)(n-2)}{2n}} \|f_j\|_{L_X^2(\mathbb{R}^{n-1})}.
    $$
    By duality, we have
    \begin{align}\label{Sj tilde est q=2, p=2n/(n+2) for hypersurf}
        \|\Tilde{S}_j f_j\|_{L_X^2(\mathbb{R}^{n-1})}\lesssim (2^j)^{-\frac{(n-1)(n-2)}{2n}} \|f_j\|_{L_Y^{\frac{2n}{n+2}} (\mathbb{R}^n)}.
    \end{align}
    
    We first consider $n=3, 4, 5$, i.e., $3\leq n\leq 5$. By \eqref{A trivial Young's inequality consequence}, we have
    \begin{align}\label{Sj tilde est q=4(n-1)/(7n-n^2-8), p=1 for hypersurf}
        \|\Tilde{S}_j f_j \|_{L^{\frac{4(n-1)}{7n-n^2-8}}(\mathbb{R}^{n-1})}\lesssim \|f\|_{L^1 (\mathbb{R}^n)}.
    \end{align}
    We note that
    \begin{align*}
        \frac{4(n-1)}{7n-n^2-8}=\begin{cases}
            2, & \text{if } n=3, \\
            3, & \text{if } n=4, \\
            8, & \text{if } n=5,
        \end{cases}
    \end{align*}
    but $7n-n^2-8<0$ if $n\geq 6$, and this is a reason in Lemma \ref{Lemma Sj critical uniform est} why we split the cases into two cases where $n\leq 5$ and $n\geq 6$. Interpolating \eqref{Sj tilde est q=2, p=2n/(n+2) for hypersurf} and \eqref{Sj tilde est q=4(n-1)/(7n-n^2-8), p=1 for hypersurf} yields
    \begin{align*}
        \|\Tilde{S}_j f_j\|_{L^{\frac{2(n-1)^2}{n^2-3n+4}}(\mathbb{R}^{n-1})}\lesssim (2^j)^{-\frac{n-3}{2}} \|f_j\|_{L^{\frac{2(n-1)}{n+1}}(\mathbb{R}^n)},
    \end{align*}
    and thus, by \eqref{Sj Sj tilde norm for hypersurf}, we have
    \begin{align*}
        \|S_j f\|_{L^{\frac{2(n-1)^2}{n^2-3n+4}}(\mathbb{R}^{n-1})}\lesssim \|f\|_{L^{\frac{2(n-1)}{n+1}}(\mathbb{R}^n)}.
    \end{align*}
    This proves \eqref{Sj upper unif for hypersurf higher for low dimensions 3, 4, 5}.
    
    We next consider $n\geq 6$. Interpolating \eqref{Sj tilde est q=2, p=2n/(n+2) for hypersurf} with a trivial $L^1 (\mathbb{R}^n)\to L^\infty (\mathbb{R}^{n-1})$ bound
    \begin{align*}
        \|\Tilde{S}_j f_j \|_{L_X^\infty (\mathbb{R}^{n-1})}\lesssim \|f_j\|_{L_Y^1 (\mathbb{R}^n)}
    \end{align*}
    yields
    \begin{align*}
        \|\Tilde{S}_j f_j\|_{L_X^{\frac{2n-3}{n-2}}(\mathbb{R}^{n-1})}\lesssim (2^j)^{-\frac{(n-1)(n-2)^2}{n(2n-3)}}\|f_j\|_{L_Y^{\frac{2n^2-3n}{n^2+n-4}} (\mathbb{R}^n)}.
    \end{align*}
    It follows from \eqref{Sj Sj tilde norm for hypersurf} that
    \begin{align}\label{Sj est q=(2n-3)/(n-2) for hypersurf}
        \|S_j f\|_{L^{\frac{2n-3}{n-2}} (\mathbb{R}^{n-1})}\lesssim (2^j)^{\frac{n^2-7n+8}{2n(2n-3)}}\|f\|_{L^{\frac{2n^2-3n}{n^2+n-4}}(\mathbb{R}^n)}.
    \end{align}
    Interpolation between \eqref{Sj est q=(2n-1)/(n-2) for hypersurf} and \eqref{Sj est q=(2n-3)/(n-2) for hypersurf} yields \eqref{Sj upper unif for hypsurf higher}. For $n=\dim M\geq 3$, we note that $\frac{n^2-7n+8}{n^2+n-4}<0$ if and only if $n=3, 4, 5$, and this is another reason why we need to consider the cases $n\leq 5$ and $n\geq 6$ separately.
\end{proof} 

We are now ready to prove \eqref{S(Vu) WTS hypersurface}. Interpolation between \eqref{Dyadic hypsurf p0 q0 estimate}, \eqref{Purple dot estimate}, \eqref{Sj upper unif for hypersurf higher for low dimensions 3, 4, 5}, and \eqref{Sj upper unif for hypsurf higher} gives us that, for some $\alpha_n (p, q)>0$,
\begin{align*}
    \|S_j f\|_{L^q (\mathbb{R}^{n-1})}\lesssim \begin{cases}
        (2^j)^{-\alpha_n (p, q)} \|f\|_{L^p (\mathbb{R}^n)}, & \text{if } \frac{n}{p}-\frac{n-1}{q}=2,\; 3\leq n\leq 5,\; \text{ and } \frac{2(n-1)^2}{n^2-3n+4}<q<\frac{2(n-1)}{n-3}, \\
        (2^j)^{-\alpha_n (p, q)} \|f\|_{L^p (\mathbb{R}^n)}, & \text{if } \frac{n}{p}-\frac{n-1}{q}=2,\; n\geq 6,\; \text{ and } \frac{2n^2-5n+4}{n^2-4n+8}<q<\frac{2(n-1)}{n-3}, \\
        \|f\|_{L^p (\mathbb{R}^n)}, & \text{if } \frac{n}{p}-\frac{n-1}{q}=2,\; 3\leq n\leq 6,\; \text{ and } q\in \left\{\frac{2(n-1)^2}{n^2-3n+4}, \frac{2(n-1)}{n-3} \right\}, \\
        \|f\|_{L^p (\mathbb{R}^n)}, & \text{if } \frac{n}{p}-\frac{n-1}{q}=2,\; n\geq 6,\; \text{ and } q\in \left\{\frac{2n^2-5n+4}{n^2-4n+8}, \frac{2(n-1)}{n-3} \right\}.
    \end{cases}
\end{align*}
Summing these over all $1\leq j\leq \lfloor \log_2 \lambda \rfloor$, we obtain
\begin{align*}
    \|S_\lambda f\|_{L^q (\mathbb{R}^{n-1})}\lesssim \begin{cases}
        \|f\|_{L^p (\mathbb{R}^n)}, & \text{if } \frac{n}{p}-\frac{n-1}{q}=2,\; 3\leq n\leq 5,\; \text{ and } \frac{2(n-1)^2}{n^2-3n+4}<q<\frac{2(n-1)}{n-3}, \\
        \|f\|_{L^p (\mathbb{R}^n)}, & \text{if } \frac{n}{p}-\frac{n-1}{q}=2,\; n\geq 6,\; \text{ and } \frac{2n^2-5n+4}{n^2-4n+8}<q<\frac{2(n-1)}{n-3}, \\
        (\log \lambda)\|f\|_{L^p (\mathbb{R}^n)}, & \text{if } \frac{n}{p}-\frac{n-1}{q}=2,\; 3\leq n\leq 5,\; \text{ and } q=\frac{2(n-1)^2}{n^2-3n+4}, \\
        (\log \lambda)\|f\|_{L^p (\mathbb{R}^n)}, & \text{if } \frac{n}{p}-\frac{n-1}{q}=2,\; 4\leq n\leq 5,\; \text{ and } q=\frac{2(n-1)}{n-3}, \\
        (\log \lambda)\|f\|_{L^p (\mathbb{R}^n)}, & \text{if } \frac{n}{p}-\frac{n-1}{q}=2,\; n\geq 6,\; \text{ and } q\in \left\{\frac{2n^2-5n+4}{n^2-4n+8}, \frac{2(n-1)}{n-3} \right\},
    \end{cases}
\end{align*}
Here, we used the fact that the case of $(n, q)=(3, \infty)=\left(3, \frac{2(n-1)}{n-3}\right)$ does not hold by Lemma \ref{Lemma S0 universal estimates for hypersurface}. This completes the proof of \eqref{S(Vu) WTS hypersurface}, the proof of Proposition \ref{Prop Reduction for hypersurface thm}, and hence, Theorem \ref{Thm for hypersurfaces}.

\section{Proof of Theorem \ref{Thm for codim 2}}\label{S: proof of codim 2 thm}
By the discussion in \S \ref{S: Preliminaries}, we prove Proposition \ref{Prop Reduction for codim 2 thm} to complete the proof of Theorem \ref{Thm for codim 2}. We define $S_j$ and $S_0$ as in \eqref{Sj definition} and \eqref{S0 definition}, respectively. Let $\Sigma$ be an $(n-2)$-dimensional submanifold of $M$. We need an analogue of Lemma \ref{Lemma S0 universal estimates for hypersurface} first.

\begin{lemma}\label{Lemma S0 universal est for codim 2}
    If $\Sigma$ is an $(n-2)$-dimensional submanifold of $M$ and $\frac{n}{p}-\frac{n-2}{q}=2$, then
    \begin{align*}
        \|S_0 f\|_{L^q (\Sigma)}\lesssim \|f\|_{L^p (M)},\quad 2\leq q<\infty.
    \end{align*}
\end{lemma}

As in \S \ref{S: Proof for any hypersurfaces}, this lemma follows from the following proposition.

\begin{proposition}\label{Prop Sobolev trace for codim 2}
    Suppose $\frac{n}{p}-\frac{n-2}{q}=2$, $1<p, q<\infty$, and $n\geq 3$. We write coordinates in $\mathbb{R}^n$ as $(y, s)\in \mathbb{R}^{n-2} \times \mathbb{R}^2$, where $s=(s_1, s_2)\in \mathbb{R}^2$. Define
    \begin{align*}
        k(x, y, s)=(|x-y|+|s|)^{2-n},\quad \text{where } x\in \mathbb{R}^{n-2}.
    \end{align*}
    Then the operator
    \begin{align*}
        Tf(x):=\int k(x, y, s) f(y, s)\:ds\:dy
    \end{align*}
    defines a bounded linear map $T:L^p (\mathbb{R}^n) \to L^q (\mathbb{R}^{n-2})$.
\end{proposition}

As in Theorem \ref{Thm for hypersurfaces}, this proposition is a reason why the case of $(n, q)=(3, \infty)$ is not covered in Theorem \ref{Thm for codim 2}. Since one can prove Proposition \ref{Prop Sobolev trace for codim 2} by using the proof of Proposition \ref{Prop Sobolev trace for hypsurf}, we skip its proof here. By Lemma \ref{Lemma S0 universal est for codim 2}, it is enough to consider the estimates of the $S_j$ for $1\leq j\leq \lfloor \log_2 \lambda \rfloor$. If $\Tilde{S}_j$ is as in \eqref{Sj and Sj tilde set up}, then using notations in \eqref{Sj Sj tilde notations} and \eqref{Sj Sj tilde Lp notations}, we have analogues of \eqref{Sj Sj tilde norm for hypersurf} as follows.
\begin{align}\label{Sj Sj tilde norm for codim 2}
    \begin{split}
        & \|S_j \|_{L_y^p (\mathbb{R}^n) \to L_x^q (\mathbb{R}^{n-2})}=\lambda^{\frac{n-3}{2}} (\lambda^{-1} 2^j)^{\frac{n+1}{2}+\frac{n-2}{q}-\frac{n}{p}} \|\Tilde{S}_j \|_{L_Y^p (\mathbb{R}^n) \to L_X^q (\mathbb{R}^{n-2})},\quad \text{and} \\
        & \|S_j \|_{L_y^p (\mathbb{R}^n) \to L_x^q (\mathbb{R}^{n-2})}=(2^j)^{\frac{n-3}{2}} \|\Tilde{S}_j \|_{L_Y^p (\mathbb{R}^n) \to L_X^q (\mathbb{R}^{n-2})},\quad \text{if } \frac{n}{p}-\frac{n-2}{q}=2.
    \end{split}
\end{align}
We also note that $f_j$ is compactly supported by \eqref{Support of f in the dyadic operator} and \eqref{Support of fj in the dyadic operator}. Again, as in the compuation of \eqref{A trivial Young's inequality consequence}, we have that
\begin{align}\label{A trivial Young's inequality consequence for codim 2 submflds}
    \|\Tilde{S}_j\|_{L_Y^p (\mathbb{R}^n)\to L_X^q (\mathbb{R}^{n-2})}\lesssim 1,\quad \text{for any } 1\leq p\leq q\leq \infty.
\end{align}
We have the following lemmas analogous to Lemma \ref{Lemma hypsurf p0 q0 estimate}-\ref{Lemma Sj critical uniform est}.

\begin{lemma}\label{Lemma codim 2 universal middle est}
    If $1\leq j\leq \lfloor \log_2 \lambda \rfloor$, then
    \begin{align}\label{Sj intersection point est for codim 2}
        \|S_j f\|_{L^{\frac{2n-2}{n-2}}(\mathbb{R}^{n-2})}\lesssim j^{\frac{n-2}{2n-2}} (2^j)^{-\frac{1}{2n-2}} \|f\|_{L^{\frac{2n-2}{n}}(\mathbb{R}^n)}.
    \end{align}
\end{lemma}

\begin{proof}
    Let $\Tilde{S}_j$ be as in \eqref{Sj and Sj tilde set up}. As in the proof of Lemma \ref{Lemma hypsurf p0 q0 estimate}, by Theorem \ref{Thm 2.1 in GS} and Theorem \ref{Thm Section 4 in Hu}, we have that
    \begin{align}\label{Sj tilde L2 to L2 codim 2 est}
        \begin{split}
            \|\Tilde{S}_j f_j \|_{L_X^2 (\mathbb{R}^{n-2})}\lesssim (2^j)^{-\frac{n-2}{2}} (\log 2^j)^{\frac{1}{2}} \|f_j\|_{L_Y^2 (\mathbb{R}^n)}\lesssim (2^j)^{-\frac{n-2}{2}} j^{\frac{1}{2}} \|f_j \|_{L_Y^2 (\mathbb{R}^n)}.
        \end{split}
    \end{align}
    We also know the trivial $L^1 \to L^\infty$ bound
    \begin{align}\label{Sj tilde L1 to Linfty bound}
        \|\Tilde{S}_j f_j \|_{L_X^\infty (\mathbb{R}^{n-2})}\lesssim \|f_j \|_{L_Y^1 (\mathbb{R}^n)}.
    \end{align}
    By interpolation between this and \eqref{Sj tilde L2 to L2 codim 2 est},
    \begin{align*}
        \|\Tilde{S}_j f_j \|_{L_X^{\frac{2n-2}{n-2}} (\mathbb{R}^{n-2})}\lesssim j^{\frac{n-2}{2n-2}} (2^j)^{-\frac{(n-2)^2}{2n-2}} \|f_j \|_{L_Y^{\frac{2n-2}{n}}(\mathbb{R}^n)}.
    \end{align*}
    Then \eqref{Sj intersection point est for codim 2} follows from this and \eqref{Sj Sj tilde norm for codim 2}.
\end{proof}

\begin{lemma}\label{Lemma codim 2 universal bottom est}
    If $1\leq j\leq \lfloor \log_2 \lambda \rfloor$, then
    \begin{align}\label{Sj lower opr norm for codim 2}
        \|S_j f\|_{L^{\frac{2(n-2)}{n-3}}(\mathbb{R}^{n-2})}\lesssim j^{\frac{n-3}{2(n-2)}} \|f\|_{L^{\frac{2n}{n+1}}(\mathbb{R}^n)}.
    \end{align}
\end{lemma}

\begin{proof}
    By \eqref{Sj Sj tilde norm for codim 2} and \eqref{Sj tilde L2 to L2 codim 2 est},
    \begin{align}\label{Sj L2 to L2 estimate on codim 2}
        \|S_j f \|_{L^2 (\mathbb{R}^{n-2})} \lesssim \lambda^{-1} (2^j)^{\frac{1}{2}} j^{\frac{1}{2}} \|f\|_{L^2 (\mathbb{R}^n)}.
    \end{align}
    On the other hand, if $p_1=\frac{n}{n-1}$ so that $\left(\frac{1}{2}, \frac{1}{2} \right)$, $\left(\frac{n+1}{2n}, \frac{n-3}{2(n-2)}\right)$, and $\left(\frac{1}{p_1}, 0 \right)$ are collinear in the $\left(\frac{1}{p}, \frac{1}{q} \right)$ plane, then by \eqref{A trivial Young's inequality consequence for codim 2 submflds},
    \begin{align*}
        \|\Tilde{S}_j f_j \|_{L_X^\infty (\mathbb{R}^{n-2})}\lesssim \|f_j \|_{L_Y^{p_1} (\mathbb{R}^n)},
    \end{align*}
    and so, by \eqref{Sj Sj tilde norm for codim 2},
    \begin{align*}
        \|S_j f\|_{L^\infty (\mathbb{R}^{n-2})}\lesssim \lambda^{\frac{n-3}{2}} (\lambda^{-1} 2^j)^{\frac{n+1}{2}-\frac{n}{p_1}} \|f\|_{L^{p_1} (\mathbb{R}^n)}=\lambda^{n-3}(2^j)^{-\frac{n-3}{2}}\|f\|_{L^{p_1} (\mathbb{R}^n)}.
    \end{align*}
    Interpolating this and \eqref{Sj L2 to L2 estimate on codim 2} yields \eqref{Sj lower opr norm for codim 2}.
\end{proof}

\begin{lemma}\label{Lemma codim 2 universal top est}
    If $1\leq j\leq \lfloor \log_2 \lambda \rfloor$, then
    \begin{align}\label{Sj upper estimate for codim 2}
        \|S_j f\|_{L^2 (\mathbb{R}^{n-2})}\lesssim (2^j)^{\frac{n-4}{2n}} j^{\frac{n-2}{2n}} \|f\|_{L^{\frac{2n}{n+2}} (\mathbb{R}^n)}.
    \end{align}
\end{lemma}

\begin{proof}
    By \eqref{A trivial Young's inequality consequence for codim 2 submflds},
    \begin{align*}
        \|\Tilde{S}_j f_j \|_{L_X^1 (\mathbb{R}^{n-2})}\lesssim \|f_j\|_{L^1 (\mathbb{R}^n)}.
    \end{align*}
    By this and \eqref{Sj Sj tilde norm for codim 2},
    \begin{align}\label{Sj L1 to L1 estimate for codim 2}
        \|S_j f\|_{L^1 (\mathbb{R}^{n-2})} \lesssim (2^j)^{\frac{n-3}{2}} \|f\|_{L^1 (\mathbb{R}^n)}.
    \end{align}
    The estimate \eqref{Sj upper estimate for codim 2} then follows from interpolation between \eqref{Sj L1 to L1 estimate for codim 2} and \eqref{Sj intersection point est for codim 2}.
\end{proof}

We now come back to the proof of Proposition \ref{Prop Reduction for codim 2 thm}. Suppose $n\geq 4$. Interpolating \eqref{Sj intersection point est for codim 2} and \eqref{Sj upper estimate for codim 2} gives us that
\begin{align*}
    \|S_j f\|_{L^q (\mathbb{R}^{n-2})}\lesssim (2^j)^{-\frac{n^2-5n+8}{2n}+\frac{(n-2)^2}{nq}} j^{\frac{n-2}{n}-\frac{n-2}{q}} \|f\|_{L^p (\mathbb{R}^n)},\quad \text{for } \frac{n}{p}-\frac{n-2}{q}=2 \text{ and } 2\leq q\leq \frac{2n-2}{n-2},
\end{align*}
and thus, for some $\alpha(p, q)>0$,
\begin{align}\label{Sj codim 2 Lp to Lq lower}
    \|S_j \|_{L^p (\mathbb{R}^n)\to L^q (\mathbb{R}^{n-2})}\lesssim \begin{cases}
        (2^j)^{-\alpha(p, q)}, & \text{if } \frac{2(n-2)^2}{n^2-5n+8}<q\leq \frac{2n-2}{n-2}, \\
        j^{\frac{n-3}{2(n-2)}}, & \text{if }q=\frac{2(n-2)^2}{n^2-5n+8}.
    \end{cases}
\end{align}
Similarly, interpolating \eqref{Sj lower opr norm for codim 2} and \eqref{Sj upper estimate for codim 2}, we have, for some $\alpha(p, q)>0$,
\begin{align*}
    \|S_j \|_{L^p (\mathbb{R}^n)\to L^q (\mathbb{R}^{n-2})}\lesssim \begin{cases}
        (2^j)^{-\alpha(p, q)}, & \text{if } \frac{2n-2}{n-2}<q< \frac{2(n-2)}{n-3}, \\
        j^{\frac{n-3}{2(n-2)}}, & \text{if }q=\frac{2(n-2)}{n-3}.
    \end{cases}
\end{align*}
Using this with \eqref{Sj codim 2 Lp to Lq lower}, for some $\alpha(p, q)>0$, if $\frac{n}{p}-\frac{n-2}{q}=2$, then
\begin{align}\label{Sj Lp to Lq codim 2 higher dim}
    \begin{split}
        \|S_j\|_{L^p (\mathbb{R}^n)\to L^q (\mathbb{R}^{n-2})}\lesssim \begin{cases}
        (2^j)^{-\alpha (p, q)}, & \text{if } n\geq 4 \text{ and } \frac{2(n-2)^2}{n^2-5n+8}<q<\frac{2(n-2)}{n-3}, \\
        j^{\frac{n-3}{2(n-2)}}, & \text{if } n\geq 4 \text{ and } q\in \left\{\frac{2(n-2)}{n-3}, \frac{2(n-2)^2}{n^2-5n+8}\right\}.
        \end{cases}
    \end{split}
\end{align}
We also note that from Lemma \ref{Lemma codim 2 universal top est}, if $n=3$, then
\begin{align*}
    \|S_j f\|_{L^2 (\mathbb{R}^{n-2})}\lesssim (2^j)^{-\frac{1}{6}} j^{\frac{1}{6}} \|f\|_{L^{\frac{6}{5}}(\mathbb{R}^n)}.
\end{align*}
Interpolating this with \eqref{Sj lower opr norm for codim 2}, we have, for some $\alpha(p, q)>0$ with $\frac{n}{p}-\frac{n-2}{q}=2$,
\begin{align}\label{Sj Lp to Lq codim 2 3D}
    \begin{split}
        \|S_j \|_{L^p (\mathbb{R}^n)\to L^q (\mathbb{R}^{n-2})}\lesssim 
        (2^j)^{-\alpha (p, q)}, \quad\text{if } n=3 \text{ and } 2\leq q<\infty,        
    \end{split}
\end{align}

Since we already considered $S_0$, summing all $1\leq j\leq \lfloor \log_2 \lambda \rfloor$ in \eqref{Sj Lp to Lq codim 2 higher dim} and \eqref{Sj Lp to Lq codim 2 3D}, we have Proposition \ref{Prop Reduction for codim 2 thm}. This completes the proof of Theorem \ref{Thm for codim 2}.

\section{Proof of Theorem \ref{Thm for log improved}}\label{S: proof for log imp}

\subsection{Curves in surfaces}
In this subsection, we show \eqref{Estimate for log imp}. Let $P=\sqrt{-\Delta_g}$. To prove \eqref{Estimate for log imp}, we will use the estimates when $V\equiv 0$ from \cite{Chen2015improvement}, \cite{BlairSogge2018concerning}, \cite{Blair2018logarithmic}, \cite{XiZhang2017improved}, and \cite{Park2023Restriction}, for $P=\sqrt{-\Delta_g}$ and $\epsilon(\lambda)=(\log (2+\lambda))^{-1}$,
\begin{align}\label{Log imp V=0 spectral proj est}
    \|\mathds{1}_{[\lambda, \lambda+\epsilon(\lambda)]} (P)\|_{L^2 (M)\to L^q (\gamma_i)} \lesssim \lambda^{\delta(q, \gamma_i)} (\epsilon(\lambda))^{\kappa(q, \gamma_i)},\quad i=1, 2, 3.
\end{align}
By this and Lemma \ref{Lemma Resolvent opr norm from L2 to Lq for univ est}, we have
\begin{align}\label{Log imp V=0 est}
    \|(-\Delta_g-(\lambda+i\epsilon(\lambda))^2)^{-1}\|_{L^2 (M)\to L^q (\gamma_i)}\lesssim \lambda^{\delta(q, \gamma_i)-1} (\epsilon(\lambda))^{-1+\kappa(q, \gamma_i)}, \quad i=1, 2, 3,
\end{align}
where $(q, i)$, $\delta(q, \gamma_i)$, and $\kappa(q, \gamma_i)$ are as in \eqref{gamma i classifications}, \eqref{n=2 lambda exponent classifications}, and \eqref{n=2 log lambda exponent classifications}, respectively. By \cite[Theorem 5.1]{BlairHuangSireSogge2022UniformSobolev}, we have, for $u\in \mathrm{Dom}(H_V) \cap C(M)$
\begin{align}\label{Log imp est for infty}
    \|u\|_{L^\infty (\gamma_i)} \leq \|u\|_{L^\infty (M)}\lesssim \lambda^{-\frac{1}{2}} (\epsilon(\lambda))^{-\frac{1}{2}} \|(H_V-(\lambda+i\epsilon(\lambda))^2)u\|_{L^2 (M)},\quad i=1, 2, 3.
\end{align}
This proves the $q=\infty$ case of \eqref{Estimate for log imp} for any curve $\gamma_1$, and thus, we are left to prove the cases where $q<\infty$ for the other curves $\gamma_i$.

We will follow the argument in \cite{BlairHuangSireSogge2022UniformSobolev} to prove \eqref{Estimate for log imp}. Let $\eta\in C_0^\infty (\mathbb{R})$ be such that
\begin{align*}
    \eta(t)=1 \text{ for } t\in (-1/2, 1/2),\quad \text{and} \quad \mathrm{supp}(\eta)\subset (-1, 1).
\end{align*}
Let
\begin{align}\label{Definition of T and c0}
    T=c_0 (\epsilon(\lambda))^{-1},
\end{align}
where $c_0>0$ is a small real number which will be specified later. We shall write
\begin{align*}
    (-\Delta_g-(\lambda+i\epsilon(\lambda))^2)^{-1}=T_\lambda+R_\lambda,
\end{align*}
where
\begin{align}\label{T lambda R lambda set up}
    \begin{split}
        & T_\lambda=T_\lambda^0+T_\lambda^1, \\
        & T_\lambda^0=\frac{i}{\lambda+i\epsilon(\lambda)} \int_0^\infty \eta(t) \eta(t/T) e^{i\lambda t} e^{-\epsilon(\lambda)t} \cos (tP)\:dt, \\
        & T_\lambda^1=\frac{i}{\lambda+i\epsilon(\lambda)} \int_0^\infty (1-\eta(t))\eta(t/T) e^{i\lambda t} e^{-\epsilon(\lambda)t} \cos (tP)\:dt, \\
        & R_\lambda=\frac{i}{\lambda+i\epsilon(\lambda)}\int_0^\infty (1-\eta(t/T)) e^{i\lambda t} e^{-\epsilon(\lambda)t} \cos (tP)\:dt,
    \end{split}
\end{align}

To consider $R_\lambda$ first, we set
\begin{align}\label{m lambda setup}
    \tau \mapsto m_\lambda (\tau):=\frac{i}{\lambda+i\epsilon(\lambda)}\int_0^\infty (1-\eta(t/T)) e^{i\lambda t} e^{-\epsilon(\lambda)t} \cos (\tau P)\:dt,
\end{align}
which satisfies
\begin{align}\label{Size est for m lambda log}
    |m_\lambda (\tau)|\lesssim (\lambda\epsilon (\lambda))^{-1} (1+(\epsilon (\lambda))^{-1}|\lambda-\tau|)^{-N},\quad \text{for } N=1, 2, 3, \cdots,\quad \text{if } \tau\geq 0,\; \lambda\geq 1.
\end{align}
Since $R_\lambda=m_\lambda (P)$, by \eqref{Size est for m lambda log} and an orthogonality argument (cf. the proof of Lemma \ref{Lemma Resolvent opr norm from L2 to Lq for univ est}), we have
\begin{align}\label{R lambda log imp}
    \|R_\lambda\|_{L^2 (M)\to L^q (\gamma_i)}\lesssim \lambda^{\delta(q, \gamma_i)-1} (\epsilon(\lambda))^{-1+\kappa(q, \gamma_i)},
\end{align}
and
\begin{align}\label{R lambda with laplacian log imp}
    \|R_\lambda \circ (-\Delta_g-(\lambda+i\epsilon(\lambda))^2)\|_{L^2 (M)\to L^q (\gamma_i)}\lesssim \lambda^{\delta(q, \gamma_i)-1} (\epsilon(\lambda))^{-1+\kappa(q, i)}\cdot (\lambda \epsilon(\lambda)).
\end{align}
Since $T_\lambda=(-\Delta_g-(\lambda+\epsilon(\lambda))^2)^{-1}-R_\lambda$, it follows from \eqref{Log imp V=0 est} and \eqref{R lambda log imp} that
\begin{align}\label{T lambda log imp est}
    \|T_\lambda \|_{L^2 (M)\to L^q (\gamma_i)}\lesssim \lambda^{\delta(q, \gamma_i)-1} (\epsilon (\lambda))^{-1+\kappa (q, \gamma_i)}.
\end{align}
For any given $\epsilon_0>0$, if $c_0>0$ as in \eqref{Definition of T and c0} is small enough, then by the arguments in \cite[(5.10)]{BlairHuangSireSogge2022UniformSobolev} (see also \cite{Berard1977onthewaveequation}), the kernel of $T_\lambda^1$ is continuous, and so, we have that
\begin{align*}
    \|T_\lambda^1 \|_{L^1 (M)\to L^\infty (M)}=O(\lambda^{-\frac{1}{2}}\lambda^{Cc_0})=O(\lambda^{-\frac{1}{2}+\epsilon_0}),\quad \text{for all } 0<\epsilon_0\ll 1,
\end{align*}
which implies that
\begin{align}\label{T lambda L 1 to L infty}
    \|T_\lambda^1 f\|_{L^\infty (\gamma_i)}\leq \|T_\lambda^1 f\|_{L^\infty (M)}\lesssim \lambda^{-\frac{1}{2}+\epsilon_0} \|f\|_{L^1 (M)},\quad \text{for all } 0<\epsilon_0\ll 1.
\end{align}
It follows from \cite[(5.11)]{BlairHuangSireSogge2022UniformSobolev} that
\begin{align}\label{T lambda 0 kernel est}
    \begin{split}
        |T_\lambda^0 (x, y)|\leq \begin{cases}
        C|\log (\lambda d_g (x, y)/2)|,& \text{if } d_g (x, y)\leq \lambda^{-1}, \\
        C \lambda^{-\frac{1}{2}}(d_g (x, y))^{-\frac{1}{2}}, & \text{if } \lambda^{-1}\leq d_g (x, y)\ll 1.
        \end{cases}
    \end{split}
\end{align}

We now write
\begin{align}\label{u decomposition log imp}
    \begin{split}
        u&=(-\Delta_g-(\lambda+i\epsilon(\lambda))^2)^{-1}\circ (-\Delta_g-(\lambda+i\epsilon(\lambda))^2)u \\
        &=T_\lambda (-\Delta_g+V-(\lambda+i\epsilon(\lambda))^2)u+R_\lambda(-\Delta_g -(\lambda+i\epsilon(\lambda))^2)u-T_\lambda (Vu).
    \end{split}
\end{align}
We compute each of the three terms separately as above. By \eqref{T lambda log imp est}, we have
\begin{align*}
    \|T_\lambda \circ (-\Delta_g+V-(\lambda+i\epsilon (\lambda))^2)u\|_{L^q (\gamma_i)}\lesssim \lambda^{\delta(q, \gamma_i)-1} (\epsilon(\lambda))^{-1+\kappa(q, \gamma_i)} \|(H_V-(\lambda+i\epsilon(\lambda))^2)u\|_{L^2 (M)}.
\end{align*}
By \eqref{R lambda with laplacian log imp},
\begin{align*}
    \begin{split}
        \|R_\lambda \circ (-\Delta_g-(\lambda+i\epsilon(\lambda))^2)u\|_{L^q (\gamma_i)}&\lesssim \lambda^{\delta(q, \gamma_i)-1} (\epsilon (\lambda))^{-1+\kappa(q, \gamma_i)} (\lambda \epsilon(\lambda)) \|u\|_{L^2 (M)} \\
        &\lesssim \lambda^{\delta(q, \gamma_i)-1} (\epsilon(\lambda))^{-1+\kappa(q, \gamma_i)} \|(H_V-(\lambda+i\epsilon(\lambda))^2)u\|_{L^2 (M)}.
    \end{split}
\end{align*}
Here, we used the spectral theorem in the last inequality. By these two estimates and \eqref{u decomposition log imp}, it suffices to show that
\begin{align}\label{T lambda perturb est for log imp n=2}
    \|T_\lambda (Vu)\|_{L^q (\gamma_i)}\leq C_V \lambda^{\delta(q, \gamma_i)-1}(\epsilon(\lambda))^{-1+\kappa(q, \gamma_i)} \|(H_V-(\lambda+i\epsilon(\lambda))^2)u\|_{L^2 (M)}.
\end{align}
Since
\begin{align*}
    T_\lambda (Vu)=T_\lambda^0 (Vu)+T_\lambda^1 (Vu),
\end{align*}
we will compute the $T_\lambda^0$ part and $T_\lambda^1$ part separately, and combine them at the end. By the triangle inequality and Minkowski's integral inequality as in \eqref{Triangle Minkowski argument}, we have
\begin{align}\label{T lambda 0 (Vu) initial inequality}
    \|T_\lambda^0 (Vu)\|_{L^q (\gamma_i)}\leq \sup_y \left(\int|T_\lambda(\gamma(r), y)|^q\:dy \right)^{\frac{1}{q}} \|u\|_{L^\infty (M)} \|V\|_{L^1 (M)}.
\end{align}
Using \eqref{T lambda 0 kernel est}, by the proof of Proposition \ref{Prop Reduction for curves Kato case}, we have that
\begin{align*}
    \sup_y \left(\int |T_\lambda^0 (\gamma(r), y)|^q\:dr \right)^{\frac{1}{q}} \lesssim \begin{cases}
    \lambda^{-\frac{1}{q}}, & \text{if } 2<q<\infty, \\
    \lambda^{-1} (\log \lambda)^{\frac{1}{2}}, & \text{if } q=2,
    \end{cases}
\end{align*}
and thus, by \eqref{Log imp est for infty} and \eqref{T lambda 0 (Vu) initial inequality},
\begin{align*}
    \|T_\lambda^0 (Vu)\|\leq \begin{cases}
        C_V \lambda^{-\frac{1}{2}-\frac{1}{q}} (\epsilon(\lambda))^{-\frac{1}{2}} \|(H_V-(\lambda+i\epsilon(\lambda))^2)u\|_{L^2 (M)}, & \text{if } 2<q<\infty, \\
        C_V \lambda^{-\frac{3}{2}} (\epsilon(\lambda))^{-1} \|(H_V-(\lambda+i\epsilon(\lambda))^2)u\|_{L^2 (M)}, & \text{if } q=2,
    \end{cases}
\end{align*}
which satisfies a better (or the same) estimate than the bound posited in \eqref{T lambda perturb est for log imp n=2}.

For $T_\lambda^1 (Vu)$, by \eqref{Log imp est for infty} and \eqref{T lambda L 1 to L infty}, we have that
\begin{align*}
    \|T_\lambda^1 (Vu)\|_{L^q(\gamma_i)} &=\left(\int_{\gamma_i} |T_\lambda^1 (Vu)|^q\:dr \right)^{\frac{1}{q}} \\
    &\lesssim \left(\int_{\gamma_i} (\lambda^{-\frac{1}{2}+\epsilon_0})^q \|Vu\|_{L^1 (M)}^q\:dr \right)^{\frac{1}{q}} \\
    &\lesssim \lambda^{-\frac{1}{2}+\epsilon_0} \|V\|_{L^1 (M)} \|u\|_{L^\infty (M)}\leq C_V \lambda^{-\frac{1}{2}+\epsilon_0} \|u\|_{L^\infty (M)} \\
    &\leq C_V \lambda^{-1+\epsilon_0} (\epsilon(\lambda))^{-\frac{1}{2}} \|(H_V-(\lambda+i\epsilon(\lambda))^2)u\|_{L^2 (M)}.
\end{align*}

Putting these together yields
\begin{align*}
    \|T_\lambda (Vu)\|_{L^q (\gamma_i)}\leq \begin{cases}
    C_V \lambda^{-\frac{1}{2}-\frac{1}{q}} (\epsilon(\lambda))^{-\frac{1}{2}} \|(H_V-(\lambda+i\epsilon(\lambda))^2)u\|_{L^2 (M)},& \text{if } 2<q<\infty, \\
    C_V \lambda^{-1+\epsilon_0} (\epsilon(\lambda))^{-\frac{1}{2}} \|(H_V-(\lambda+i\epsilon(\lambda))^2)u\|_{L^2 (M)}, & \text{if } q=2,
    \end{cases}
\end{align*}
when $\epsilon_0>0$ is sufficiently small. These estimates satisfy \eqref{T lambda perturb est for log imp n=2} for $q<\infty$ as in \eqref{n=2 lambda exponent classifications} and \eqref{n=2 log lambda exponent classifications}, completing the proof of \eqref{Estimate for log imp}.

\subsection{Hypersurfaces and codimension 2 submanifolds}\label{SS: general hypsurf and codim 2 cases}
In this subsection, we show \eqref{Estimate for log imp submflds for 4/3} and \eqref{Estimate for log imp submflds}. Let $P=\sqrt{-\Delta_g}$, $\dim M=3$ or $4$, and $\Sigma$ be a hypersurface or codimension $2$ submanifold. As before, for interested $(q, k)$ in this subsection, by \cite{Chen2015improvement}, if $\epsilon(\lambda)=(\log (2+\lambda))^{-1}$, then
\begin{align*}
    \|\mathds{1}_{[\lambda, \lambda+\epsilon(\lambda)]} (P)\|_{L^2 (M)\to L^q (\Sigma)}\lesssim \lambda^{\delta(q, k)} (\epsilon(\lambda))^{\frac{1}{2}},
\end{align*}
which in turn implies that, by Lemma \ref{Lemma Resolvent opr norm from L2 to Lq for univ est},
\begin{align*}
    \|(-\Delta_g-(\lambda+i \epsilon(\lambda))^2)^{-1}\|_{L^2 (M) \to L^q (\Sigma)}\lesssim \lambda^{\delta(q, k)-1}(\epsilon(\lambda))^{-\frac{1}{2}}.
\end{align*}
If we set $T_\lambda, T_\lambda^0, T_\lambda^1$, and $R_\lambda$ as in \eqref{T lambda R lambda set up}, then by the same arguments as in \eqref{R lambda log imp}-\eqref{T lambda log imp est}, we have
\begin{align*}
    \|R_\lambda \|_{L^2 (M)\to L^q (\Sigma)},\; \|T_\lambda \|_{L^2 (M)\to L^q (\Sigma)} \lesssim \lambda^{\delta(q, k)-1}(\epsilon(\lambda))^{-\frac{1}{2}},
\end{align*}
and
\begin{align*}
    \|R_\lambda \circ (-\Delta_g-(\lambda+i\epsilon(\lambda))^2) u\|_{L^q (\Sigma)} &\lesssim \lambda^{\delta(q, k)-1} (\epsilon(\lambda))^{-\frac{1}{2}} (\lambda \epsilon(\lambda)) \|u\|_{L^2 (M)} \\
    &\lesssim \lambda^{\delta(q, k)-1} (\epsilon(\lambda))^{-\frac{1}{2}} \|(H_V-(\lambda+i\epsilon(\lambda))^2)u\|_{L^2 (M)}.
\end{align*}
In the last inequality, we used the spectral theorem. With this in mind, since
\begin{align*}
    u&=(-\Delta_g-(\lambda+i\epsilon(\lambda))^2)^{-1}\circ (-\Delta_g -(\lambda+i\epsilon(\lambda))^2)u \\
    &=T_\lambda (H_V-(\lambda+i\epsilon(\lambda))^2)u+R_\lambda (-\Delta_g-(\lambda+i\epsilon(\lambda))^2)u-T_\lambda (Vu),
\end{align*}
we would have \eqref{Estimate for log imp submflds for 4/3} and \eqref{Estimate for log imp submflds}, if we could show that
\begin{align}\label{log imp submflds WTS}
    \begin{split}
        \|T_\lambda (Vu)\|_{L^q (\Sigma)}\leq \begin{cases}
            C_V \lambda^{\delta(q, k)-1}(\epsilon(\lambda))^{-\frac{1}{2}}\|(H_V-(\lambda+i\epsilon(\lambda))^2)u\|_{L^2 (M)}, & \text{if } (n, k, q, V) \text{ is as in \eqref{(n, k, q, V) for n=3 with 4/3}}, \\
            C_V \lambda^{\delta(q, k)-1}(\epsilon(\lambda))^{-\frac{1}{2}}\|(H_V-(\lambda+i\epsilon(\lambda))^2)u\|_{L^2 (M)}, & \text{if } (n, k, q, V) \text{ is as in \eqref{(n, k, q, V) except 4/3}},
        \end{cases}
    \end{split}
\end{align}
Since $T_\lambda (Vu)=T_\lambda^0 (Vu)+T_\lambda^1 (Vu)$ as in \eqref{T lambda R lambda set up}, we compute $T_\lambda^0 (Vu)$ and $T_\lambda^1 (Vu)$, separately. We note that $T_\lambda^0$ is a ``local'' operator as in the local operator $S_\lambda$ in \eqref{S lambda W lambda set up}. By the proof of Proposition \ref{Prop Reduction for hypersurface thm} and Proposition \ref{Prop Reduction for codim 2 thm}, if $n\in \{3, 4\}$, and
\begin{align*}
    \frac{n}{p}-\frac{n-1}{q}=2,\quad k=n-1,\quad \frac{2n}{n-1}\leq q< \frac{2(n-1)}{n-3},
\end{align*}
or
\begin{align*}
    \frac{n}{p}-\frac{n-2}{q}=2,\quad k=n-2,\quad \frac{2(n-2)^2}{n^2-5n+8}< q< \frac{2(n-2)}{n-3},\quad q\geq 2,
\end{align*}
we have
\begin{align*}
    \|T_\lambda^0 (Vu)\|_{L^q (\Sigma)}\lesssim \|Vu\|_{L^p (M)}.
\end{align*}
By the argument in \eqref{S(vu) est computation hypersurface}, it follows from \cite[Theorem 1.3]{BlairHuangSireSogge2022UniformSobolev} that
\begin{align*}
    \|T_\lambda^0 (Vu)\|_{L^q (\Sigma)}&\lesssim \|V\|_{L^{\frac{n}{2}} (M)}\|u\|_{L^{\frac{np}{n-2p}}(M)} \\
    &\leq C_V \lambda^{\sigma\left(\frac{np}{n-2p} \right)-1} (\epsilon(\lambda))^{-\frac{1}{2}} \|(H_V-(\lambda+i\epsilon(\lambda))^2)u\|_{L^2 (M)} \\
    &=C_V \lambda^{\frac{n-1}{2}-\frac{k}{q}-1} (\epsilon(\lambda))^{-\frac{1}{2}} \|(H_V-(\lambda+i\epsilon(\lambda))^2)u\|_{L^2 (M)} \\
    &=C_V \lambda^{\delta(q, k)-1} (\epsilon(\lambda))^{-\frac{1}{2}} \|(H_V-(\lambda+i\epsilon(\lambda))^2)u\|_{L^2 (M)},
\end{align*}
as desired. For $T_\lambda^1 (Vu)$, recall from \cite[(3.25)]{BlairHuangSireSogge2022UniformSobolev} (cf. \cite{Berard1977onthewaveequation}) that
\begin{align*}
    \|T_\lambda^1 f\|_{L^\infty (\Sigma)} \leq \|T_\lambda^1 f\|_{L^\infty (M)}\lesssim \lambda^{\frac{n-3}{2}} \lambda^{C c_0} \|f\|_{L^1 (M)},
\end{align*}
for a sufficiently small $0<c_0 \ll 1$. This gives us that
\begin{align}\label{T lambda 1 n case}
    \begin{split}
        \|T_\lambda^1 (Vu)\|_{L^q (\Sigma)}&=\left(\int_\Sigma |(T_\lambda^1 (Vu))(z)|^q\:dz \right)^{\frac{1}{q}} \\
        &\lesssim \lambda^{\frac{n-3}{2}+C c_0} \left(\int_\Sigma \|Vu\|_{L^1 (M)}^q\:dz\right)^{\frac{1}{q}} \\
        &\lesssim \lambda^{\frac{n-3}{2}+C c_0} \|Vu\|_{L^1 (M)}.
    \end{split}
\end{align}

Suppose the condition \eqref{(n, k, q, V) for n=3 with 4/3} holds. Note that $V\in L^{\frac{4}{3}} (M)$ since $V\in L^{\frac{3}{2}} (M)$ and $M$ is compact. By \eqref{T lambda 1 n case}, H\"older's inequality, and \cite[Theorem 1.3]{BlairHuangSireSogge2022UniformSobolev}, taking $0<c_0 \ll 1$, if $\alpha(q, n)$ is as in \eqref{log lambda exponent from L2(M) to Lq(M)}, we have
\begin{align*}
    \|T_\lambda^1 (Vu)\|_{L^q (\Sigma)}&\lesssim \lambda^{Cc_0} \|V\|_{L^{\frac{4}{3}} (M)} \|u\|_{L^4 (M)} \\
    &\leq C_V \lambda^{\sigma(4)-1+Cc_0} (\epsilon(\lambda))^{-1+\alpha(q, n)} \|(H_V-(\lambda+i\epsilon(\lambda))^2)u\|_{L^2 (M)} \\
    &\leq C_V \lambda^{-\frac{3}{4}+Cc_0} (\epsilon(\lambda))^{-1+\alpha(q, n)} \|(H_V-(\lambda+i\epsilon (\lambda))^2)u\|_{L^2 (M)} \\
    &\leq C_V \lambda^{\delta(q, k)-1} (\epsilon(\lambda))^{-\frac{1}{2}} \|(H_V-(\lambda+i\epsilon(\lambda))^2)u\|_{L^2 (M)},
\end{align*}
which satisfies the first estimate in \eqref{log imp submflds WTS}. Similarly, if \eqref{(n, k, q, V) except 4/3} holds for $n=3$, i.e.,
\begin{align*}
    \begin{cases}
        (n, k)=(3, 2), \; 4<q<\infty, \text{ or} \\
        (n, k)=(3, 1), \; 4<q<\infty,
    \end{cases}
\end{align*}
then $V\in L^{q'}(M)$ by $V\in L^{\frac{3}{2}} (M)$ and compactness of $M$. By \eqref{T lambda 1 n case}, H\"older's inequality, and \cite[Theorem 1.3]{BlairHuangSireSogge2022UniformSobolev} (note that $\alpha(q, n)=\frac{1}{2}$ when \eqref{(n, k, q, V) except 4/3} holds for $n=3$),
\begin{align*}
    \|T_\lambda^1 (Vu)\|_{L^q (\Sigma)}&\lesssim \lambda^{Cc_0} \|V\|_{L^{q'} (M)}\|u\|_{L^q (M)} \\
    &\leq C_V \lambda^{Cc_0 +\sigma(q)-1}(\epsilon(\lambda))^{-\frac{1}{2}} \|(H_V-(\lambda+i\epsilon(\lambda))^2)u\|_{L^2 (M)} \\
    &\leq C_V \lambda^{\delta(q, k)-1}(\epsilon(\lambda))^{-\frac{1}{2}}\|(H_V-(\lambda+i\epsilon(\lambda))^2)u\|_{L^2 (M)},
\end{align*}
which satisfies the second estimate in \eqref{log imp submflds WTS}. If \eqref{(n, k, q, V) except 4/3} holds for $n=4$, i.e.,
\begin{align*}
    \begin{cases}
        (n, k)=(4, 3),\; 3<q<6, \text{ or} \\
        (n, k)=(4, 2),\; 2<q<4,
    \end{cases}
\end{align*}
then by \eqref{T lambda 1 n case} and H\"older's inequality,
\begin{align*}
    \|T_\lambda^1 (Vu)\|_{L^q (\Sigma)}&\lesssim \lambda^{\frac{1}{2}+Cc_0}\|V\|_{L^2 (M)} \|u\|_{L^2 (M)} \\
    &\leq C_V \lambda^{\delta(q, k)-\epsilon_0} \|u\|_{L^2 (M)} \\
    &\leq C_V \lambda^{\delta(q, k)-\epsilon_0} (\lambda \epsilon(\lambda))^{-1} (\lambda \epsilon(\lambda))\|u\|_{L^2 (M)} \\
    &\lesssim C_V \lambda^{\delta(q, k)-1-\epsilon_0} (\epsilon(\lambda))^{-1} \|(H_V-(\lambda+i\epsilon(\lambda))^2)u\|_{L^2 (M)},
\end{align*}
where we choose $0<c_0, \epsilon_0\ll 1$ sufficiently small. In the last inequality, we used the spectral theorem. This is better than the bound posited in \eqref{log imp submflds WTS}. This proves \eqref{log imp submflds WTS} when \eqref{(n, k, q, V) except 4/3} holds, which completes the proof of \eqref{Estimate for log imp submflds for 4/3} and \eqref{Estimate for log imp submflds}.

\section{Proof of Theorem \ref{Thm for tori}}\label{S: Thm for tori}
\subsection{General curve segments}\label{SS: general curve in 2D torus}

In this subsection, we show \eqref{2D torus est for general curve}. Let $P=\sqrt{-\Delta_{\mathbb{T}^2}}$ and $\gamma$ be any curve segment in $\mathbb{T}^2$. Recall that, for all previous results, we needed a spectral projection bounds for $\sqrt{-\Delta_g}$. To use our previous arguments, we then need a spectral projection bound for $\mathbb{T}^2$ first.

\begin{lemma}\label{Lemma 2D torus projection for general curve}
    If $\delta(q, 1)=\frac{1}{2}-\frac{1}{q}$ for $q>4$, then
    \begin{align*}
        \|\mathds{1}_{[\lambda, \lambda+\frac{1}{T}]} (P) f\|_{L^q (\gamma)}\lesssim (T^{-\frac{1}{2}} \lambda^{\delta(q, 1)}+T^{\frac{1}{4}}\lambda^{\frac{1}{4}})\|f\|_{L^2 (\mathbb{T}^2)},\quad 1\leq T\leq \lambda.
    \end{align*}
\end{lemma}

\begin{proof}
    Let $\chi\in \mathcal{S}(\mathbb{R})$ be even, nonnegative, and
    \begin{align*}
        \chi(0)=1,\quad \mathrm{supp} (\widehat{\chi})\subset (-\epsilon_0, \epsilon_0)\; \text{ for }\; 0<\epsilon_0 \ll 1.
    \end{align*}
    Since the operator $\chi (T(\lambda-P))$ is invertible on the range of the spectral projector $\mathds{1}_{[\lambda, \lambda+\frac{1}{T}]} (P)$ and
    \begin{align*}
        \| \chi(T(\lambda-P))^{-1} \circ \mathds{1}_{[\lambda, \lambda+\frac{1}{T}]} (P) \|_{L^2 (M)\to L^2 (M)} \lesssim 1,
    \end{align*}
    and so, it suffices to show that
    \begin{align*}
        \|\chi(T(\lambda-P)) f\|_{L^q(\gamma)}\lesssim \left(\frac{\lambda^{\delta(q, 1)}}{T^{\frac{1}{2}}}+(T\lambda)^{\frac{1}{4}} \right)\|f\|_{L^2 (\mathbb{T}^2)},\quad \lambda^{-1}\leq T^{-1}\leq 1.
    \end{align*}
    By a $TT^*$ argument, this is equivalent to saying that
    \begin{align}\label{2D torus proj bound TT* WTS}
        \|\chi^2 (T(\lambda-P))f\|_{L^q (\gamma)}\lesssim \left(\frac{\lambda^{2\delta(q, 1)}}{T}+(T\lambda)^{\frac{1}{2}} \right)\|f\|_{L^{q'}(\gamma)}.
    \end{align}
    By Euler's formula,
    \begin{align*}
        \chi^2 (T(\lambda-P))f&=\frac{1}{2\pi}\int_{-\infty}^\infty e^{itT(\lambda-P)} \widehat{\chi^2} (t) f\:dt \\
        &=\frac{1}{2\pi T}\int_{-\infty}^\infty e^{it\lambda} (e^{-itP} f) \widehat{\chi^2}(t/T)\:dt \\
        &=\frac{1}{\pi T} \int_{-\infty}^\infty e^{it\lambda} \widehat{\chi^2}(t/T) (\cos tP)f\:dt-\chi^2 (T(\lambda+P))f.
    \end{align*}
    If $\eta\in C_0^\infty (\mathbb{R})$ is a cutoff function supported near the origin, since the contribution of $\chi^2 (T(\lambda+P))$ is negligible, modulo $O(\lambda^{-N})$ errors, it suffices to consider $S_0 f+S_1 f$, where
    \begin{align*}
        & S_0 f=\frac{1}{\pi T} \int_{-\infty}^\infty e^{it\lambda} \eta(t) \widehat{\chi^2}(t/T) (\cos tP)f \:dt, \\
        & S_1 f=\frac{1}{\pi T}\int_{-\infty}^\infty e^{it\lambda} (1-\eta(t)) \widehat{\chi^2}(t/T) (\cos tP)f\:dt.
    \end{align*}
    By (the proof of) \cite[\S 3]{BurqGerardTzvetkov2007restrictions} (see also \cite[\S 3.1]{Hu2009lp}), we have that
    \begin{align*}
        \|S_0 f\|_{L^q (\gamma)}\lesssim \frac{\lambda^{2\delta(q, 1)}}{T} \|f\|_{L^{q'}(\gamma)},
    \end{align*}
    which satisfies \eqref{2D torus proj bound TT* WTS}, and hence, it suffices to show that
    \begin{align}\label{2D torus S1 est for general curve}
        \|S_1 f\|_{L^q (\gamma)}\lesssim (T\lambda)^{\frac{1}{2}} \|f\|_{L^{q'}(\gamma)}.
    \end{align}
    By the choice of $\eta\in C_0^\infty (\mathbb{R})$, we have
    \begin{align*}
        & \left|\frac{1}{\pi T}\int_{-\infty}^\infty e^{it\lambda} (1-\eta(t)) \widehat{\chi^2}(t/T) (\cos tP)(x, y)\:dt \right| \\
        &\lesssim \left|\frac{1}{\pi T} \int_{-\infty}^\infty e^{it\lambda} \widehat{\chi^2}(t/T) \sum_{l\in \mathbb{Z}^2} (\cos t\sqrt{-\Delta_{\mathbb{R}^2}})(x-(y+l))\:dt \right| \\
        &\lesssim \frac{1}{T}\cdot \lambda^{\frac{1}{2}} \sum_{1\leq |x-(y+l)|\leq T,\; l\in \mathbb{Z}^2} |x-(y+l)|^{-\frac{1}{2}} \\
        &\lesssim (T\lambda)^{\frac{1}{2}}.
    \end{align*}
    Here, we lifted our computation to the universal cover $\mathbb{R}^2$ by the usual lifting argument or the classical Poisson summation formula, and we used \cite[(3.5.15)]{Sogge2014hangzhou} to get the second inequality (see also \cite[(6.7)]{BlairHuangSogge2022Improved}). The desired inequality \eqref{2D torus S1 est for general curve} then follows from Young's inequality.
\end{proof}
We first note that the $q=\infty$ case was already studied in \cite{BlairHuangSireSogge2022UniformSobolev}. Indeed, by \cite[Theorem 5.2]{BlairHuangSireSogge2022UniformSobolev}, we have, for $u\in \mathrm{Dom}(H_V)\cap C(M)$,
\begin{align*}
    \|u\|_{L^\infty (\gamma)}\leq \|u\|_{L^\infty (\mathbb{T}^2)}\leq C_V \lambda^{-\frac{1}{3}} \|(H_V-(\lambda+i\lambda^{-\frac{1}{3}})^2)u\|_{L^2 (\mathbb{T}^2)},
\end{align*}
which proves \eqref{2D torus est for general curve} for $q=\infty$.

Thus, we may assume that $4<q<\infty$. To do so, we take $\frac{1}{T}=\lambda^{-\frac{1}{3}+\frac{4}{3q}}$ so that $T^{-\frac{1}{2}} \lambda^{\frac{1}{2}-\frac{1}{q}}=T^{\frac{1}{4}}\lambda^{\frac{1}{4}}$ in Lemma \ref{Lemma 2D torus projection for general curve}, and thus,
\begin{align*}
    \| \mathds{1}_{[\lambda, \lambda+\frac{1}{T}]} (P)f\|_{L^q (\gamma)}\lesssim \lambda^{\delta(q, 1)} T^{-\frac{1}{2}} \|f\|_{L^2 (\mathbb{T}^2)},\quad \frac{1}{T}=\lambda^{-\frac{1}{3}+\frac{4}{3q}}.
\end{align*}
Let $\epsilon(\lambda)=\lambda^{-\frac{1}{3}+\frac{4}{3q}}$ as in \eqref{Epsilon setup for curves on 2D tori}. When $V\equiv 0$, by Lemma \ref{Lemma 2D torus projection for general curve} and Lemma \ref{Lemma Resolvent opr norm from L2 to Lq for univ est},
\begin{align*}
    \|(-\Delta_{\mathbb{T}^2}-(\lambda+i\epsilon(\lambda))^2)^{-1}\|_{L^2 (\mathbb{T}^2) \to L^q (\gamma)} \lesssim \lambda^{\delta(q, 1)-1}(\epsilon(\lambda))^{-\frac{1}{2}}.
\end{align*}
Set $T_\lambda, T_\lambda^0, T_\lambda^1$, and $R_\lambda$ as in \eqref{T lambda R lambda set up} so that we can write
\begin{align*}
    (-\Delta_{\mathbb{T}^2}-(\lambda+i\epsilon(\lambda))^2)^{-1}=T_\lambda+R_\lambda,\quad \text{where } T_\lambda=T_\lambda^0+T_\lambda^1.
\end{align*}
As in \cite{BlairHuangSireSogge2022UniformSobolev}, an orthogonality argument gives
\begin{align*}
    & \|R_\lambda \|_{L^2 (\mathbb{T}^2)\to L^q (\gamma)}\lesssim \lambda^{\delta(q, 1)-1}(\epsilon(\lambda))^{-\frac{1}{2}}, \\
    & \|R_\lambda \circ (-\Delta_{\mathbb{T}^2}-(\lambda+i\epsilon(\lambda))^2)\|_{L^2 (\mathbb{T}^2)\to L^q (\gamma)}\lesssim \lambda^{\delta(q, 1)}(\epsilon(\lambda))^{\frac{1}{2}}.
\end{align*}
Since $T_\lambda=(-\Delta_{\mathbb{T}^2}-(\lambda+i\epsilon(\lambda))^2)^{-1}-R_\lambda$, we have
\begin{align*}
    \|T_\lambda \|_{L^2 (\mathbb{T}^2)\to L^q (\gamma)}\lesssim \lambda^{\delta(q, 1)-1} (\lambda \epsilon(\lambda))^{-\frac{1}{2}}.
\end{align*}
Putting these all together, we would have \eqref{2D torus est for general curve} if we could show that
\begin{align}\label{T(Vu) est for 2D tori general curve}
    \begin{split}
        \|T_\lambda (Vu)\|_{L^q (\gamma)}\leq C_V \lambda^{\delta(q, 1)-\frac{5}{6}}\|(H_V-(\lambda+i\lambda^{-\frac{1}{3}})^2)u\|_{L^2 (\mathbb{T}^2)},\quad 4<q<\infty.
    \end{split}
\end{align}
Let $T_\lambda^0$ and $T_\lambda^1$ as in \eqref{T lambda R lambda set up}. Recall that, by an argument in \cite[\S 3, \S 5]{BlairHuangSireSogge2022UniformSobolev}, we know that
\begin{align}\label{T lambda 0 pointwise bound for 2D tori general curve}
    |T_\lambda^0 (x, y)|\lesssim\begin{cases}
        |\log (\lambda d_g (x, y)/2)|, & \text{if } d_g (x, y)\leq \lambda^{-1}, \\
        \lambda^{-\frac{1}{2}} (d_g (x, y))^{-\frac{1}{2}}, & \text{if } \lambda^{-1}\leq d_g (x, y)\leq 1,
    \end{cases}
\end{align}
which gives the same bound as in \eqref{S lambda kernel est for n=2} up to some uniform constant, and so, we can use the arguments in \S \ref{S: Proof for curves for universal estimates} here. By the argument in \eqref{Triangle Minkowski argument} and \eqref{y sup S lambda kernel integration on curves estimate} and \cite[Theorem 5.2]{BlairHuangSireSogge2022UniformSobolev}, we have
\begin{align}\label{T lambda 0 (Vu) est for 2D tori general curve}
    \begin{split}
        \|T_\lambda^0 (Vu)\|_{L^q (\gamma)} &\lesssim\|u\|_{L^\infty (\mathbb{T}^2)} \|V\|_{L^1 (\mathbb{T}^2)} \sup_y \left(\int |T_\lambda^0 (\gamma(r), y)|^q\:dr \right)^{\frac{1}{q}} \\
        &\leq C_V \lambda^{-\frac{1}{3}-\frac{1}{q}} \|(H_V-(\lambda+i\lambda^{-\frac{1}{3}})^2)u\|_{L^2 (\mathbb{T}^2)} \\
        &=C_V \lambda^{\delta(q, 1)-\frac{5}{6}} \|(H_V-(\lambda+i\lambda^{-\frac{1}{3}})^2)u\|_{L^2 (\mathbb{T}^2)},
    \end{split}
\end{align}
which satisfies the bound in \eqref{T(Vu) est for 2D tori general curve}. For $T_\lambda^1$, we will further decompose $T_\lambda^1$ as in \cite{BlairHuangSireSogge2022UniformSobolev}. Let $\beta\in C_0^\infty ((1/2, 2))$ be a Littlewood-Paley bumpfunction that satisfies
\begin{align}\label{Littlewodd-Paley bumpfcn setup}
    \begin{split}
        & \sum_{j=-\infty}^\infty \beta(2^{-j} t)=1,\quad \text{for } t>0, \\
        & \beta_0 (t)=1-\sum_{j=1}^\infty \beta(2^{-j} |t|)\in C_0^\infty,
    \end{split}
\end{align}
and thus, $\beta_0 (t)\equiv 1$ for $t>0$ near the origin. We then consider a dyadic decomposition
\begin{align*}
    T_\lambda^1=T_\lambda^{1, 0}+\sum_{j=1}^\infty T_\lambda^{1, j},
\end{align*}
where
\begin{align}\label{T lambda (1, 0), (1, j) set up}
    \begin{split}
        & T_\lambda^{1, 0}=\frac{i}{\lambda+i\epsilon(\lambda)}\int_0^\infty \beta_0 (t) (1-\eta(t)) \eta(t/T) e^{i\lambda t} e^{-\epsilon(\lambda)t} \cos tP\:dt, \\
        & T_\lambda^{1, j}=\frac{i}{\lambda+i\epsilon(\lambda)}\int_0^\infty \beta(2^{-j}t) (1-\eta(t))\eta(t/T) e^{i\lambda t} e^{-\epsilon(\lambda) t} \cos tP\:dt,\quad 1\leq 2^j \lesssim (\epsilon(\lambda))^{-1}.
    \end{split}
\end{align}
Since $T_\lambda^{1, 0}$ plays the same role as the ``local'' operator $T_\lambda^0$, by the same argument in \eqref{T lambda 0 (Vu) est for 2D tori general curve}, we have that
\begin{align*}
    \|T_\lambda^{1, 0} (Vu)\|_{L^q (\gamma)}\leq C_V \lambda^{\delta(q, 1)-\frac{5}{6}} \|(H_V-(\lambda+i\lambda^{-\frac{1}{3}})^2)u\|_{L^2 (\mathbb{T}^2)},
\end{align*}
which satisfies \eqref{T(Vu) est for 2D tori general curve}. For $T_\lambda^{1, j}$, we recall from \cite[(5.33)]{BlairHuangSireSogge2022UniformSobolev} that
\begin{align*}
    \|T_\lambda^{1, j} f\|_{L^\infty (\mathbb{T}^2)}\lesssim 2^{3j/2} \lambda^{-\frac{1}{2}} \|f\|_{L^1 (\mathbb{T}^2)},\quad 2\leq 2^j\lesssim (\epsilon(\lambda))^{-1}=\lambda^{\frac{1}{3}-\frac{4}{3q}}.
\end{align*}
Using this and \cite[Theorem 5.2]{BlairHuangSireSogge2022UniformSobolev} for $\|u\|_{L^\infty (\mathbb{T}^2)}$, we have
\begin{align*}
    \sum_{j=1}^\infty \|T_\lambda^{1, j} (Vu)\|_{L^q (\gamma)}&\leq (\epsilon(\lambda))^{-\frac{3}{2}} \lambda^{-\frac{1}{2}} \|Vu\|_{L^1 (\mathbb{T}^2)} \\
    &\leq\lambda^{-\frac{4}{3q}} \|V\|_{L^1 (\mathbb{T}^2)} \|u\|_{L^\infty (\mathbb{T}^2)} \\
    &\leq C_V \lambda^{-\frac{1}{3}-\frac{4}{3q}} \|(H_V-(\lambda+i\lambda^{-\frac{1}{3}})^2)u\|_{L^2 (\mathbb{T}^2)},
\end{align*}
which is better than the bound posited in \eqref{T(Vu) est for 2D tori general curve}, completing the proof.

\subsection{Geodesic segments}
In this subsection, we show \eqref{2D torus est for geodesic}. As above, we may need a spectral projection bound for the case where $\gamma$ is a geodesic in $\mathbb{T}^2$.

\begin{lemma}\label{Lemma 2D tori geodesic spectral proj bound}
    Let $\gamma$ be a geodesic in $\mathbb{T}^2$. Then
    \begin{align}\label{Spec bounds for geod in tori}
       \|\mathds{1}_{[\lambda, \lambda+\frac{1}{T}]}(\sqrt{-\Delta_{\mathbb{T}^2}}) f\|_{L^q (\gamma)}\lesssim \lambda^{\frac{1}{4}} T^{-\frac{1}{4}}\|f\|_{L^2 (\mathbb{T}^2)},
    \end{align}
    where for all $0<\delta_0\ll 1$
    \begin{align}\label{2D tori q and T conditions}
        1\leq T\leq \lambda^{\frac{1}{2}-\delta_0}\; \text{ and }\; 2\leq q<\frac{8}{3}.
    \end{align}
    The estimate \eqref{Spec bounds for geod in tori} is sharp in the sense that there exist a function $\Psi_\lambda$ and a geodesic $\gamma$ in $\mathbb{S}^1 \times \mathbb{S}^1$ equipped with the product metric $g$ such that
    \begin{align}\label{Sharp estimate for geod in tori}
        \|\mathds{1}_{[\lambda, \lambda+\frac{1}{T}]} (\sqrt{-\Delta_{\mathbb{T}^2}}) \Psi_\lambda\|_{L^q (\gamma)}\gtrsim \lambda^{\frac{1}{4}}T^{-\frac{1}{4}}\|\Psi_\lambda \|_{L^2 (\mathbb{S}^1\times \mathbb{S}^1)},\quad 1\leq T\leq \lambda^{\frac{1}{2}-\delta_0},\quad 2\leq q\leq 4.
    \end{align}
\end{lemma}

\begin{remark}
    We note that \eqref{2D tori q and T conditions} does not contain the endpoint $q=4$, whereas the estimate \eqref{Sharp estimate for geod in tori} contains the endpoint $q=4$, and thus, it would be interesting to extend the estimate \eqref{Spec bounds for geod in tori} to $q=4$.

    Also note that if we choose $T=\log \lambda$ in \eqref{Sharp estimate for geod in tori}, then we have
    \begin{align}\label{log lambda sharp for subcritical exponents}
        \|\mathds{1}_{[\lambda, \lambda+(\log \lambda)^{-1}]} (\sqrt{-\Delta_{\mathbb{T}^2}}) \Psi_\lambda\|_{L^q (\gamma)}\gtrsim \lambda^{\frac{1}{4}} (\log \lambda)^{-\frac{1}{4}}\|\Psi_\lambda \|_{L^2 (\mathbb{S}^1 \times \mathbb{S}^1)},\quad 2\leq q\leq 4.
    \end{align}
    If we identify $\mathbb{S}^1 \times \mathbb{S}^1$ as $\mathbb{T}^2$, then $\mathbb{T}^2$ has zero curvatures, that is, nonpositive sectional curvatures, and thus, the above estimate \eqref{log lambda sharp for subcritical exponents} means the sharpness of the following estimate
    \begin{align}\label{log lambda est for subcritical exp on 2D tori}
         \|\mathds{1}_{[\lambda, \lambda+(\log \lambda)^{-1}]} (\sqrt{-\Delta_{\mathbb{T}^2}}) f\|_{L^q (\gamma)}\lesssim \lambda^{\frac{1}{4}} (\log \lambda)^{-\frac{1}{4}}\|f \|_{L^2 (\mathbb{T}^2)},\quad 2\leq q\leq 4,
    \end{align}
    which is obtained by the interpolation of \cite[Theorem 1.1]{BlairSogge2018concerning} and \cite[Theorem 1.1]{Blair2018logarithmic}. We remark that the estimate \eqref{log lambda est for subcritical exp on 2D tori} for $2\leq q<4$ can also be obtained by the proof of Lemma \ref{Lemma 2D tori geodesic spectral proj bound} below. We also note that \cite[Theorem 2]{XiZhang2017improved} also showed the same bound for negatively curved manifolds, so it would also be interesting if we could find a sharp example for the bound in negatively curved manifolds, since \eqref{Sharp estimate for geod in tori} holds on $\mathbb{S}^1 \times \mathbb{S}^2$, which can be thought of as a manifold with zero sectional curvatures.
\end{remark}

\begin{proof}[Proof of Lemma \ref{Lemma 2D tori geodesic spectral proj bound}]
    Let $\chi\in \mathcal{S}(\mathbb{R})$ be an even function such that $\chi(0)=1$ and $\mathrm{supp}(\widehat{\chi})\subset (-\epsilon_0, \epsilon_0)$ for $0<\epsilon_0 \ll 1$, as usual. Let $P=\sqrt{-\Delta_{\mathbb{T}^2}}$ and $\gamma$ be a geodesic in $\mathbb{T}^2$. We choose the same pseudo-differential cutoff $Q_{\theta, \lambda}$ as in \cite{BlairSogge2018concerning}. That is, when we consider the same local coordinates as in \cite{BlairSogge2018concerning} so that $\gamma$ can be identified as
    \begin{align*}
        \{(t, 0): 0\leq t\leq \epsilon_0\},\quad 0<\epsilon_0\ll 1,
    \end{align*}
    if $\theta=\lambda^{-\delta}$ with $0<\delta<\frac{1}{2}$ and $\chi_1\in C_0^\infty (\mathbb{R})$ satisfies $\chi_1 (s)=1$ for $|s|\leq 1$ and $\chi_1 (s)=0$ for $|s|\geq 2$, we can define the compound symbols
    \begin{align*}
        q_{\theta, \lambda} (x, y, \xi)=\chi_1 (\theta^{-1} d_g (x, \gamma)) \chi_1 (\theta^{-1} d_g (y, \gamma)) \chi_1 (\theta^{-1}|\xi_2|/|\xi|)\Upsilon (|\xi|/\lambda),
    \end{align*}
    where $\chi_1\in C_0^\infty (\mathbb{R})$ is a smooth bump function supported near the origin and $\Upsilon \in C^\infty (\mathbb{R})$ satisfies
    \begin{align*}
        \Upsilon(s)=\begin{cases}
            1, & \text{if } s\in [c_0, c_0^{-1}], \\
            0, & \text{if } s\not\in [c_0/2, 2c_0^{-1}],
        \end{cases}
    \end{align*}
    where $c_0>0$ is sufficiently small. Then the pseudo-differential cutoff $Q_{\theta, \lambda}$ is the operator whose integral kernel $Q_{\theta, \lambda} (x, y)$ is of the form
    \begin{align*}
        Q_{\theta, \lambda} (x, y)=(2\pi)^{-2}\int_{\mathbb{R}^2} e^{i\langle x-y, \xi \rangle} q_{\theta, \lambda}(x, y, \xi)\:d\xi.
    \end{align*}
    Using \cite[Proposition 2.2]{BlairSogge2018concerning} and the argument in \cite[\S 3]{BlairSogge2018concerning}, one can find that
    \begin{align*}
        \|(I-Q_{\theta, \lambda})\circ \chi (T(\lambda-P)) f\|_{L^2 (\gamma)}\lesssim_\epsilon (\lambda^{-\delta})^{-\frac{1}{2}-\epsilon}\|f\|_{L^2 (\mathbb{T}^2)},
    \end{align*}
    where $0<\epsilon<\frac{1}{2}-\delta<\frac{1}{2}$. If we choose $0<\delta\ll 1$ sufficiently small, the bound here is better than or equal to what we need. In fact, we could say more than this. Since $e^{ikP}$ for $k\in \mathbb{N}$ maps $L^2 (M)$ to $L^2 (M)$ with norm $1$, as in \cite{BlairSogge2018concerning} (cf. \cite[p.198]{BlairSogge2018concerning}), we can focus on the operator $S_{\lambda, \theta}$ defined by
    \begin{align*}
        f\mapsto \int a(t) e^{it\lambda} (I-Q_{\theta, \lambda}) e^{-itP}f\:dt,\quad a\in C_0^\infty ((-1, 1)).
    \end{align*}
    If we denote by $K_{\lambda, \theta}$ the integral kernel of the $S_{\lambda, \theta}S_{\lambda, \theta}^*$ operator, then by the same argument as in \cite[\S 5]{BlairSogge2018concerning}, we can show that, modulo $O(\lambda^{-1})$ errors,
    \begin{align*}
        & |K_{\lambda, \theta} (\gamma(s), \gamma(s'))|\lesssim \lambda^{\frac{1}{2}} |s-s'|^{-\frac{1}{2}},\quad \text{if } s, s'\in [0, 1], \\
        & |K_{\lambda, \theta } (\gamma(s), \gamma(s'))|\leq C_N \lambda^{-N} \text{ for all } N,\quad \text{if } s, s'\in [0, 1] \text{ and } |s-s'|\geq \lambda^{-1}\theta^{-2-2\epsilon}
    \end{align*}
    (cf. \cite[(5.10)-(5.11)]{BlairSogge2018concerning}). With this in mind, if $2\leq q<4$ and
    \begin{align*}
        \frac{1}{r}=1-\left(\frac{1}{q'}-\frac{1}{q} \right)=\frac{2}{q}, \quad \text{that is, } r=\frac{q}{2},
    \end{align*}
    then, for $\theta=\lambda^{-\delta}$ with $0<\delta<\frac{1}{2}$,
    \begin{align*}
        &\left(\int |K_{\lambda, \theta} (\gamma(s), \gamma(s'))|^r\:ds \right)^{\frac{1}{r}},\; \left(\int |K_{\lambda, \theta} (\gamma(s), \gamma(s'))|^r\:ds' \right)^{\frac{1}{r}} \\
        &\lesssim \lambda^{\frac{1}{2}} \left(\int_{|s-s'|\leq \lambda^{-1}\theta^{-2-2\epsilon}} |s-s'|^{-\frac{1}{2}\cdot \frac{q}{2}}\:ds \right)^{\frac{2}{q}} \\
        &=\frac{4}{(4-q)^{\frac{2}{q}}}\lambda^{1-\frac{2}{q}} (\theta^{-1-\epsilon})^{\frac{4-q}{q}} \\
        &=\frac{4}{(4-q)^{\frac{2}{q}}} \lambda^{1-\frac{2}{q}+\delta\cdot \frac{(1+\epsilon)(4-q)}{q}},\quad 0<\delta<\frac{1}{2}.
    \end{align*}
    Taking $0<\delta\ll \frac{1}{2}$ sufficiently small, by Young's inequality, when \eqref{2D tori q and T conditions} holds,
    \begin{align}\label{Tori geod escape time result}
        \|S_{\lambda, \theta}S_{\lambda, \theta}^* f\|_{L^q (\gamma)}\leq C_q \lambda^{1-\frac{2}{q}+\epsilon'} \|f\|_{L^{q'}(\gamma)},\quad \text{for all } 0<\epsilon'\ll 1.
    \end{align}
    By a $TT^*$ argument, we have
    \begin{align*}
        \|S_{\lambda, \theta}\|_{L^q (\gamma)}\leq C_q \lambda^{\frac{1}{2}-\frac{1}{q}+\epsilon'}\|f\|_{L^2 (\mathbb{T}^2)},
    \end{align*}
    Using this and the argument in \cite[\S 3]{BlairSogge2018concerning}, we have
    \begin{align*}
        \|(I-Q_{\theta, \lambda})\circ \chi(T(\lambda-P)) f\|_{L^q (\gamma)}\leq C_q \lambda^{\frac{1}{2}-\frac{1}{q}+\epsilon'}\|f\|_{L^2 (\mathbb{T}^2)},\quad 0<\epsilon'\ll 1,
    \end{align*}
    which is better than \eqref{Spec bounds for geod in tori} when \eqref{2D tori q and T conditions} holds, since $\lambda^{\frac{1}{2}-\frac{1}{q}+\epsilon'}<\frac{\lambda^{\frac{1}{4}}}{T^{\frac{1}{4}}}$ for $1\leq T\leq \lambda^{\frac{1}{2}-\delta_0}$. This is a reason why we focus on \eqref{2D tori q and T conditions} for the range of $q$'s and do not focus on high $q$'s.
    
    We would then have \eqref{Spec bounds for geod in tori}, if we could show that
    \begin{align*}
        \|Q_{\theta, \lambda}\circ \chi (T(\lambda-P)) f\|_{L^q (\gamma)}\lesssim \left(\frac{\lambda}{T} \right)^{\frac{1}{4}} \|f\|_{L^2 (\mathbb{T}^2)},
    \end{align*}
    when \eqref{2D tori q and T conditions} holds. By a $TT^*$ argument, it suffices to show that
    \begin{align}\label{TT* WTS for 2D tori geodesic}
        \|Q_{\theta, \lambda}\circ \chi^2 (T(\lambda-P)) \circ Q_{\theta, \lambda}^* f\|_{L^q (\gamma)}\lesssim \left(\frac{\lambda}{T} \right)^{\frac{1}{2}} \|f\|_{L^{q'} (\mathbb{T}^2)}.
    \end{align}
    Recall from \cite[(2.11)]{BlairSogge2018concerning} that
    \begin{align}\label{PDO cutoff uniform L1}
        \sup_x \int |Q_{\theta, \lambda} (x, y)|\:dy\lesssim 1,\quad \sup_y \int |Q_{\theta, \lambda} (x, y)|\:dx\lesssim 1.
    \end{align}
    Let $\beta\in C_0^\infty ((1/2, 2))$ be a Littlewood-Paley bumpfunction  as in \eqref{Littlewodd-Paley bumpfcn setup}. We set the ``local'' operator $L_\lambda$ defined by
    \begin{align*}
        L_\lambda=\frac{1}{2\pi T} \int e^{i\lambda t} e^{-itP} \beta_0 (|t|) \widehat{\chi^2} (t/T)\:dt,
    \end{align*}
    and the ``global'' operator $G_\lambda$ defined by
    \begin{align*}
        G_\lambda=\frac{1}{2\pi T} \int e^{i\lambda t} e^{-itP} (1-\beta_0 (|t|)) \widehat{\chi^2} (t/T)\:dt
    \end{align*}
    so that we have $\chi^2 (T(\lambda-P))=L_\lambda+G_\lambda$.

    We first consider the local operator $L_\lambda$. By the method of stationary phase (and Egorov's theorem), we can write (cf. the proof of \cite[Lemma 3.1]{Park2023Restriction})
    \begin{align}\label{Stat phase for L lambda for geod in 2D tori}
        (Q_{\theta, \lambda}\circ e^{-itP} \circ Q_{\theta, \lambda}^*) (x, y)=\lambda^2 \int e^{i\lambda (\varphi(t, x, \xi)-y\cdot \xi)} a_{\theta, \lambda} (t, x, y, \xi)\:d\xi+O(\lambda^{-N}),\quad \text{for any } N\geq 1,
    \end{align}
    where $a_{\theta, \lambda}\in C_0^\infty$ with the size estimate $|\partial_{t, x, y, \xi}^\alpha a_{\theta, \lambda}|\leq C_\alpha$ and the phase function $\varphi$ satisfies, for small $|t|$,
    \begin{align*}
        & \kappa_t: \mathbb{R}^4\to \mathbb{R}^4 \text{ is the Hamiltonian flow of } p(x, \xi)=|\xi|_{g(x)}, \text{ and homogeneous in } \xi, \\
        & \kappa_t (d_\xi \varphi(t, x, \xi), \xi)=(x, d_x \varphi(t, x, \xi)),\quad \text{with } \kappa_t (y, \xi(0))=(x, \xi(t)), \\
        & \partial_t \varphi+p(x, d_x \varphi)=0,\quad \varphi(0, x, \xi)=\langle x, \xi \rangle.
    \end{align*}
    Here of course, the metric $g$ is the Euclidean metric. Taking $N\geq 1$ large enough, we can ignore the contribution of $O(\lambda^{-N})$ in \eqref{Stat phase for L lambda for geod in 2D tori}, and so, by the proof of \cite[Lemma 5.1.3]{Sogge1993fourier}, we have that, modulo $O(\lambda^{-N})$ errors,
    \begin{align*}
        (Q_{\theta, \lambda}\circ L_\lambda \circ Q_{\theta, \lambda}) (x, y)=\frac{\lambda^{\frac{1}{2}}}{T} e^{i\lambda d_g (x, y)} a_\lambda (x, y),
    \end{align*}
    where $a_\lambda\in C_0^\infty$ satisfies $|\partial_{x, y}^\alpha a_\lambda (x, y)|\leq C_\alpha$. By Young's inequality, we have
    \begin{align*}
        \|(Q_{\theta, \lambda}\circ L_\lambda\circ Q_{\theta, \lambda}^*)f\|_{L^q (\gamma)}\lesssim \frac{\lambda^{\frac{1}{2}}}{T}\|f\|_{L^{q'}(\gamma)},\quad 2\leq q<\frac{8}{3}.
    \end{align*}
    This is better than \eqref{TT* WTS for 2D tori geodesic}, and so, we can focus on the contribution of the global operator $G_\lambda$. 
    
    To show\eqref{TT* WTS for 2D tori geodesic}, we now want to show that
    \begin{align}\label{TT^* G lambda WTS for 2D tori geodesic}
        \|(Q_{\theta, \lambda}\circ G_\lambda\circ Q_{\theta, \lambda}^*)f\|_{L^q (\gamma)}\lesssim \frac{\lambda^{\frac{1}{2}}}{T^{\frac{1}{2}}}\|f\|_{L^{q'}(\gamma)},\quad 2\leq q<\frac{8}{3}.
    \end{align}
    By direct computation with \eqref{PDO cutoff uniform L1},
    \begin{align}\label{TT^* kernel bound reduction for 2D tori geodesic}
        \begin{split}
            |(Q_{\theta, \lambda}\circ G_\lambda \circ Q_{\theta, \lambda}^*)(x, y)| &=\left|\int (Q_{\theta, \lambda}\circ G_\lambda\circ Q_{\theta, \lambda}^*) (x, y)\:dz \right| \\
            &\leq \int |(Q_{\theta, \lambda}\circ G_\lambda) (x, z)| |Q_{\theta, \lambda}^* (z, y)|\:dz \\
            &\leq \sup_{x, z} |(Q_{\theta, \lambda}\circ G_\lambda) (x, z)| \int |Q_{\theta, \lambda}^*(z, y)|\:dz \\
            &\lesssim \sup_{x, z} |(Q_{\theta, \lambda}\circ G_\lambda) (x, z)|.
        \end{split}
    \end{align}
    Recall that the last quantity $|(Q_{\theta, \lambda}\circ G_\lambda) (x, z)|$ is already studied in \cite{BlairHuangSogge2022Improved}. In fact, we shall follow the argument in \cite{BlairHuangSogge2022Improved} in the rest of the computation for \eqref{TT^* G lambda WTS for 2D tori geodesic}. By Young's inequality, we would have \eqref{TT^* G lambda WTS for 2D tori geodesic} if we could show that
    \begin{align}\label{Q G kernel bound for 2D tori geodesic}
        |(Q_{\theta, \lambda}\circ G_\lambda) (x, z)|\lesssim \frac{\lambda^{\frac{1}{2}}}{T^{\frac{1}{2}}},\quad x, z\in \mathbb{T}^2.
    \end{align}
    By Euler's formula, when we set
    \begin{align*}
        \Tilde{G}_\lambda=\frac{1}{\pi T} \int e^{i\lambda t} \cos (t\sqrt{-\Delta_{\mathbb{T}^2}})(1-\beta_0 (|t|))\widehat{\chi^2}(t/T)\:dt,
    \end{align*}
    we would have \eqref{Q G kernel bound for 2D tori geodesic}, if we could show that
    \begin{align}\label{Q G tilde kernel bound for 2D tori geodesic}
        |(Q_{\theta, \lambda}\circ \Tilde{G}_\lambda) (x, z)|\lesssim \frac{\lambda^{\frac{1}{2}}}{T^{\frac{1}{2}}},\quad x, z\in \mathbb{T}^2.
    \end{align}
    Since we set $1-\beta_0 (|t|)=\sum_{j=0}^\infty \beta(2^{-j}|t|)$, if we write
    \begin{align*}
        \Tilde{G}_{\lambda, j}=\frac{1}{\pi T}\int e^{i\lambda t} \cos (t\sqrt{-\Delta_{\mathbb{T}^2}}) \beta(2^{-j}|t|) \widehat{\chi^2}(t/T)\:dt,
    \end{align*}
    then by the finite speed of propagation, we can write
    \begin{align*}
        \Tilde{G}_\lambda=\sum_{1\leq 2^j\lesssim T} \Tilde{G}_{\lambda, j}.
    \end{align*}
    If we lift our computation to the universal cover as usual (or the classical Poisson summation formula), then since the universal cover is $\mathbb{R}^2$ with the usual Euclidean metric, we can write
    \begin{align*}
        (\cos t\sqrt{-\Delta_{\mathbb{T}^2}}) (x, z)=\sum_{l\in \mathbb{Z}^2} (\cos t\sqrt{\Delta_{-\mathbb{R}^2}}) (x-(z+l)),
    \end{align*}
    where the torus $\mathbb{T}^2$ is identified as the cube $Q=(-\pi, \pi]\times (-\pi, \pi]$, and so, we may abuse notation a bit identifying $x\in \mathbb{T}^2$ as $x\in \mathbb{R}^2$.
    
    Before going further, we note that we can restrict our attention to $|x-(z+l)|\approx 2^j$. Indeed, if $\Tilde{G}_{\lambda, j} (x, z)$ denotes the kernel of $\Tilde{G}_{\lambda, j}$, then we can write
    \begin{align*}
        \Tilde{G}_{\lambda, j}(x, z)=\sum_{l\in \mathbb{Z}^2} \frac{1}{\pi T}\int e^{i\lambda t} (\cos t\sqrt{-\Delta_{\mathbb{R}^2}})(x-(z+l)) \beta(2^{-j}|t|) \widehat{\chi^2}(t/T)\:dt.
    \end{align*}
    By finite speed of propagation, this kernel vanishes if $|x-(z+l)|\geq 2\cdot 2^j$ since $\beta(|t|/2^j)=0$ for $|t|\geq 2\cdot 2^j$. If $|t|\leq \frac{2^j}{2}$, then by the singular support properties of $\cos t\sqrt{-\Delta_{\mathbb{R}^2}}$, we could think of this as a smooth function (cf. \cite[(5.14)]{BlairSogge2015OnKakeyaNikodym}), and thus, integrating by parts as many as we want, we obtain $\Tilde{G}_{\lambda, j}(x, z)=O(\lambda^{-N})$ for every $N$ when $|x-(z+l)|\leq \frac{2^j}{4}$. We thus may assume that $|x-(z+l)|\approx 2^j$ for each fixed $l\in \mathbb{Z}^2$.
    
    If we denote by $K_{\lambda, j, \theta} (x, z)$ the kernel of $Q_{\theta, \lambda}\circ \Tilde{G}_{\lambda, j}$, then we can write
    \begin{align*}
        \sum_{l\in \mathbb{Z}^2} K_{\lambda, j, \theta} (x, (z+l))=\sum_{l\in \mathbb{Z}^2} \frac{1}{\pi T}\int e^{i\lambda t} (Q_{\theta, \lambda}\circ \cos t\sqrt{-\Delta_{\mathbb{R}^2}})(x, (z+l)) \beta(2^{-j}|t|) \widehat{\chi^2}(t/T)\:dt.
    \end{align*}
    If we let $\mathds{1}=(1, 0)$, for fixed $x, z$, and $2^j$, we define the following as in \cite[(6.31)-(6.32)]{BlairHuangSogge2022Improved}.
    \begin{align*}
        & D_{\text{main}}=\left\{l\in \mathbb{Z}^2: \left|\pm \frac{x-(z+l)}{|x-(z+l)|}-\mathds{1} \right|\leq C 2^{-j},\quad |x-(z+l)|\approx 2^j \right\}, \\
        & D_{\text{error}}=\left\{l\in \mathbb{Z}^2: \left|\pm \frac{x-(z+l)}{|x-(z+l)|}-\mathds{1} \right|\geq C 2^{-j},\quad |x-(z+l)|\approx 2^j \right\}.
    \end{align*}
    Using integration by parts, stationary phase argument, and the arguments in \cite{BlairSogge2018concerning}, the first author, Huang, and Sogge \cite[(6.37) and (6.39)]{BlairHuangSogge2022Improved} showed the following:
    \begin{proposition}[\cite{BlairHuangSogge2022Improved}]\label{Proposition BHS22 kernel bounds}
        We have
        \begin{align*}
            & \sum_{l\in D_{\text{error}}} |K_{\lambda, j, \theta} (x, z+l)|\leq C_N (\lambda 2^{-j})^{-N} \lambda^{\frac{5}{2}} T^{-1} (2^j)^{\frac{3}{2}}, \quad \text{for all } N\geq 1,\\
            & \sum_{l\in D_{\text{main}}} |K_{\lambda, j, \theta} (x, z+l)|\lesssim \lambda^{\frac{1}{2}} T^{-1} 2^{\frac{j}{2}}.
        \end{align*}
    \end{proposition}
    By Proposition \ref{Proposition BHS22 kernel bounds}, the contribution of $D_{\text{error}}$ is better than \eqref{Q G tilde kernel bound for 2D tori geodesic} since $2^j\lesssim T\leq \lambda^{\frac{1}{2}-\delta_0}$ for $0<\delta_0 \ll 1$. It also follows from Proposition \ref{Proposition BHS22 kernel bounds} that the contribution of $D_{\text{main}}$ is
    \begin{align*}
        \sum_{1\leq 2^j\lesssim T} \lambda^{\frac{1}{2}} T^{-1}2^{\frac{j}{2}}\lesssim \frac{\lambda^{\frac{1}{2}}}{T^{\frac{1}{2}}},
    \end{align*}
    which satisfies \eqref{Q G tilde kernel bound for 2D tori geodesic}. This completes the proof of \eqref{Spec bounds for geod in tori}.

    For sharpness, we follow the argument in \cite[\S 7]{BlairHuangSogge2022Improved}. Since $\mathbb{T}^2 \cong \mathbb{S}^1 \times \mathbb{S}^1$, let $0\leq \beta\in C_0^\infty ((1/2, 2))$ be a Littlewood-Paley bump function, and $P_{\mathbb{S}^1}=\sqrt{-\Delta_{\mathbb{S}^1}}$. If the $\mu_k$ are eigenvalues of $P_{\mathbb{S}^1}$ and $\{e_k\}$ is the associated orthonormal basis, fix $x_0 \in \mathbb{S}^1$ so that $|e_j (x_0)|\approx 1$. If $\beta(P_{\mathbb{S}^1}/(\lambda^{1/2}T^{-1/2}))(x, y)$ denotes the kernel of the operator $\beta(P_{\mathbb{S}^1}/(\lambda^{1/2}T^{-1/2}))$, then we define
    \begin{align*}
        \Psi_\lambda (\theta, x)=(\lambda^{1/2}T^{-\frac{1}{2}})^{-\frac{1}{2}} e^{i\lambda \theta} \beta(P_{\mathbb{S}^1}/(\lambda^{1/2}T^{-1/2})) (x_0, x).
    \end{align*}
    We want to show that this $\Psi_\lambda$ satisfies the bound \eqref{Sharp estimate for geod in tori} where the geodesic $\gamma$ is chosen as $\gamma=\mathbb{S}^1 \times \{x_0\}$. Since the $\mu_k$ are eigenvalues of $P_{\mathbb{S}^1}$ and $\{e_k\}$ is the associated orthonormal basis,
    \begin{align*}
        \beta(P_{\mathbb{S}^1}/(\lambda T^{-1})^{\frac{1}{2}})(x, y)=\sum_{j=1}^\infty \beta(\mu_j/(\lambda T^{-1})^{\frac{1}{2}}) e_j (x) \overline{e_j (y)}.
    \end{align*}
    This implies that, for $2\leq q\leq 4$,
    \begin{align*}
        \|\Psi_\lambda\|_{L^q (\gamma)}^q&=\int_0^{2\pi} |\Psi_\lambda (\theta, x_0)|^q\:d\theta \\
        &=(\lambda^{1/2}T^{-1/2})^{-\frac{q}{2}} \int_0^{2\pi} |\beta(P_{\mathbb{S}_1}/(\lambda^{1/2} T^{-1/2}))(x_0, x_0)|^q\:d\theta \\
        &\approx (\lambda^{1/2}T^{-1/2})^{-\frac{q}{2}} \left(\sum_j \beta(\mu_j/(\lambda^{1/2}T^{-1/2})) |e_j (x_0)|^2  \right)^q \\
        &\approx (\lambda^{1/2}T^{-1/2})^{-\frac{q}{2}} \left(\#\{\mu_j: \mu_j \approx (\lambda^{1/2}T^{-1/2})\} \right)^q \\
        &\approx (\lambda^{1/2}T^{-1/2})^{-\frac{q}{2}} \left((\lambda^{1/2}T^{-1/2})^{2-1} \right)^q \\
        &=(\lambda^{1/2}T^{-1/2})^{\frac{q}{2}}.
    \end{align*}
    Since it is known that $\|\Psi_\lambda \|_{L^2 (\mathbb{S}^1 \times \mathbb{S}^1)}\approx 1$ (cf. \cite[(7.10)]{BlairHuangSogge2022Improved}, \cite[\S 4.3]{Sogge1993fourier}), we have that
    \begin{align*}
        \frac{\|\Psi_\lambda \|_{L^q (\gamma)}}{\|\Psi_\lambda\|_{L^2 (\mathbb{S}^1 \times \mathbb{S}^1)}} \approx \lambda^{\frac{1}{4}} T^{-\frac{1}{4}}.
    \end{align*}
    This proves \eqref{Sharp estimate for geod in tori}.
\end{proof}

We now come back to the proof of \eqref{2D torus est for geodesic}. By Lemma \ref{Lemma 2D tori geodesic spectral proj bound} and Lemma \ref{Lemma Resolvent opr norm from L2 to Lq for univ est}, when $V\equiv 0$ and $\epsilon(\lambda)$ is as in \eqref{Window for tori geodesic cases},
\begin{align*}
    \|(-\Delta_{\mathbb{T}^2}-(\lambda+i\epsilon(\lambda))^2)^{-1}\|_{L^2 (\mathbb{T}^2)\to L^q (\gamma)}\lesssim \lambda^{-\frac{3}{4}} (\epsilon(\lambda))^{-\frac{3}{4}}, \quad \text{if }\; 2\leq q<\frac{8}{3}.
\end{align*}
With this in mind, if we set $T_\lambda, T_\lambda^0, T_\lambda^1$, and $R_\lambda$ as in \eqref{T lambda R lambda set up} so that
\begin{align*}
    (-\Delta_{\mathbb{T}^2}-(\lambda+i\epsilon(\lambda))^2)^{-1}=T_\lambda+R_\lambda=T_\lambda^0+T_\lambda^1+R_\lambda,
\end{align*}
we have, as in the previous subsection,
\begin{align*}
    & \|R_\lambda \|_{L^2 (\mathbb{T}^2)\to L^q (\gamma)},\; \|T_\lambda \|_{L^2 (\mathbb{T}^2)\to L^q (\gamma)}\lesssim \lambda^{-\frac{3}{4}} (\epsilon(\lambda))^{-\frac{3}{4}}, \quad \text{if }\; 2\leq q<\frac{8}{3}, \\
    & \|R_\lambda \circ (-\Delta_{\mathbb{T}^2}-(\lambda+i\epsilon(\lambda))^2)\|_{L^2 (\mathbb{T}^2)\to L^q (\gamma)}\lesssim (\lambda \epsilon(\lambda)) \|R_\lambda\|_{L^2 (\mathbb{T}^2)\to L^q (\gamma)},
\end{align*}
and thus, we would have \eqref{2D torus est for geodesic}, if we could show that
\begin{align}\label{T lambda (Vu) est for 2D tori geodesic}
    \|T_\lambda (Vu)\|_{L^q (\gamma)}\leq C_V \lambda^{-\frac{5}{6}} (\epsilon(\lambda))^{-\frac{3}{2}} \|(H_V-(\lambda+i\lambda^{-\frac{1}{3}})^2)\|_{L^2 (\mathbb{T}^2)}.
\end{align}
Since the argument in \eqref{T lambda 0 (Vu) est for 2D tori general curve} holds for any curve segment, it should also hold for geodesics, and thus, we have
\begin{align*}
    \begin{split}
        \|T_\lambda^0 (Vu)\|_{L^q (\gamma)} &\lesssim\|u\|_{L^\infty (\mathbb{T}^2)} \|V\|_{L^1 (\mathbb{T}^2)} \sup_y \left(\int |T_\lambda^0 (\gamma(r), y)|^q\:dr \right)^{\frac{1}{q}} \\
        &\leq C_V \lambda^{-\frac{1}{3}-\frac{1}{q}} \|(H_V-(\lambda+i\lambda^{-\frac{1}{3}})^2)u\|_{L^2 (\mathbb{T}^2)},
    \end{split}
\end{align*}
and this is better than the bound posited in \eqref{T lambda (Vu) est for 2D tori geodesic}. It then suffices to find the desired estimate for $T_\lambda^1 (Vu)$. As before, we set $T_\lambda^{1, 0}$ and $T_\lambda^{1, j}$ as in \eqref{T lambda (1, 0), (1, j) set up}. Since the operator $T_\lambda^{1, 0}$ is also a ``local'' operator as $T_\lambda^0$, by the same argument, we have
\begin{align*}
    \|T_\lambda^{1, 0} (Vu)\|_{L^q (\gamma)} \leq C_V \lambda^{-\frac{1}{3}-\frac{1}{q}} \|(H_V-(\lambda+i\lambda^{-\frac{1}{3}})^2)u\|_{L^2 (\mathbb{T}^2)},
\end{align*}
which is better than we need, and thus, we can focus on the operators $T_\lambda^{1, j}$. By the proof of \cite[(4.21)]{BlairHuangSireSogge2022UniformSobolev}, we have
\begin{align*}
    |T_\lambda^{1, j} (x, y)|\lesssim \lambda^{-1}\cdot \lambda^{\frac{1}{2}} 2^{\frac{3}{2}j}=\lambda^{-\frac{1}{2}} 2^{\frac{3}{2}j},\quad x, y\in \mathbb{T}^2,\quad 2\leq 2^j\lesssim (\epsilon(\lambda))^{-1},
\end{align*}
where $T_\lambda^{1, j} (x, y)$ is the kernel of the operator $T_\lambda^{1, j}$. Using the argument as above (cf. \eqref{T lambda 0 (Vu) est for 2D tori general curve}), it then follows from \cite[Theorem 5.2]{BlairHuangSireSogge2022UniformSobolev} that
\begin{align*}
    \sum_{j=1}^\infty \|T_\lambda^{1, j} (Vu)\|_{L^q (\gamma)}&\leq \|u\|_{L^\infty (\mathbb{T}^2)} \|V\|_{L^1 (\mathbb{T}^2)} \sum_{2\leq 2^j\lesssim (\epsilon(\lambda))^{-1} } \sup_y \left(\int |T_\lambda^{1, j} (\gamma(r), y)|^2\:dr \right)^{\frac{1}{2}} \\
    &\leq \|u\|_{L^\infty (\mathbb{T}^2)} \|V\|_{L^1 (\mathbb{T}^2)} \sum_{2\leq 2^j\lesssim (\epsilon(\lambda))^{-1} } \lambda^{-\frac{1}{2}} 2^{\frac{3}{2}j} \\
    &\leq \left(\lambda^{-\frac{1}{3}} \|(H_V-(\lambda+i\lambda^{-\frac{1}{3}})^2)u\|_{L^2 (\mathbb{T}^2)} \right)\|V\|_{L^1 (\mathbb{T}^2)} \left(\lambda^{-\frac{1}{2}}(\epsilon(\lambda))^{-\frac{3}{2}}\right)  \\
    &=C_V \lambda^{-\frac{5}{6}} (\epsilon(\lambda))^{-\frac{3}{2}} \|(H_V-(\lambda+i\lambda^{-\frac{1}{3}})^2)u\|_{L^2 (\mathbb{T}^2)},
\end{align*}
which yields \eqref{T lambda (Vu) est for 2D tori geodesic}, and this completes the proof.

\subsection{Hypersurfaces and codimension 2 submanifolds in Tori}
In this subsection, we prove \eqref{Estimate n=3 k q for tori 4/3 potential} and \eqref{Estimate n=4 k q for tori L2 potential} to finish proving Theorem \ref{Thm for tori}. We let $P=\sqrt{-\Delta_{\mathbb{T}^n}}$, and let $\epsilon_1 (\lambda)$, $\epsilon_2 (\lambda)$, and $\epsilon_3 (\lambda)$ be as in \eqref{Window set up for nD tori for submfld}. By the same argument as in Lemma \ref{Lemma 2D torus projection for general curve}, we have a similar spectral projection bound for $\mathbb{T}^3$ or $\mathbb{T}^4$.
\begin{lemma}\label{Lemma nD torus projection for submfld}
    If $n\in \{3, 4\}$ and $k\in \{n-2, n-1\}$, then
    \begin{align*}
        \|\mathds{1}_{[\lambda, \lambda+\frac{1}{T}]} (P) f\|_{L^q (\Sigma)}\lesssim (T^{-\frac{1}{2}}\lambda^{\delta(q, k)}+(T\lambda)^{\frac{n-1}{4}})\|f\|_{L^2 (\mathbb{T}^2)},\quad 1\leq T\leq \lambda,\quad q\geq 2.
    \end{align*}
\end{lemma}
Taking $T$ so that $\frac{\lambda^{\delta(q, k)}}{T^{\frac{1}{2}}}=(T\lambda)^{\frac{n-1}{4}}$, it follows from Lemma \ref{Lemma nD torus projection for submfld} that
\begin{align}\label{Result of spec proj bound for tori submflds}
    \|\mathds{1}_{[\lambda, \lambda+\frac{1}{T}]} (P) f\|_{L^q (\Sigma)}\lesssim \lambda^{\delta(q, k)} T^{-\frac{1}{2}},\quad T^{-1}=\lambda^{-\frac{2}{n+1}\left(\frac{n-1}{2}-\frac{2k}{q} \right)}.
\end{align}
As before, by this and Lemma \ref{Lemma Resolvent opr norm from L2 to Lq for univ est}, since $\epsilon_1(\lambda)=\lambda^{-\frac{2}{n+1}\left(\frac{n-1}{2}-\frac{2k}{q}\right)}$, it follows that
\begin{align*}
    \|(-\Delta_{\mathbb{T}^2}-(\lambda+i\epsilon_1 (\lambda))^2)^{-1}\|_{L^2 (\mathbb{T}^n)\to L^q (\Sigma)}\lesssim \lambda^{\delta(q, k)-1} (\epsilon_1 (\lambda))^{-\frac{1}{2}}.
\end{align*}
We set $T_\lambda, T_\lambda^0, T_\lambda^1, T_\lambda^{1, 0}, T_\lambda^{1, j}$, and $R_\lambda$ as in \eqref{T lambda R lambda set up} and \eqref{T lambda (1, 0), (1, j) set up}, i.e.,
\begin{align*}
    \begin{split}
        & T=c_0 (\epsilon_1 (\lambda))^{-1},\quad 0<c_0\ll 1, \\
        & T_\lambda=T_\lambda^0+T_\lambda^1, \\
        & T_\lambda^0=\frac{i}{\lambda+i\epsilon_1(\lambda)} \int_0^\infty \eta(t) \eta(t/T) e^{i\lambda_1 t} e^{-\epsilon_1(\lambda)t} \cos (tP)\:dt, \\
        & T_\lambda^1=\frac{i}{\lambda+i\epsilon_1(\lambda)} \int_0^\infty (1-\eta(t))\eta(t/T) e^{i\lambda t} e^{-\epsilon_1(\lambda)t} \cos (tP)\:dt, \\
        & R_\lambda=\frac{i}{\lambda+i\epsilon_1(\lambda)}\int_0^\infty (1-\eta(t/T)) e^{i\lambda t} e^{-\epsilon_1(\lambda)t} \cos (tP)\:dt, \\
        & T_\lambda^{1, 0}=\frac{i}{\lambda+i\epsilon_1(\lambda)}\int_0^\infty \beta_0 (t) (1-\eta(t)) \eta(t/T) e^{i\lambda t} e^{-\epsilon_1(\lambda)t} \cos tP\:dt, \\
        & T_\lambda^{1, j}=\frac{i}{\lambda+i\epsilon_1(\lambda)}\int_0^\infty \beta(2^{-j}t) (1-\eta(t))\eta(t/T) e^{i\lambda t} e^{-\epsilon_1(\lambda) t} \cos tP\:dt,\quad 1\leq 2^j \lesssim (\epsilon_1(\lambda))^{-1}.
    \end{split}    
\end{align*}
By the argument in \S \ref{SS: general curve in 2D torus}, we have that
\begin{align}\label{Tori submfld T lambda R lambda estimates}
    \begin{split}
        & \|R_\lambda \|_{L^2 (\mathbb{T}^2)\to L^q (\Sigma)}\lesssim \lambda^{\delta(q, k)-1}(\epsilon_1 (\lambda))^{\frac{1}{2}}, \\
        & \|R_\lambda \circ (-\Delta_{\mathbb{T}^2}-(\lambda+i\epsilon_1 (\lambda))^2)\|_{L^2 (\mathbb{T}^2)\to L^q (\Sigma)}\lesssim \lambda^{\delta(q, k)}(\epsilon_1 (\lambda))^{-\frac{1}{2}}, \\
        & \|T_\lambda \|_{L^2 (\mathbb{T}^2)\to L^q (\Sigma)}\lesssim \lambda^{\delta(q, k)-1} (\epsilon_1 (\lambda))^{-\frac{1}{2}}.
    \end{split}
\end{align}
As before, we also note that
\begin{align}\label{I T lambda R lambda perturbation term}
    \begin{split}
        u&=(-\Delta_{\mathbb{T}^n}-(\lambda+i\epsilon_1 (\lambda))^2)^{-1}\circ (-\Delta_{\mathbb{T}^n}-(\lambda+i\epsilon_1 (\lambda))^2)u \\
        &=T_\lambda (H_V-(\lambda+i\epsilon_1 (\lambda))^2)u+R_\lambda (-\Delta_{\mathbb{T}^n}-(\lambda+i\epsilon_1 (\lambda))^2)u-T_\lambda (Vu).
    \end{split}
\end{align}
As above, using \eqref{Tori submfld T lambda R lambda estimates} and \eqref{I T lambda R lambda perturbation term}, it is enough to bound $T_\lambda (Vu)$. Since $T_\lambda=T_\lambda^0+T_\lambda^1$, as above, we compute $T_\lambda^0 (Vu)$ and $T_\lambda^1 (Vu)$, separately. The rest of the proof will be similar to the argument in \S \ref{SS: general hypsurf and codim 2 cases}. We note that $T_\lambda^0$ is a ``local'' operator as in the local operator $S_\lambda$ in \eqref{S lambda W lambda set up}. Since $n, k$, and $q$ satisfying \eqref{n=3 k q for tori 4/3 potential} and \eqref{n=4 k q for tori L2 potential} also satisfy the hypothesis of Proposition \ref{Prop Reduction for hypersurface thm} and Proposition \ref{Prop Reduction for codim 2 thm}, the proof of the two propositions gives us that
\begin{align}\label{T lambda 0 (Vu) uniform est on the Sobolev trace line}
    \|T_\lambda^0 (Vu)\|_{L^q (\Sigma)}\lesssim \|Vu\|_{L^p (M)},\quad \frac{n}{p}-\frac{k}{q}=2.
\end{align}
We recall from \cite[Theorem 1.4]{BlairHuangSireSogge2022UniformSobolev} that
\begin{align}\label{Epsilon 3 window est from BHSS on tori}
    \begin{split}
        \|u\|_{L^q (\mathbb{T}^n)}\leq C_V (\epsilon_3 (\lambda))^{-\frac{1}{2}}\lambda^{\sigma(q)-1}\|(H_V-(\lambda+i\epsilon_3 (\lambda))^2)u\|_{L^2 (\mathbb{T}^2)}, \\
        \text{where } n=3, 4, \frac{2n}{n-2}\leq q<\infty, \text{ and } \epsilon_3 (\lambda)^{-\frac{1}{3}+\delta_0},
    \end{split}
\end{align}
where $\delta_0>0$ is any fixed positive real number, and
\begin{align}\label{Epsilon 2 window est from BHSS on tori}
    \begin{split}
        \|u\|_{L^q (\mathbb{T}^n)}\leq C_V \lambda^{\epsilon_0} (\epsilon_2 (\lambda))^{-\frac{n+3}{2(n+1)}} \lambda^{-\frac{n+3}{2(n+1)}} \|(H_V-(\lambda+i\epsilon_2 (\lambda))^2)u\|_{L^2 (\mathbb{T}^2)}, \\
        \text{where } n=3, 4, q=\frac{2(n+1)}{n-1}, \text{ and } \epsilon_2 (\lambda)=\begin{cases}
            \lambda^{-\frac{3}{16}+c_0}, & \text{if } n=3, \\
            \lambda^{-\frac{1}{6}}, & \text{if } n=4,
        \end{cases}
    \end{split}
\end{align}
where $\epsilon_0, c_0>0$ are arbitrarily fixed. To make use of \eqref{Epsilon 3 window est from BHSS on tori}, we note that if $\frac{n}{p}-\frac{k}{q}=2$, then $\frac{np}{n-2p}=\frac{nq}{k}$. This gives us that
\begin{align*}
    \frac{np}{n-2p}=\frac{nq}{k}\geq \frac{2n}{n-2},\quad \text{when \eqref{n=3 k q for tori 4/3 potential} and \eqref{n=4 k q for tori L2 potential} hold.} 
\end{align*}
By this, \eqref{T lambda 0 (Vu) uniform est on the Sobolev trace line}, and the argument in \eqref{S(vu) est computation hypersurface}, it follows from \eqref{Epsilon 3 window est from BHSS on tori} and H\"older's inequality that
\begin{align}\label{T lambda 0 (Vu) bound for tori submflds}
    \begin{split}
         \|T_\lambda^0 (Vu)\|_{L^q (\Sigma)}&\lesssim \|V\|_{L^{\frac{n}{2}} (\mathbb{T}^n)}\|u\|_{L^{\frac{np}{n-2p}}(\mathbb{T}^n)} \\
        &\leq C_V \lambda^{\sigma\left(\frac{np}{n-2p} \right)-1} (\epsilon_3 (\lambda))^{-\frac{1}{2}} \|(H_V-(\lambda+i\epsilon_3 (\lambda))^2)u\|_{L^2 (\mathbb{T}^n)} \\
        &=C_V \lambda^{\frac{n-1}{2}-\frac{k}{q}-1} (\epsilon_3 (\lambda))^{-\frac{1}{2}} \|(H_V-(\lambda+i\epsilon_3 (\lambda))^2)u\|_{L^2 (\mathbb{T}^n)} \\
        &=C_V \lambda^{\delta(q, k)-1} (\epsilon_3 (\lambda))^{-\frac{1}{2}} \|(H_V-(\lambda+i\epsilon_3 (\lambda))^2)u\|_{L^2 (\mathbb{T}^n)}.
    \end{split}
\end{align}
We note that, if we choose $\delta_0>0$ appropriately in \eqref{Epsilon 3 window est from BHSS on tori}, then we have
\begin{align*}
    \epsilon_3 (\lambda)=\lambda^{-\frac{1}{3}+\delta_0}=\epsilon_1 (\lambda)<1,\quad \text{when \eqref{n=3 k q for tori 4/3 potential} and \eqref{n=4 k q for tori L2 potential} hold,}
\end{align*}
and thus, \eqref{T lambda 0 (Vu) bound for tori submflds} gives us the bounds in \eqref{Estimate n=3 k q for tori 4/3 potential} and \eqref{Estimate n=4 k q for tori L2 potential}. For the operator $T_\lambda^1$, note that $T_\lambda^{1, 0}$ is also a ``local'' operator, and so, it satisfies the same bound as \eqref{T lambda 0 (Vu) bound for tori submflds} with $\epsilon_3 (\lambda)=\epsilon_1 (\lambda)$, i.e.,
\begin{align}\label{T lambda 1 0 (Vu) bound for tori submflds}
    \|T_\lambda^{1, 0} (Vu)\|_{L^q (\Sigma)}\leq C_V \lambda^{\delta(q, k)-1} (\epsilon_1 (\lambda))^{-\frac{1}{2}} \|(H_V-(\lambda+i\epsilon_1 (\lambda))^2)u\|_{L^2 (\mathbb{T}^n)}.
\end{align}

With this in mind, we first show that \eqref{Estimate n=3 k q for tori 4/3 potential}. Suppose \eqref{n=3 k q for tori 4/3 potential} holds. By \eqref{Tori submfld T lambda R lambda estimates}-\eqref{T lambda 1 0 (Vu) bound for tori submflds}, we would have \eqref{Estimate n=3 k q for tori 4/3 potential}, if we could show that
\begin{align}\label{WTS tori estimate 4/3 potential}
    \sum_{j=1}^\infty \|T_\lambda^{1, j} (Vu)\|_{L^q (\Sigma)} \leq C_V \lambda^{\delta(q, k)-1} (\epsilon_2 (\lambda))^{-\frac{3}{4}} \|(H_V-(\lambda+i\epsilon_2 (\lambda))^2)u\|_{L^2 (\mathbb{T}^3)}.
\end{align}
To see this, recall from \cite[(4.25)]{BlairHuangSireSogge2022UniformSobolev} that, for $n\geq 3$,
\begin{align*}
    \sum_{j=1}^\infty \|T_\lambda^{1, j} f\|_{L^\infty (\Sigma)} &\leq \sum_{j=1}^\infty \|T_\lambda^{1, j} f\|_{L^\infty (\mathbb{T}^n)} \\
    &\lesssim \sum_{2\leq 2^j\lesssim (\epsilon_1 (\lambda))^{-1} } \lambda^{\frac{n-3}{2}} 2^{\frac{n+1}{2}j} \|f\|_{L^1 (\mathbb{T}^n)}\lesssim \lambda^{\frac{n-3}{2}} (\epsilon_1 (\lambda))^{-\frac{n+1}{2}} \|f\|_{L^1 (\mathbb{T}^n)}.
\end{align*}
Using this, we have, for $n\geq 3$,
\begin{align}\label{T lambda 1 j compuation for tori submflds}
    \begin{split}
        \sum_{2\leq 2^j\lesssim (\epsilon_1 (\lambda))^{-1}}\|T_\lambda^{1, j} (Vu)\|_{L^q (\Sigma)} &=\sum_{2\leq 2^j\lesssim (\epsilon_1 (\lambda))^{-1}}\left(\int_\Sigma |(T_\lambda^{1, j} (Vu))(z)|^q\:dz \right)^{\frac{1}{q}} \\
        &\lesssim \lambda^{\frac{n-3}{2}} (\epsilon_1 (\lambda))^{-\frac{n+1}{2}} \left(\int_\Sigma \|Vu\|_{L^1 (\mathbb{T}^n)}^q \right)^{\frac{1}{q}} \\
        &\lesssim \lambda^{\frac{n-3}{2}} (\epsilon_1 (\lambda))^{-\frac{n+1}{2}} \|Vu\|_{L^1 (\mathbb{T}^n)}.
    \end{split}
\end{align}
We note that $V\in L^{\frac{4}{3}} (\mathbb{T}^3)$ since $V\in L^{\frac{3}{2}} (\mathbb{T}^3)$ and $M$ is compact. By \eqref{n=3 k q for tori 4/3 potential}, \eqref{T lambda 1 j compuation for tori submflds}, H\"older's inequality, and \eqref{Epsilon 2 window est from BHSS on tori}, we have, for $0<\epsilon_0 \ll 1$,
\begin{align*}
    \sum_{2\leq 2^j\lesssim (\epsilon_1 (\lambda))^{-1}}\|T_\lambda^{1, j} (Vu)\|_{L^q (\Sigma)}&\lesssim (\epsilon_1 (\lambda))^{-2}\|V\|_{L^{\frac{4}{3}} (\mathbb{T}^3)} \|u\|_{L^4 (\mathbb{T}^3)} \\
    &\leq C_V \lambda^{\frac{1}{4}-\frac{2k}{q}+\epsilon_0}(\epsilon_2 (\lambda))^{-\frac{3}{4}}\|(H_V-(\lambda+i \epsilon_2 (\lambda))^2)u\|_{L^2 (\mathbb{T}^3)} \\
    &\leq C_V \lambda^{\delta(q, k)-1} (\epsilon_2 (\lambda))^{-\frac{3}{4}} \|(H_V-(\lambda+i\epsilon_2 (\lambda))^2)u\|_{L^2 (\mathbb{T}^3)},
\end{align*}
which satisfies \eqref{WTS tori estimate 4/3 potential}, completing the proof of \eqref{Estimate n=3 k q for tori 4/3 potential}.

We next show \eqref{Estimate n=4 k q for tori L2 potential} when \eqref{n=4 k q for tori L2 potential} holds. By \eqref{Tori submfld T lambda R lambda estimates}-\eqref{T lambda 1 0 (Vu) bound for tori submflds}, we would have \eqref{Estimate n=4 k q for tori L2 potential}, if we could show that
\begin{align}\label{WTS tori estimate L2 potential}
    \sum_{j=1}^\infty \|T_\lambda^{1, j} (Vu)\|_{L^q (\Sigma)} \leq C_V \lambda^{2-\frac{2k}{q}} \|u\|_{L^2 (\mathbb{T}^4)}. 
\end{align}
Similarly, by \eqref{n=4 k q for tori L2 potential}, \eqref{T lambda 1 j compuation for tori submflds}, and H\"older's inequality, we have
\begin{align*}
    \sum_{2\leq 2^j\lesssim (\epsilon_1 (\lambda))^{-1}}\|T_\lambda^{1, j} (Vu)\|_{L^q (\Sigma)}\lesssim \lambda^{\frac{1}{2}} (\epsilon_1 (\lambda))^{-2}\|V\|_{L^2 (\mathbb{T}^4)} \|u\|_{L^2 (\mathbb{T}^4)} \leq C_V \lambda^{2-\frac{2k}{q}} \|u\|_{L^2 (\mathbb{T}^4)},
\end{align*}
which proves \eqref{WTS tori estimate L2 potential}, completing the proof of \eqref{Estimate n=4 k q for tori L2 potential}. This completes the proof of Theorem \ref{Thm for tori}.

\section{Some partial results and related future work}\label{S: Partial result and future work}
\subsection{Higher codimension analogues}
In this paper, hypersurfaces and codimension $2$ cases are not fully resolved. Indeed, we do not have estimates in Theorem \ref{Thm for hypersurfaces} for $2\leq q<\frac{2(n-1)^2}{n^2-3n+4}$ when $n\in \{4, 5\}$, $2\leq q<\frac{2n^2-5n+4}{n^2-4n+8}$ when $n\in \{6, 7\}$, or $2\leq q<\frac{2n^2-5n+4}{n^2-4n+8}$ when $n\geq 8$, and in Theorem \ref{Thm for codim 2} for $2\leq q\leq \frac{2(n-2)^2}{n^2-5n+8}$ when $n\geq 5$. Finding higher codimension analogues of Theorem \ref{Thm for hypersurfaces}-\ref{Thm for codim 2} may also be interesting. For example, if we use the arguments in the proof of Theorem \ref{Thm for codim 2}, one may obtain that if $n\geq 4$ and $k=n-3$, then
\begin{align*}
    \|u\|_{L^{\frac{2n-3}{n-2}} (\Sigma)}\leq C_V \lambda^{\delta\left(\frac{2n-3}{n-2}, n-3\right)-1}\|(H_V-(\lambda+i)^2)u\|_{L^2 (M)},\quad \text{when } V\in L^{\frac{2n(2n-3)}{7n-9} } (M),
\end{align*}
and if $n\geq 5$ and $k=n-4$, then
\begin{align*}
    \|u\|_{L^2 (\Sigma)}\leq C_V \lambda^{\delta(2, n-4)-1}\|(H_V-(\lambda+i)^2)u\|_{L^2 (M)}, \quad \text{when } V\in L^{\frac{2n}{3} } (M),
\end{align*}
but the potential either $V\in L^{\frac{2n(2n-3)}{7n-9}} (M)$ or $V\in L^{\frac{2n}{3}} (M)$ is not critically singular anymore. Getting the estimates with $V\in L^{\frac{n}{2}} (M)$ for higher codimension cases may be difficult if we follow the arguments in this paper, since the arguments in Theorem \ref{Thm for hypersurfaces}-\ref{Thm for codim 2} would imply conditions $q<\frac{2k}{n-3}$ (or $q\leq \frac{2k}{n-3}$) when $k=\dim \Sigma$, and if $k<n-3$, this gives $q<2$. This is not usual, since we usually consider $q\geq 2$ for spectral projection estimates.

\subsection{Analogues of Theorem \ref{Thm for hypersurfaces}- \ref{Thm for log improved} with Kato potentials}
In Theorem \ref{Thm for hypersurfaces}-\ref{Thm for codim 2}, we considered $V\in L^{\frac{n}{2}} (M)$ to show estimates in the theorems. If $\dim M=n=3$, we can say that there is a special Kato class potential so that we have similar estimates in Theorem \ref{Thm for hypersurfaces}, \ref{Thm for codim 2}, and \ref{Thm for log improved}. To see this, we recall the following definition.

\begin{definition}[Schechter \cite{Schechter1971Spectra}, Simon \cite{Simon1982Semigroup}]\label{Definition of M beta p}
We say $V\in M_{\beta, p}$ if
\begin{align*}
    \sup_x \int_{d_g (x, y)\leq \frac{1}{2}\mathrm{Inj}(M), y\in M} (d_g (x, y))^{\beta-n} |V(y)|^p\:dV_g (y)=:\|V\|_{\beta, p}^p<\infty,
\end{align*}
where $\mathrm{Inj} (M)$ is the injectivity radius of a compact Riemannian manifold $M$ without boundary, and $dV_g$ is the Riemannian volume form.
\end{definition}

We focus on $\beta>0$ when we consider $M_{\beta, p}$ here. Some $M_{\beta, p}$ classes are sub-classes of $\mathcal{K}(M)$. For example, if $n=3$, then $M_{1, 2}\subset \mathcal{K}(M)$. Indeed, note that, if $V\in M_{1, 2}$, then, for $0<\epsilon<1$, by the Cauchy-Schwarz inequality, in local coordinates
\begin{align*}
    \int_{|x-y|\leq \epsilon} 1\cdot |x-y|^{-1} |V(y)|\:dy \leq \left(\int_{|x-y|\leq \epsilon} 1^2\:dy \right)^{\frac{1}{2}} \left(\int_{|x-y|\leq \epsilon} |x-y|^{-2} |V(y)|^2\:dy \right)^{\frac{1}{2}} \leq \epsilon^{\frac{3}{2}} \|V\|_{1, 2},
\end{align*}
and thus, after taking supremum over $x$, when $\epsilon \to 0$, we have $V\in \mathcal{K}(M)$. It then follows that $H_V$ is self-adjoint and positive by \cite[\S 2]{BlairSireSogge2021Quasimode}, and hence, \eqref{Hv eigfcn decomp} makes sense.

If we allow a log loss, we have the following higher dimensional analogues of Theorem \ref{Thm for univ est for any curves}, which are the partial analogues of \cite[Theorem 3]{BurqGerardTzvetkov2007restrictions} and \cite[Theorem 1.3]{Hu2009lp}. We also want to remove a log loss when $(n, k)=(3, 1)$, where the submanifold is either a geodesic segment, or curve with nonvanishing curvatures, which were proved in \cite{ChenSogge2014few} and \cite{WangZhang2021Codim2} for $V\equiv 0$.

\begin{theorem}\label{Thm for n=3 general}
Suppose $M$ is a compact Riemannian manifold of dimension $n=3$ and $\Sigma$ is a submanifold of $M$ with dimension $k$. Let $\delta (q, k)$ and $\nu (q, k)$ be as in \eqref{BGT and Hu lambda exponent} and \eqref{BGT and Hu log lambda exponent}, respectively. Also assume $u\in \mathrm{Dom}(H_V)$ and $\lambda\geq 1$.
\begin{enumerate}
    \item Let $n=3$, $1\leq k\leq 2$, and $V\in M_{1, 2}\subset \mathcal{K}(M)$. Then we have
    \begin{align}\label{Univ Est for n=3 submfld}
        \|u\|_{L^q (\Sigma)}\leq C_V\lambda^{\delta(q, k)-1} (\log \lambda)^{\nu (q, k)} \|(H_V-(\lambda+i)^2)u\|_{L^2 (M)},\quad \text{for all } q\geq 2.
    \end{align}
    \item Let $n=3$ and $V\in M_{1, 2}\subset \mathcal{K}(M)$. Suppose the curve $\gamma$ is either a geodesic segment or a curve segment with nonvanishing curvatures. We then have that
    \begin{align}\label{Univ Est for codim 2 with sing pot}
        \|u\|_{L^2 (\gamma)}\leq C_V \lambda^{-\frac{1}{2}}\|(H_V-(\lambda+i)^2)u \|_{L^2(M)}.
    \end{align}
\end{enumerate}
\end{theorem}

Since $M_{1, 2}$ is a sub-class of $\mathcal{K}(M)$, it would be interesting if we could extend the sub-class $M_{1, 2}$ to whole $\mathcal{K}(M)$ when $n=3$. It would also be interesting to consider $n\geq 4$.

\begin{proof}[Proof of Theorem \ref{Thm for n=3 general}]
    We first show \eqref{Univ Est for n=3 submfld}. Let us first consider the case where $n=3$, $\Sigma$ is any submanifold, and $V\in M_{1, 2}$. As before, we write
    \begin{align*}
        (-\Delta_g-(\lambda+i)^2)^{-1}=S_\lambda +W_\lambda,
    \end{align*}
    where $S_\lambda$ and $W_\lambda$ are as in \S \ref{S: Preliminaries}, and $P=\sqrt{-\Delta_g}$. By \cite[Theorem 3]{BurqGerardTzvetkov2007restrictions} and \cite[Theorem 1.3]{Hu2009lp}, we have
    \begin{align*}
        & \|\mathds{1}_{[\lambda, \lambda+1]} (P) \|_{L^2 (M)\to L^p (\Sigma)}\lesssim \lambda^{\delta(p, k)} (\log \lambda)^{\nu (p, k)}, \\
        & \|u\|_{L^p (\Sigma)}\lesssim \lambda^{\delta(p)-1} (\log \lambda)^{\nu (p, k)} \|(-\Delta_g-(\lambda+i)^2)u\|_{L^2 (M)},
    \end{align*}
    and
    \begin{align*}
        \| (-\Delta_g -(\lambda+i)^2)^{-1} \|_{L^2 (M)\to L^p (\Sigma)}\lesssim \lambda^{\delta(p)-1} (\log \lambda)^{\nu (p, k)}.
    \end{align*}
    Using these and the arguments above in \S \ref{S: Proof for curves for universal estimates}, we obtain the following estimates for $S_\lambda$ and $W_\lambda$.
    \begin{align}\label{S W lambda est for n=3 general}
        \begin{split}
            & \|S_\lambda \|_{L^2 (M)\to L^p (\Sigma)}\lesssim \lambda^{\delta(p, k)-1} (\log \lambda)^{\nu (p, k)}, \\
            & \|W_\lambda (-\Delta_g -(\lambda+i)^2)\|_{L^2 (M)\to L^p (\Sigma)}\lesssim \lambda^{\delta(p, k)-1} (\log \lambda)^{\nu (p, k)}.
        \end{split}
    \end{align}
    We then want to bound $S_\lambda (Vu)$ as above. When $n=3$, by using an argument in \cite[\S 5]{BlairHuangSireSogge2022UniformSobolev}, we have the kernel estimates of $S_\lambda$ as follows.
    \begin{align*}
        \begin{split}
            |S_\lambda (x, y)|\lesssim \begin{cases}
            (d_g (x, y))^{2-n}=(d_g (x, y))^{-1}, & \text{if } d_g (x, y)\leq \lambda^{-1},\\
            \lambda^{\frac{n-3}{2}} (d_g (x, y))^{-\frac{n-1}{2}}=(d_g (x, y))^{-1}, & \text{if } \lambda^{-1}\leq d_g (x, y)\leq 1.
            \end{cases}
        \end{split}
    \end{align*}
    We thus have that
    \begin{align}\label{S lambda kernel bound for n=3}
        |S_\lambda (x, y)|\lesssim (d_g (x, y))^{-1},\quad \text{for }\quad d_g (x, y)\leq 1.
    \end{align}
    We want to find the following estimates of $S_\lambda (Vu)$, using Definition \ref{Definition of M beta p}.

    \begin{lemma}\label{Lemma perturbation est for n=3 or more}
    Let $n=3$. If $V\in M_{1, 2}$, then
    \begin{align*}
        \|S_\lambda (Vu)\|_{L^p (\Sigma)} \leq C_V \|u\|_{L^2 (M)},
    \end{align*}
    for all $1\leq k\leq 2$ and $p\geq 2$.
    \end{lemma}

    Since
    \begin{align*}
        u&=(-\Delta_g -(\lambda+i)^2)^{-1} \circ (-\Delta_g -(\lambda+i)^2)u =(S_\lambda+W_\lambda)\circ (-\Delta_g -(\lambda+i)^2) u \\
        &=S_\lambda \circ (H_V-(\lambda+i)^2)u +W_\lambda \circ (-\Delta_g -(\lambda+i)^2)u-S_\lambda (Vu),
    \end{align*}
    if the lemma is true, then, by \eqref{S W lambda est for n=3 general},
    \begin{align*}
        \|u\|_{L^p (\Sigma)} &\leq \|S_\lambda \circ (H_V -(\lambda+i)^2) u\|_{L^p (\Sigma)}+\|W_\lambda \circ (-\Delta_g-(\lambda+i)^2)u\|_{L^p (\Sigma)}+\|S_\lambda (Vu)\|_{L^p (\Sigma)}\\
        &\leq C_V \lambda^{\delta(p, k)-1} (\log \lambda)^{\nu (p, k)} \|(H_V-(\lambda+i)^2)u\|_{L^2 (M)}+ C_V \lambda^{-1} \cdot \lambda\|u\|_{L^2 (M)} \\
        &\leq C_V \lambda^{\delta(p, k)-1} (\log \lambda)^{\nu (p, k)} \|(H_V-(\lambda+i)^2)u\|_{L^2 (M)}+ C_V\lambda^{-1} \|(H_V-(\lambda+i)^2)u\|_{L^2 (M)} \\
        &\lesssim C_V \lambda^{\delta(p, k)-1} (\log \lambda)^{\nu (p, k)} \|(H_V-(\lambda+i)^2)u\|_{L^2 (M)},
    \end{align*}
    which completes the proof of \eqref{Univ Est for n=3 submfld}. As above, we used the spectral theorem in the second to last inequality. We now want to show Lemma \ref{Lemma perturbation est for n=3 or more}.

    \begin{proof}[Proof of Lemma \ref{Lemma perturbation est for n=3 or more}]
    Let $r\mapsto \sigma(r)$ be a coordinate map of the submanifold $\Sigma$. By the triangle inequality, we have
    \begin{align}\label{S lambda triangle ineq}
        \begin{split}
            \|S_\lambda (Vu)\|_{L^p (\Sigma)} &=\left(\int \left|\int S_\lambda (\sigma(r), y) V(y) u(y)\:dy \right|^p\:dr \right)^{\frac{1}{p}} \\
            &\leq \left(\int \left(\int |S_\lambda (\sigma(r), y)| |V(y)| |u(y)|\:dy \right)^p\:dr \right)^{\frac{1}{p}}.
        \end{split}
    \end{align}
    To bound this, we want to bound the integral in the second parentheses, i.e.,
    \begin{align*}
        \int |S_\lambda (\sigma(r), y)| |V(y)| |u(y)|\:dy.
    \end{align*}
    Using a partition of unity if necessary, we may assume $|x-y|\leq 1$ for $x, y\in M$, especially when $x=\sigma(r)$, and so, $|\sigma(r)-y|\leq 1$. By \eqref{S lambda kernel bound for n=3} and the Cauchy-Schwarz inequality, we have, for $n\geq 3$,
    \begin{align*}
        \int |S_\lambda (\sigma(r), y)| |V(y)| |u(y)|\:dy &\leq \left(\int |S_\lambda (\sigma(r), y)|^2 |V(y)|^2 \:dy \right)^{\frac{1}{2}} \left(\int |u(y)|^2\:dy \right)^{\frac{1}{2}} \\
        &\lesssim \left(\int_{|\sigma(r)-y|\leq 1} |\sigma(r)-y|^{-2} |V(y)|^2\:dy \right)^{\frac{1}{2}}\|u\|_{L^2 (M)} \leq C_V \|u\|_{L^2 (M)}.
    \end{align*}
    In the last inequality, we used the assumption $V\in M_{1, 2}$. Using this with \eqref{S lambda triangle ineq}, we have
    \begin{align*}
        \| S_\lambda (Vu) \|_{L^p (\Sigma)}\leq C_V \|u\|_{L^2 (M)},
    \end{align*}
    as desired.
\end{proof}

We next show \eqref{Univ Est for codim 2 with sing pot}. Suppose $n=3$ and the curve $\gamma$ is either a geodesic segment or a curve with nonvanishing geodesic curvatures. As before, we write
\begin{align*}
    (-\Delta_g-(\lambda+i)^2)^{-1}=S_\lambda +W_\lambda,
\end{align*}
where $S_\lambda$ and $W_\lambda$ are as in as above, and $P=\sqrt{-\Delta_g}$. By \cite[Theorem 1]{ChenSogge2014few} and \cite[Theorem 3]{WangZhang2021Codim2}, we have
\begin{align*}
    & \|\mathds{1}_{[\lambda, \lambda+1]} (P) \|_{L^2 (M)\to L^2 (\Sigma)}\lesssim \lambda^{\frac{1}{2}}, \\
    & \|u\|_{L^2 (\gamma)}\lesssim \lambda^{-\frac{1}{2}} \|(H_V-(\lambda+i)^2) u\|_{L^2 (M)},
\end{align*}
and thus,
\begin{align*}
    \| (\Delta_g -(\lambda+i)^2)^{-1} \|_{L^2 (M)\to L^2 (\gamma)}\lesssim \lambda^{-\frac{1}{2}}.
\end{align*}
As before, by this and \eqref{m lambda size est for n=2}, an orthogonality argument gives us that $S_\lambda$ and $W_\lambda=m_\lambda (P)$ satisfy
\begin{align}\label{S W lambda est for codim 2}
    \begin{split}
        & \|S_\lambda\|_{L^2 (M)\to L^2 (\gamma)}\lesssim \lambda^{-\frac{1}{2}}, \\
        & \|W_\lambda \circ (-\Delta_g -(\lambda+i)^2)\|_{L^2 (M)\to L^2 (\Sigma)}\lesssim \lambda^{\frac{1}{2}}.
    \end{split}
\end{align}
The bounds for $S_\lambda(Vu)$ follow from Lemma \ref{Lemma perturbation est for n=3 or more}, since the lemma holds for any curve. Since
\begin{align*}
    u&=(-\Delta_g-(\lambda+i)^2)^{-1}\circ (-\Delta_g-(\lambda+i)^2)u \\
    &=S_\lambda \circ (-\Delta_g+V-(\lambda+i)^2)u+W_\lambda\circ (-\Delta_g-(\lambda+i)^2)u-S_\lambda (Vu),
\end{align*}
by \eqref{S W lambda est for codim 2}, and Lemma \ref{Lemma perturbation est for n=3 or more}, we have that
\begin{align*}
    \|u\|_{L^2 (\Sigma)}&\leq \|S_\lambda\circ (H_V-(\lambda+i)^2)u\|_{L^2 (\Sigma)}+\|W_\lambda \circ (-\Delta_g-(\lambda+i)^2)u\|_{L^2 (\Sigma)}+\|S_\lambda (Vu)\|_{L^2 (\Sigma)} \\
    &\leq C_V \lambda^{-\frac{1}{2}}\|(H_V-(\lambda+i)^2)u\|_{L^2 (M)},
\end{align*}
which proves \eqref{Univ Est for codim 2 with sing pot}.
\end{proof}

We can also see that we have a log improved restriction estimate for a $3$-dimensional manifolds with nonpostive sectional curvatures as follow.
\begin{theorem}\label{Thm for future work log improved n>=3}
    Assume that $(M, g)$ is a $3$-dimensional compact Riemannian manifold of constant negative sectional curvature, $\epsilon(\lambda)=(\log (2+\lambda))^{-1}$, and that $V\in M_{1, 2}\subset \mathcal{K}(M)$. If $\gamma$ is a unit-length geodesic segment, then
    \begin{align}\label{Est for log imp for n=3}
        \|u\|_{L^2 (\gamma)} \leq C_{V, \gamma} \lambda^{-\frac{1}{2}}(\epsilon(\lambda))^{-\frac{1}{2}} \|(H_V-(\lambda+i\epsilon(\lambda))^2)u\|_{L^2 (M)}.
    \end{align}
\end{theorem}

\begin{proof}
    We can prove \eqref{Est for log imp for n=3} similarly. To see this, we first consider the estimates when $V\equiv 0$ from \cite{Blair2018logarithmic} and \cite{Zhang2017improved} for $P=\sqrt{-\Delta_g}$
    \begin{align*}
        \|\mathds{1}_{[\lambda, \lambda+1]} (P)\|_{L^2 (M)\to L^2 (\gamma)}\lesssim \lambda^{\frac{1}{2}} (\log \lambda)^{-\frac{1}{2}},
    \end{align*}
    where $\gamma$ is a geodesic segment. This in turn implies that
    \begin{align}\label{Est of inverse opr for log imp n=3}
        \|(-\Delta_g-(\lambda+i \epsilon(\lambda))^2)^{-1}\|_{L^2 (M)\to L^2 (\gamma)} \lesssim \lambda^{-\frac{1}{2}} (\log \lambda)^{\frac{1}{2}}.
    \end{align}
    Let $\eta\in C_0^\infty(\mathbb{R})$, $T_\lambda$, $T_\lambda^0$, $T_\lambda^1$, and $R_\lambda$ be as in \eqref{T lambda R lambda set up}, so that we can write
    \begin{align*}
        (-\Delta_g-(\lambda+i \epsilon (\lambda))^2)^{-1}=T_\lambda+R_\lambda.
    \end{align*}
    Setting $m_\lambda$ as in \eqref{m lambda setup}, by the arguments above, the operator $R_\lambda=m_\lambda (P)$ satisfies
    \begin{align}\label{R lambda est for log imp n=3}
        \|R_\lambda \|_{L^2 (M)\to L^2 (\gamma)}\lesssim \lambda^{-\frac{1}{2}}(\epsilon(\lambda))^{-\frac{1}{2}},
    \end{align}
    and
    \begin{align}\label{R lambda comp est for log imp n=3}
        \|R_\lambda \circ (-\Delta_g-(\lambda+i \epsilon(\lambda))^2)\|_{L^2 (M)\to L^2 (\gamma)}\lesssim \lambda^{-\frac{1}{2}} (\epsilon (\lambda))^{-\frac{1}{2}}\cdot (\lambda \epsilon(\lambda))=\lambda^{\frac{1}{2}}(\epsilon(\lambda))^{\frac{1}{2}}.
    \end{align}
    Since $T_\lambda=(-\Delta_g-(\lambda+i\epsilon(\lambda))^2)^{-1}-R_\lambda$, by \eqref{Est of inverse opr for log imp n=3} and \eqref{R lambda est for log imp n=3},
    \begin{align}\label{T lambda est for log imp n=3}
        \|T_\lambda\|_{L^2 (M)\to L^2 (\gamma)}\lesssim \lambda^{-\frac{1}{2}} (\epsilon(\lambda))^{-\frac{1}{2}}.
    \end{align}
    For any given small $\epsilon_0>0$, if $c_0>0$ is small enough, then by \cite[(3.25)]{BlairHuangSireSogge2022UniformSobolev} (see also \cite{Berard1977onthewaveequation})
    \begin{align*}
        \|T_\lambda^1\|_{L^1 (M)\to L^\infty (M)}=O(\lambda^{Cc_0})=O(\lambda^{\epsilon_0}),
    \end{align*}
    and so,
    \begin{align}\label{T lambda 1 est for log imp n=3}
        \|T_\lambda^1 f\|_{L^\infty (M)}\leq \|T_\lambda^1 f\|_{L^\infty (M)}\lesssim \lambda^{\epsilon_0} \|f\|_{L^1 (M)}.
    \end{align}
    The proof of \eqref{T lambda 0 kernel est} in \cite[(5.11)]{BlairHuangSireSogge2022UniformSobolev} also applies to all $\dim M=n\geq 3$, and so, if $n=3$, then we have the following bounds for the kernel $T_\lambda^0$
    \begin{align*}
        \begin{split}
            |T_\lambda^0 (x, y)|\lesssim \begin{cases}
            (d_g (x, y))^{2-n}=(d_g (x, y))^{-1}, & \text{if } d_g (x, y)\leq \lambda^{-1}, \\
            \lambda^{\frac{n-3}{2}}(d_g (x, y))^{-\frac{n-1}{2}}=(d_g (x, y))^{-1}, & \text{if } \lambda^{-1}\leq d_g (x, y)\ll 1,
        \end{cases}
        \end{split}
    \end{align*}
    and thus
    \begin{align}\label{T lambda 0 kernel est for log imp n=3}
        |T_\lambda^0 (x, y)|\lesssim (d_g (x, y))^{-1},\quad \text{when } |x-y|\ll 1.
    \end{align}

    As above, we can write
    \begin{align*}
        u=T_\lambda (-\Delta_g+V-(\lambda+i\epsilon(\lambda))^2)u+R_\lambda (-\Delta_g -(\lambda+i\epsilon(\lambda))^2)u-T_\lambda (Vu).
    \end{align*}
    By \eqref{T lambda est for log imp n=3}, we have
    \begin{align*}
        \|T_\lambda \circ (H_V-(\lambda+i\epsilon(\lambda))^2)u\|_{L^2 (\gamma)}\lesssim \lambda^{-\frac{1}{2}} (\epsilon(\lambda))^{-\frac{1}{2}} \|(H_V-(\lambda+i\epsilon(\lambda))^2)u\|_{L^2 (M)}.
    \end{align*}
    By \eqref{R lambda comp est for log imp n=3}, we have
    \begin{align*}
        \|R_\lambda \circ (-\Delta_g -(\lambda+i\epsilon(\lambda))^2)u\|_{L^2 (\gamma)}&\lesssim \lambda^{\frac{1}{2}} (\epsilon(\lambda))^{\frac{1}{2}} \|u\|_{L^2 (M)} \\
        &\lesssim \lambda^{-\frac{1}{2}} (\epsilon (\lambda))^{-\frac{1}{2}}\|(H_V -(\lambda+i\epsilon(\lambda))^2)u\|_{L^2 (M)}.
    \end{align*}
    To show \eqref{Est for log imp for n=3}, it thus suffices to show that
    \begin{align}\label{Perturb claim for log imp n=3}
        \|T_\lambda (Vu)\|_{L^2 (\gamma)}\leq C_V \lambda^{-\frac{1}{2}} (\epsilon(\lambda))^{-\frac{1}{2}} \|(H_V-(\lambda+i\epsilon(\lambda))^2)u\|_{L^2 (M)},
    \end{align}
    where $0<\epsilon_0\ll 1$. Since $T_\lambda=T_\lambda^0+T_\lambda^1$, we consider $T_\lambda^0 (Vu)$ and $T_\lambda^1 (Vu)$ separately. For $T_\lambda^0 (Vu)$, we repeat the proof of Lemma \ref{Lemma perturbation est for n=3 or more}. By \eqref{T lambda 0 kernel est for log imp n=3}, $V\in M_{1, 2}$, and H\"older's inequality, we have that
    \begin{align*}
        \int |T_\lambda^0 (\gamma(r), y)| |V(y)| |u(y)|\:dy \lesssim \left(\int_{|\gamma(r)-y|\leq 1} (d_g (\gamma(r), y))^{-2} |V(y)|^2\:dy \right)^{\frac{1}{2}} \|u\|_{L^2 (M)}\leq C_V \|u\|_{L^2 (M)},
    \end{align*}
    which in turn implies that
    \begin{align*}
        \|T_\lambda^0 (Vu)\|_{L^2 (\gamma)} \leq \left(\int \left(\int |T_\lambda^0 (\gamma(r), y)| |V(y)| |u(y)|\:dy \right)^2\:dr \right)^{\frac{1}{2}} \leq C_V \|u\|_{L^2 (M)}.
    \end{align*}
    For $T_\lambda^1 (Vu)$, by \eqref{T lambda 1 est for log imp n=3} and the Cauchy-Schwarz inequality, we have that
    \begin{align*}
        \|T_\lambda^1 (Vu)\|_{L^2 (\gamma)}& \left(\int |T_\lambda^1 (Vu)|^2\:dr \right)^{\frac{1}{2}} \\
        &\lesssim \left(\int (\lambda^{\epsilon_0})^2 \|Vu\|_{L^1 (M)}^2\:dr \right)^{\frac{1}{2}} \\
        &\lesssim \lambda^{\epsilon_0} \|V\|_{L^2(M)}\|u\|_{L^2 (M)} \\
        &\leq \lambda^{\epsilon_0} \|V\|_{1, 2} \|u\|_{L^2 (M)}\leq C_V \lambda^{\epsilon_0} \|u\|_{L^2 (M)}.
    \end{align*}
    Putting these together, by the spectral theorem, we have
    \begin{align*}
        \|T_\lambda (Vu)\|_{L^2 (\gamma)} \leq C_V \lambda^{\epsilon_0} \|u\|_{L^2 (M)}\lesssim C_V \lambda^{-1+\epsilon_0} \|(H_V-(\lambda+i\epsilon(\lambda))^2)u\|_{L^2 (M)}.
    \end{align*}
    Taking $\epsilon_0>0$ sufficiently small, we have \eqref{Perturb claim for log imp n=3}, completing the proof of \eqref{Est for log imp for n=3}.
\end{proof}

\bibliographystyle{amsalpha}
\bibliography{references}

\end{document}